\documentclass[11pt,oneside,a4paper,reqno]{amsart}
\usepackage[utf8]{inputenc}
\usepackage[T1]{fontenc}
\usepackage{lmodern}
\usepackage{etoolbox}
\usepackage[protrusion=false]{microtype}

\usepackage{geometry}
\usepackage[all]{xy}
\pdfpagewidth=\paperwidth
\pdfpageheight=\paperheight
\usepackage{xcolor}

\usepackage{mathtools} 
\usepackage{amssymb} 
\usepackage{amsthm} 
\usepackage{thmtools} 

\usepackage{tikz}
\usepackage{manfnt}
\usetikzlibrary{matrix}
\usetikzlibrary{arrows}
\usetikzlibrary{positioning}
\usetikzlibrary{calc}
\usepackage{tikz-cd}

\usepackage{quiver}
\usepackage{stackengine,scalerel}

\usepackage{enumitem}
\setenumerate{listparindent=\parindent}
\setlist[enumerate]{
	font=\normalfont,
	label=(\roman*),
	topsep=3pt,
	itemsep=-0.3ex,
	partopsep=1ex,
	parsep=1ex}

\setenumerate[1]{label=(\alph*)}
\setenumerate[2]{label=(\roman*)}

\def\namedlabel#1#2{\begingroup
	#2%
	\def\@currentlabel{#2}%
	\phantomsection\label{#1}\endgroup
}
\makeatother

\makeatletter
\g@addto@macro\bfseries{\boldmath}
\makeatother

\usepackage{datetime}
\usepackage{url}

\DeclareMathOperator{\id}{id}

\DeclareMathOperator{\spaan}{span}

\DeclareMathOperator{\intt}{int}

\DeclareMathOperator{\Aut}{Aut}

\DeclareMathOperator{\ox}{\otimes}

\DeclareMathOperator{\Mor}{Mor}
\DeclareMathOperator{\Morp}{Mor_p}

\DeclareMathOperator{\pr}{Prp}
\DeclareMathOperator{\per}{Per}
\DeclareMathOperator{\pim}{Pim}
\DeclareMathOperator{\kat}{Kat}

\newcommand\restr[2]{{
		\left.\kern-\nulldelimiterspace 
		#1 
		\vphantom{|} 
		\right|_{#2} 
}}

\newcommand{\ol}[1]{\overline{#1}}

\newcommand{\wt}[1]{\widetilde{#1}}

\newcommand{\udot}{\mathbin{\ThisStyle{%
			\ensurestackMath{\stackinset{c}{0pt}{c}{0pt}
				{\scalebox{.6}{$\SavedStyle\bullet$}}{\SavedStyle\cup}}}}}

\newcommand{\unis}[3]{{#1} \uplus_{#2} {#3}}
\newcommand{\unif}[1]{{#1}_{\uplus}}
\newcommand{\pers}[4]{{#1} \uplus_{#2}^{#3} {#4} }
\newcommand{\perf}[2]{{#1}_{#2}}
\newcommand{\mins}[3]{{#1} \udot_{#2} {#3}}
\newcommand{\minf}[1]{{#1}_{\udot}}

\newcommand{\decor}[1]{\langle{#1}\rangle}

\DeclareMathOperator{\unilim}{\textstyle\varprojlim^{\uplus}}
\DeclareMathOperator{\perlim}{\textstyle\varprojlim^{\udot}}
\DeclareMathOperator{\minlim}{\textstyle\varprojlim^{\vee}}

\DeclareMathOperator{\unilima}{\textstyle\varinjlim^{\uplus}}
\DeclareMathOperator{\perlima}{\textstyle\varinjlim^{\udot}}
\DeclareMathOperator{\minlima}{\textstyle\varinjlim^{\vee}}
\DeclareMathOperator{\mlima}{\textstyle\mathcal{M}-\varinjlim}

\newcommand{\catname}[1]{{\normalfont\mathbf{#1}}}
\newcommand{\LCH}{\catname{LCH}}

\newcommand{\CC}{\mathbb{C}}

\newcommand{\NN}{\mathbb{N}}

\newcommand{\RR}{\mathbb{R}}
\newcommand{\TT}{\mathbb{T}}

\newcommand{\ZZ}{\mathbb{Z}}

\newcommand{\Bb}{\mathcal{B}}

\newcommand{\Hh}{\mathcal{H}}
\newcommand{\Ii}{\mathcal{I}}
\newcommand{\Jj}{\mathcal{J}}
\newcommand{\Kk}{\mathcal{K}}
\newcommand{\Ll}{\mathcal{L}}
\newcommand{\Mm}{\mathcal{M}}

\newcommand{\Oo}{\mathcal{O}}
\newcommand{\Pp}{\mathcal{P}}
\newcommand{\Qq}{\mathcal{Q}}

\newcommand{\Tt}{\mathcal{T}}
\newcommand{\Uu}{\mathcal{U}}
\newcommand{\Vv}{\mathcal{V}}

\newcommand{\Zz}{\mathcal{Z}}

\usepackage[
	pdfauthor = {Alexander Mundey},
	pdftitle = {Fibrewise compactifications and generalised limits in commutative and noncommutative topology},
	unicode=true,
	colorlinks=true,
	linkcolor=blue,
	citecolor=orange,
	urlcolor=purple,
	breaklinks=true
]{hyperref} 

\usepackage[
noabbrev,
capitalise,
nameinlink
]{cleveref} 

\newtheorem{thm}{Theorem}[section]
\newtheorem{cor}[thm]{Corollary}
\newtheorem{lem}[thm]{Lemma}
\newtheorem{prop}[thm]{Proposition}

\theoremstyle{definition}
\newtheorem{dfn}[thm]{Definition}

\theoremstyle{remark}
\newtheorem{rmk}[thm]{Remark}

\newtheorem{example}[thm]{Example}

\usepackage[symbol]{footmisc}

\title[Fibrewise compactifications and generalised limits]{Fibrewise compactifications and generalised limits in commutative and noncommutative topology}
\author{Alexander Mundey}

\date{\today}

\subjclass[2020]{46L85 (Primary) 55R70, 46M40 (Secondary)}
\keywords{Fibrewise compactification, perfection, unified space, generalised limit, regulated limit, boundary path space, noncommutative topology, $C^*$-algebras, multiplier algebras, Cuntz--Pimsner algebras}

\begin{document}

	\begin{abstract}
		We introduce fibrewise compactifications in both the setting of locally compact Hausdorff spaces and continuous maps, and the parallel setting of $C^*$-algebras and nondegenerate multiplier-valued $*$-homomorphisms.  In both situations, we use fibrewise compactifications to define regulated limits.
		In the topological setting, regulated limits extend classical inverse limits so that the resulting limit space remains locally compact; examples include the path spaces of directed graphs. In the operator-algebraic setting, regulated limits realise a direct-limit construction for multiplier-valued $*$-homomorphisms; examples include the cores of relative Cuntz–Pimsner algebras.
	\end{abstract}

\maketitle

\section{Introduction}
Locally compact Hausdorff spaces are a mainstay of modern mathematics and are often considered to be ``nice'' topological spaces, but the category of locally compact Hausdorff spaces and continuous maps has limitations. Notably, the category is not closed under taking inverse limits. The purpose of this article is to provide a solution to this problem, along with the analogous problem for $C^*$-algebras.  

Consider a sequence of locally compact Hausdorff spaces and continuous maps
\begin{equation}\label{eq:projective_sequence}
	\begin{tikzcd}[ampersand replacement=\&,cramped]
		{X_0} \& {X_1} \& {X_2} \& \cdots
		\arrow["{f_0}"', from=1-2, to=1-1]
		\arrow["{f_1}"', from=1-3, to=1-2]
		\arrow["{f_2}"', from=1-4, to=1-3]
	\end{tikzcd}.
\end{equation}
The inverse limit $\varprojlim(X_i,f_i)$ of \labelcref{eq:projective_sequence} is Hausdorff, but in general it is not locally compact. This can be remedied by insisting that the maps $f_i \colon X_{i+1} \to X_i$ are \emph{proper} \cite[Theorem~3.7.13]{Eng89}; that is, $f_i^{-1}(K)$ is compact for all compact $K \subseteq X_i$
. In practice, one may not always have the luxury of proper maps.

A related situation occurs in the setting of $C^*$-algebras. Given a sequence of $C^*$-algebras and $*$-homomorphisms
\begin{equation}\label{eq:inductive_sequence}
	\begin{tikzcd}[ampersand replacement=\&,cramped]
		{A_0} \& {A_1} \& {A_2} \& \cdots
		\arrow["{\varphi_0}", from=1-1, to=1-2]
		\arrow["{\varphi_1}", from=1-2, to=1-3]
		\arrow["{\varphi_2}", from=1-3, to=1-4]
	\end{tikzcd}
\end{equation}
we can form the direct limit $\varinjlim(A_i,\varphi_i)$. Gelfand duality implies that if the $A_i$ are commutative, and the maps $\varphi_i \colon A_i \to A_{i+1}$ take an approximate identity of $A_i$ to one of $A_{i+1}$, then the sequence \labelcref{eq:inductive_sequence} is dual to one of the form \labelcref{eq:projective_sequence} for which the $f_i$ are continuous and proper. This does not reflect the full situation of \labelcref{eq:projective_sequence}, since there the $f_i$ are not necessarily proper. The appropriate noncommutative notion of a continuous (non-proper) function is a $*$-homomorphism $\varphi \colon A \to \Mm(B)$ between a $C^*$-algebra $A$ and the multiplier algebra of a $C^*$-algebra $B$ such that $\varphi(A)B$ is norm dense in $B$ (see \cite{aHRW10}). We think of such a map as a \emph{morphism} from $A$ to $B$ in the sense of \cite{Lan95}. The problem is that if we replace the $*$-homomorphisms of \labelcref{eq:inductive_sequence} with morphisms $\varphi_i \colon A_i \to \Mm(A_{i+1})$, then it is unclear how to make sense of the direct limit.

In this article, we provide a remedy the problems occurring in both situations, dealing with the topological setting in \cref{sec:topological} and the $C^*$-algebraic setting in \cref{sec:CStar}. We do so by introducing \emph{fibrewise compactifications} in both situations. Fibrewise compactifications allow us to extend a continuous map to a proper continuous map, or a morphism between $C^*$-algebras to a proper morphism in the sense of \cite{ELP99}. We use fibrewise compactifications to construct more general notions of limits called \emph{regulated limits}.

In the topological setting, the idea of fibrewise compactifications  dates back to Whyburn's unified space of a continuous map \cite{Why53,Why66}. The theory has since been further developed and refined \cite{AnDe14,Cain69,Jam89}.
 The idea is to start with a continuous map $f \colon X \to Y$ and ``compactify its fibres'' to extend $f$ to a continuous proper map $\widetilde{f} \colon \widetilde{X} \to Y$, whose domain $\wt{X}$ is a new locally compact Hausdorff space containing $X$ as an open subspace. 
 The original approach of
 Whyburn was to ``glue''  $Y$ to the fibres $\{f^{-1}(y)\}_{y \in Y}$ of $f$ ``at infinity''.
 The existing literature on fibrewise compactifications assumes that $f \colon X \to Y$ is surjective, so we formulate a more general definition (see \cref{dfn:perfection}) to handle non-surjective maps. We call a fibrewise compactification in which $\widetilde{f}$ is also surjective a \emph{perfection} since the maps $\wt{f}$ are perfect (continuous, proper, and surjective).

Inductively applying fibrewise compactifications to \labelcref{eq:projective_sequence}, yields a new sequence 
 \begin{equation}\label{eq:projective_sequence_proper}
 	\begin{tikzcd}[ampersand replacement=\&,cramped]
 		\wt{X}_0 \& \wt{X}_1 \& \wt{X}_2 \& \cdots
 		\arrow["\wt{f}_0"', from=1-2, to=1-1]
 		\arrow["\wt{f}_1"', from=1-3, to=1-2]
 		\arrow["\wt{f}_2"', from=1-4, to=1-3]
 	\end{tikzcd}
 \end{equation}
 of locally compact Hausdorff spaces,
  connected by proper continuous maps. 
  Regulated limits in this setting are a class of limits of the form $\varprojlim(\wt{X}_i,\wt{f}_i)$ that depend on a suitable sequence of open sets $U_i \subseteq X_i$.  Regulated limits are locally compact and Hausdorff, and the original limit $\varprojlim(X_i,f_i)$ continuously embeds in $\varprojlim(\wt{X}_i,\wt{f}_i)$ (see \cref{cor:regulated_limit}). 
  
  Different choices of fibrewise compactification result in different regulated limits. We highlight three constructions that are somewhat canonical. The first construction is based on Whyburn's unified space, which we call the \emph{unified limit} $\unilim(X_i,f_i)$. The second construction $\minlim(X_i,f_i)$ is the smallest regulated limit of $(X_i,f_i)_{i \in \NN}$, and the third $\perlim(X_i,f_i)$ is the smallest regulated limit in which the universal projections $\perlim(X_i,f_i) \to \wt{X}_i$ are surjective.

Examples of regulated limits---and part of the author's motivation for introducing them (cf. \cite[Chapter~3]{MunPhD})---comes from directed-graph $C^*$-algebras. The \emph{path space} $E^{\le \infty}$ and \emph{boundary path space} $\partial E$ of a directed (or topological) graph $E$ are locally compact Hausdorff spaces that contain the space of one-sided infinite paths $E^{\infty}$ of the graph (see \cite{deCa21,Web14,Yen06}). In \cref{ex:top_graph_1} we show that $E^{\le \infty}$ can be realised as a regulated limit of the form $\unilim(X_i,f_i)$, and in \cref{ex:top_graph_2} we show that $\partial E$ can be realised as a regulated limit of the form $\perlim(X_i,f_i)$.  Regulated limits provide a systematic way of studying these spaces.

In the $C^*$-algebraic setting, we extend the noncommutative topology paradigm to fibrewise compactifications and regulated limits. A noncommutative fibrewise compactification (see \cref{dfn:perfection_nc}) of a morphism $\varphi \colon A \to \Mm(B)$ consists of a new $C^*$-algebra $\wt{B}$ containing $B$ as an ideal together with a $*$-homomorphism $\wt{\varphi} \colon A \to \wt{B}$ that takes an approximate unit of $A$ to an approximate unit of $\wt{B}$. Following \cite{ELP99}, we call such a $*$-homomorphism a \emph{proper morphism} from $A$ to $\wt{B}$. 
 We also insist that $ \alpha_B \circ \wt{\varphi} = \varphi$, where $\alpha_B \colon \wt{B} \to \Mm(B)$ is the $*$-homomorphism induced by the inclusion of $B$ as an ideal in $\wt{B}$. A noncommutative perfection is a noncommutative fibrewise compactification in which $\wt{\varphi}$ is also injective. 
 In \cref{prop:fw_commutative_same} we show that for commutative $C^*$-algebras $A$ and $B$, a noncommutative fibrewise compactification of $\varphi$ corresponds directly to a fibrewise compactification of the induced map $\varphi_* \colon \widehat{B} \to \widehat{A}$ between the spectra.
 
The noncommutative analogous of Whyburn's unified spaces, \emph{unified algebras}, correspond precisely to split extensions of $C^*$-algebras, and we show in \cref{prop:unified_cstar} that every split extension of \emph{commutative} $C^*$-algebras is isomorphic to the continuous functions vanishing at infinity on a unified space. In \cref{thm:unified_quotient} we show that sub-fibrewise compactifications the unified algebra of $\varphi \colon A \to \Mm(B)$ correspond to a more general class of extensions of $A$ by $B$, complementing the classification of universal extensions of \cite{ELP99}.
 The ideals defined by Pimsner \cite[Definition~3.8]{Pim97} and Katsura \cite[Definition~3.2]{Kat04cor} in their studies of Cuntz--Pimsner algebras arise naturally as essential ingredients in our constructions.

We use our noncommutative fibrewise compactifications to develop a notion of regulated limit for sequences of morphisms $(\varphi_i \colon A_i \to \Mm(A_{i+1}))_{i \in \NN}$. We inductively apply noncommutative fibrewise compactifications to such a sequence to get a directed sequence
\[
\begin{tikzcd}[ampersand replacement=\&,cramped]\label{eq:inductive_sequence_proper}
	{A_0} \& {\wt{A}_1} \& {\wt{A}_2} \& \cdots
	\arrow["{\wt{\varphi}_0}", from=1-1, to=1-2]
	\arrow["{\wt{\varphi}_1}", from=1-2, to=1-3]
	\arrow["{\wt{\varphi}_2}", from=1-3, to=1-4]
\end{tikzcd}
\]
 of $C^*$-algebras and proper morphisms $\wt{\varphi}_i$. A regulated limit in this setting is a limit of the form  $\varinjlim(\wt{A_i},\wt{\varphi}_i)$ that depends on a suitable sequence of ideals $J_i \trianglelefteq A_i$ (see \cref{dfn:regulated_limit_alg}). 
A regulated limit $\varinjlim(\wt{A_i},\wt{\varphi}_i)$ is compatible with the sequence $(\varphi_i \colon A_i \to \Mm(A_{i+1}))_{i \in \NN}$ in the sense that each $A_i$ is an ideal in $\wt{A}_i$, and if $\alpha_i \colon \wt{A}_i \to \Mm(A_i)$ denotes the induced morphism, then $\varphi_i = \wt{\varphi}_i \circ \alpha_i$. 
As in the topological setting, we focus on three  regulated limits $\unilima(A_i,\varphi_i)$, $\perlima(A_i,\varphi_i)$, and $\minlima(A_i,\varphi_i)$. 

In \cref{prop:cuntz-pimsner_characterisation2} we show that the cores of relative Cuntz--Pimsner algebras $\Oo_{X,J}$ associated to a $C^*$-correspondence $(\varphi, X_A)$ are examples of regulated limits. The core of the Toeplitz algebra $\Tt_X$ is a regulated limit of the form $\unilima(A_i,\varphi_i)$, and the core Cuntz--Pimsner algebra $\Oo_{X}$ is a regulated limit of the form $\perlima(A_i,\varphi_i)$, where in this instance $A_i = \Kk_A(X^{\ox i})$ is the algebra of generalised compact operators on the $i$-th tensor power of $X$.
 This example gives insight in to the sense in which the Toeplitz algebra and Cuntz--Pimsner algebra constructions are limiting procedures. 
 
 Although our attention is restricted regulated limits of locally compact Hausdorff spaces and $C^*$-algebras, it appears as though there is more scope for similar constructions. For groupoids---which in many cases bridge the gap between topological spaces and noncommutative $C^*$-algebras---related work has already appeared.  
 Groupoid-equivariant fibrewise compactifications were introduced in \cite{AnDe14}, and in \cite[Theorem~3.3.43]{MunPhD} they are used to prove a groupoid-theoretic version of \cref{prop:cuntz-pimsner_characterisation2} in the special case of topological graphs. More recently, these constructions have been used to construct universal actions of \'etale groupoids \cite[Theorem~4.16]{Jun23}.

 The paper is partitioned into two sections, with the structure of the $C^*$-algebraic \cref{sec:CStar} mirroring that of the topological \cref{sec:topological}.  The reader interested in topology can read \cref{sec:topological} without \cref{sec:CStar}, and the reader interested in
 $C^*$-algebras should be able to read most of \cref{sec:CStar} without having read \cref{sec:topological}, although some context will certainly be lost.  
 
 In Subsections~\ref{sec:fw_comp} and \ref{sec:nc_fw} we give definitions of fibrewise compactifications and establish some basic facts. 
  In Subsection~\ref{sec:unified_space} we introduce Whyburn's unified space, and in Subsection~\ref{sec:unified_alg} we introduce the analogous unified algebra. In Subsection~\ref{sec:sub-fibrewise} we classify sub-fibrewise compactifications of the unified space, and in Subsection~\ref{sec:quotient_fw} we classify quotient fibrewise compactifications of the unified algebra. In Subsections~\ref{sec:commutative_composition} and \ref{sec:nc_composition} we examine the composition of fibrewise compactifications in preparation for introducing regulated limits in Subsections~\ref{subsec:lch_inverse_limits} and \ref{sec:nc_limits}. In Subsection~\ref{sec:cores} we prove that the cores of relative Cuntz--Pimsner algebras are regulated limits.

\subsection*{Acknowledgements}

The author would like to thank Adam Rennie and Aidan Sims for their suggestions. 
This research was supported by Australian Research Council grant DP220101631, University of Wollongong AEGiS CONNECT grant 141765, and an Australian Government Research Training Program (RTP) Scholarship.

\section{The topological setting}\label{sec:topological}

Recall that a map $f \colon X \to Y$ between topological spaces is \emph{proper} if $f^{-1}(K)$ is compact for all compact $K \subseteq Y$. In particular, the \emph{fibres} $f^{-1}(y)$ of $f$ are compact for all $y \in Y$. 
We say that $f$ is \emph{perfect} if it is continuous, closed, surjective, and for every $y \in Y$ the fibre $f^{-1}(y)$ is compact in $X$. If $Y$ is locally compact and Hausdorff, then perfect maps coincide with continuous proper surjections. Almost every topological space encountered in this paper is locally compact and Hausdorff, so we never need to worry about the distinction between perfect maps and continuous proper surjections. 

Perfect maps respect local compactness; the continuous image of a locally compact space under a perfect map is again locally compact \cite{Mic72}. We recall that a subset of a topological space is \emph{precompact} if it has compact closure. Locally compact Hausdorff spaces admit bases consisting of precompact open sets \cite[Theorem~3.2.2]{Eng89}.

\subsection{Fibrewise compactifications and perfections}\label{sec:fw_comp}

Given a continuous map $f \colon X \to Y$ we define
\begin{align}\label{eq:pimsner_set}
	\begin{split}
		\pr(f) 
		\coloneqq \bigcup \{ W \mid W \subseteq Y &\text{ is a } \text{precompact open set}\\
	 &\text{ such that } f^{-1}{(\ol{W})} \text{ is compact}\}.\\
	\end{split}
\end{align} 
The set $\pr(f)$ is open in $Y$. In \cref{lem:f-reg_characterisation} we show that $\pr(f)$ is the largest open subset $U$ of $Y$ such that the restriction $f|_{f^{-1}(U)} \colon f^{-1}(U) \to U$ is proper, but we do not require this fact for now. 

We are interested in the process of taking a continuous map and extending it to either a proper or perfect map. 
 We capture this idea with the following definition.

\begin{dfn}\label{dfn:perfection} Let $f \colon X \to Y$ be a continuous map between locally compact Hausdorff spaces. A \emph{fibrewise compactification} of $f$ is a pair $(Z,g)$ consisting of a locally compact Hausdorff space $Z$ and a continuous proper map $g \colon Z \to Y$ such that
\begin{enumerate}[labelindent=0pt,labelwidth=\widthof{\ref{perf:commuting}},label=(F\arabic{enumi}), ref=(F\arabic*),leftmargin=!]
\item\label{perf:commuting} there is an open inclusion $\iota_X \colon X \to Z$  such that $g \circ \iota_X = f$; and 

\item\label{perf:density} the restriction of $g$ to ${Z \setminus \overline{\iota_X(X)}}$ is a homeomorphism onto a  subset of $\pr(f)$ that is closed in the subspace topology of $\pr(f)$. 
\end{enumerate}
If $g$ is also surjective, then $g$ is perfect, and we call $(Z,g)$ a \emph{perfection} of $f$. If $\ol{\iota_X(X)} = Z$, then we call $(Z,g)$ a \emph{strict fibrewise compactification}. We say that fibrewise compactifications $(Z,g)$ and $(Z',g')$ of $f$ are \emph{isomorphic} if there is a homeomorphism $h \colon Z \to Z'$ such that $g' = h \circ g$ and $h|_X = \id_X$. 
\end{dfn}

Condition~\ref{perf:commuting} says that $g$ is a proper map extending $f$ to a larger domain,
 while condition~\ref{perf:density} ensures that the extension is ``not too large''. The inclusion, in $Z$, of points that lie outside of $\ol{\iota_X(X)}$ is sometimes necessary to ensure surjectivity of $g$.

The notion of a ``fibrewise compactification'' has appeared in the literature previously (see \cite[\S 8]{Jam89} and \cite{AnDe14}). In these definitions it is assumed that
\begin{enumerate}
	\item the map $f \colon X \to Y$ is surjective, and
	\item $\ol{\iota_X(X)}$ is dense in $Z$.
\end{enumerate}
If $f \colon X \to Y$ is surjective, then a strict fibrewise compactification of $f$ corresponds to a fibrewise compactification of $f$ in the sense of \cite{AnDe14,Jam89}.
If $f \colon X \to Y$ is not surjective, then things are more subtle. As the following lemma shows, points on the boundary of $f(X)$ can act as an obstruction to properness of $f$.

\begin{lem}\label{lem:boundary_issues}
	Suppose that $f \colon X \to Y$ is a continuous map between locally compact Hausdorff spaces. If $y \in \overline{f(X)} \setminus f(X)$, then for any precompact open neighbourhood $W \subseteq Y$ of $y$, the preimage $f^{-1}(\overline{W})$ is not compact in $X$. In particular, if $f$ is proper, then it is a closed map. 
\end{lem}

\begin{proof}
	Fix a net $y_\lambda \to y$ with $y_\lambda \in f(X) \cap W$ and choose $x_\lambda \in f^{-1}(y_\lambda)$. If $(x_\lambda)$ admits a subnet $(x_{\lambda'})$ that converges to some $x \in X$, then $f(x) = \lim_{\lambda'} f(x_{\lambda'}) = \lim_{\lambda'} y_{\lambda'} = y$. This contradicts that $y \notin f(X)$.
	If $f$ is proper and $C \subseteq X$ is closed, then $f|_C \colon C \to Y$ is proper, so $\ol{f(C)} = f(C)$. 
\end{proof}

The following example illustrates the boundary issues of \cref{lem:boundary_issues}.

\begin{example}\label{ex:boundary_issues}
	Let $X = (0,1)^2 \subseteq \RR^2$, and $Y = (0,1)$ with the Euclidean topologies, and define $f \colon X \to Y$ by $f(x_1,x_2) = x_1$. Let $S^1$ denote the unit circle, let $Z = (0,1) \times S^1$, and define $g \colon Z \to Y$ by $g(x_1,x_2) = x_1$. Let $\iota_X \colon X \to Z$ be the inclusion induced by the inclusion of $(0,1)$ in its one-point compactification $S^1$. 	
	Then $(Z,g)$ is a strict fibrewise compactification of $f$ (in fact it is a perfection). 
	
	Now we enlarge the codomain of $f$. Let $Y' = \RR$ and define $f' \colon X \to Y'$ by $f'(x_1,x_2) = x_1$. As before, define an extension $g' \colon Z \to Y'$ of $f'$ by $g'(x_1,x_2) = x_1$. Then $g'(Z) = (0,1)$ is not closed in $Y$.  \cref{lem:boundary_issues} implies that $g'$ is not proper, so $(Z,g')$ is not a fibrewise compactification of $f'$. 
	To address this, a fibrewise compactification of $f'$ must have nonempty fibres $(f')^{-1}(\{0\})$ and $(f')^{-1}(\{0\})$
\end{example}

To address the boundary issues of \cref{lem:boundary_issues} we introduce a somewhat inefficient fibrewise compactification in \cref{sec:unified_space}, and then restrict to sub-fibrewise compactifications in the following sense.

\begin{dfn}
	Let $f \colon X \to Y$ be a continuous map between locally compact Hausdorff spaces and suppose that $(Z,g)$ is a fibrewise compactification of $f$. A \emph{sub-fibrewise compactification} of $(Z,g)$ is a fibrewise compactification $(W,h)$ of $f$ such that $W$ is a subspace of $Z$ containing $X$, and $g|_{W} = h$. If $g$ and $h$ are both surjective, then we call $(W,h)$ a \emph{sub-perfection} of $(Z,g)$. 
\end{dfn}

Sub-fibrewise compactifications are automatically closed.

\begin{lem}
	\label{lem:subperfections_closed}
	Suppose that $(Z,g)$ is a fibrewise compactification of a continuous map $f \colon X \to Y$ between locally compact Hausdorff spaces. If $(W,h)$ is a sub-fibrewise compactification of $(Z,g)$, then $W$ is closed in $Z$. 
\end{lem}

\begin{proof}
	Suppose that $(w_{\lambda})$ is a net in $W$ that converges to some $z \in Z$. Then $h(w_{\lambda}) = g(w_{\lambda}) \to g(z)$. Fix a precompact open neighbourhood $U$ of $g(z)$. By passing to a subnet we may assume that each $h(w_{\lambda}) \in U$. Each $w_{\lambda}$ belongs to the compact set $h^{-1}(\ol{U})$ so there is a subnet of $(w_{\lambda})$ that converges in $h^{-1}(\ol{U}) \subseteq W$. Since $W$ is a subspace of the Hausdorff space $Z$, this limit must be $z$, so $z \in W$. Hence, $W$ is closed in $Z$.
\end{proof}

A fibrewise compactification may always be restricted to a strict fibrewise compactification. The following result is a consequence of \cref{lem:subperfections_closed}. 

	\begin{lem}\label{dfn:closure_subfw}\label{prop:strict_fw}
		Let $(Z,g)$ be fibrewise compactification of $f \colon X \to Y$. The sub-fibrewise compactification
		\[
		([Z],[g]) \coloneqq (\ol{\iota_X(X)},g|_{\ol{\iota_X(X)}})
		\]
		is strict. If $(W,h)$ is a sub-fibrewise compactification of $(Z,g)$, then $([Z],[g]) = ([W],[h])$. 
	\end{lem}

\subsection{The unified space}\label{sec:unified_space}
In this subsection we introduce the \emph{unified space} construction of  Whyburn~\cite{Why53}. It is our first, and simplest, example of a fibrewise compactification. Most examples of fibrewise compactifications that we consider in this article are sub-fibrewise compactifications of the unified space. 
 Given a continuous map $f \colon X \to Y$, the intuition is that we may ``glue'' a copy of the space $Y$ to $X$ to compactify the fibres of $f$. 

\begin{dfn}
	\label{dfn:unified_space}
	Let $f \colon X \to Y$ be a continuous map between locally compact Hausdorff spaces. Let $\unis{X}{f}{Y} \coloneqq X \sqcup Y$ as a set, and define $\unif{f} \colon \unis{X}{f}{Y} \to Y$  by 
	\[
	\unif{f}(x) = \begin{cases}
		f(x) & \text{if } x \in X,\\
		x & \text{if } x \in Y.
	\end{cases}
	\]
	Let $\tau_X$ and $\tau_Y$ denote the topologies on $X$ and $Y$, and let $\kappa_X$ denote the collection of compact subsets of $X$. We equip $\unis{X}{f}{Y}$ with the topology generated by the subbase
\begin{equation}
	\label{eq:unified_base}
		\Bb = \tau_X \cup \big\{ \unif{f}^{-1}(V) \cap \big( (X \sqcup Y) \setminus K\big) \bigm| V \in \tau_Y \text{ and } K \in \kappa_X \big\}.
\end{equation}
	We refer to the space $\unis{X}{f}{Y}$ and the pair $(\unis{X}{f}{Y}, \unif{f})$ as the \emph{unified space} of $f$.
\end{dfn}

\begin{rmk}
	For any $V \in \tau_Y$, we have $\unif{f}^{-1}(V) \in \Bb$, so $\unif{f}$ is continuous. 
\end{rmk}

\begin{lem}
	The subbase $\Bb$ of \labelcref{eq:unified_base} is a base. 
\end{lem}

\begin{proof}
	Since $\varnothing$ is compact, the elements of $\Bb$ cover $X \sqcup Y$. Fix $B_1,B_2 \in \Bb$. If $B_1,B_2 \in \tau_X$, then $B_1 \cap B_2 \in \tau_X \subseteq \Bb$.
	If $B_1 \in \tau_X$ and $B_2 = \unif{f}^{-1}(V) \cap ((X \sqcup Y) \setminus K)$ for some $V \in \tau_Y$ and $K \in \kappa_X$, then $B_1 \cap B_2 = B_1 \cap f^{-1}(V) \cap X \setminus K \in \tau_X \subseteq \Bb$. If $B_1 = \unif{f}^{-1}(V_1) \cap ((X \sqcup Y) \setminus K_1)$ and $B_2 = \unif{f}^{-1}(V_2) \cap ((X \sqcup Y) \setminus K_2)$ for some $V_1,V_2 \in \tau_Y$ and $K_1,K_2 \in \kappa_X$, then $B_1 \cap B_2 = \unif{f}^{-1}(V_1 \cap V_2) \cap ((X \sqcup Y) \setminus (K_1 \cup K_2)) \in \Bb$. So $\Bb$ is a base. 
\end{proof}

The following lemma provides some alternate characterisations of the topology on $\unis{X}{f}{Y}$. In particular, it shows that \cref{dfn:unified_space} agrees with the unified space definition of~\cite{Why66}. 

\begin{lem}\label{lem:unified_topology}
	Let $f \colon X \to Y$ be a continuous map between locally compact Hausdorff spaces. Then $\unis{X}{f}{Y}$ is a Hausdorff space.
	\begin{enumerate}
		\item \label{itm:open}	A set $U \subseteq \unis{X}{f}{Y}$ is open if and only if
		\begin{enumerate}[label=(\roman{enumii})]
			\item\label{itm:open1}$U \cap X$ is open in $X$;
			\item\label{itm:open2} $U \cap Y$ is open in $Y$; and
			\item\label{itm:open3} for any compact $K \subseteq U \cap Y$, the set $f^{-1}(K) \cap (X \setminus U)$ is compact in $X$. 
		\end{enumerate}
		\item \label{itm:closed}
		A set $C \subseteq \unis{X}{f}{Y}$ is closed if and only if 
		\begin{enumerate}[label=(\roman{enumii})]
			\item $C \cap X$ is closed in $X$;
			\item $C \cap Y$ is closed in $Y$; and
			\item for any compact set $K \subseteq Y \setminus C$ the set $f^{-1}(K) \cap C$ is compact in $X$. 
		\end{enumerate}
		\item \label{itm:nets}
		Let $(x_\lambda)_{\lambda \in \Lambda}$ be a net in $\unis{X}{f}{Y}$. 
		\begin{enumerate}[label=(\roman{enumii})]
			\item If $x \in X$, then $x_\lambda \to x$ if and only if there exists $\lambda_0 \in \Lambda$ such that $\lambda \ge \lambda_0$ implies $x_{\lambda} \in X$ and $x_{\lambda} \to x$ in $X$.
			\item If $x \in Y$, then $x_\lambda \to x$ if and only if $\unif{f}(x_\lambda) \to \unif{f}(x) = x$ and for any compact $K \in \kappa_X$ there exists $\lambda_K \in \Lambda$ such that $\lambda \ge \lambda_K$ implies that $x_{\lambda} \in (X \setminus K) \sqcup Y$.
		\end{enumerate}		
	\end{enumerate}
\end{lem}

\begin{proof}
	We first show that $\unis{X}{f}{Y}$ is Hausdorff. Fix $x_1,x_2 \in \unis{X}{f}{Y}$. If $x_1,x_2 \in X$, then since $X$ is Hausdorff $x_1$ and $x_2$ can be separated by neighbourhoods in $\tau_X$. If $x_1,x_2 \in Y$, then since $Y$ is Hausdorff there exist disjoint open sets $U_1,U_2 \in \tau_Y$ with $x_1 \in U_1$ and $x_2 \in U_2$. Then $\unif{f}^{-1}(U_1)$ and $\unif{f}^{-1}(U_2)$ are disjoint open neighbourhoods of $x_1$ and $x_2$ in $\unis{X}{f}{Y}$. 
	Finally, suppose that $x_1 \in X$ and $x_2 \in Y$. Fix a precompact open neighbourhood $U \in \tau_X$ of $x_1$ and an open neighbourhood $V \in \tau_Y$ of $x_2$. Then $U$ and $\unif{f}^{-1}(V) \cap ((X \sqcup Y) \setminus \ol{U})$ are disjoint open neighbourhoods of $x_1$ and $x_2$ in $\unis{X}{f}{Y}$. So $\unis{X}{f}{Y}$ is Hausdorff. 
	
		 For \labelcref{itm:open}, we first show that each basic open set in $\Bb$ satisfies conditions \labelcref{itm:open1,itm:open2,itm:open3} of \labelcref{itm:open}. If $U \in \tau_X \subseteq \Bb$, then \labelcref{itm:open1,itm:open2,itm:open3} hold. So suppose that $V \in \tau_Y$ and $K \in \kappa_X$, and consider the basic open set $B \coloneqq \unif{f}^{-1}(V) \cap ( (X \sqcup Y) \setminus K)$.
	Then $B \cap X = f^{-1}(V) \cap (X \setminus K)$ is open in $X$, giving \labelcref{itm:open1}, and $B  \cap Y = V$, giving \labelcref{itm:open2}. For \labelcref{itm:open3}, observe that if $C$ is a compact subset of $V$, then $
		f^{-1}(C) \cap (X \setminus B) = f^{-1}(C) \cap K 
	$ is compact since it is a closed subset of $K$. So open sets in $\unis{X}{f}{Y}$ satisfy \labelcref{itm:open1,itm:open2,itm:open3} of \labelcref{itm:open}.
	
	For the converse, it is shown in \cite[\S 4]{Why53} that the sets $U \subseteq X \sqcup Y$ satisfying \labelcref{itm:open1,itm:open2,itm:open3} of \labelcref{itm:open} determine a topology on $X \sqcup Y$, so it suffices to show that $\Bb$ forms a base for this topology. To this end, fix $x \in \unis{X}{f}{Y}$ and suppose that $W \subseteq \unis{X}{f}{Y}$ is a neighbourhood of $x$ satisfying \labelcref{itm:open1,itm:open2,itm:open3}. We show that there exists $B \in \Bb$ such that $ x \in B \subseteq W$. 
	If $x \in W \cap X$, then $W \cap X \in \tau_X \subseteq \Bb$.
	 So suppose that $x \in W \cap Y$. Since $Y$ is locally compact and $W \cap Y$ is open in $Y$, there exists a precompact open neighbourhood $V$ of $x$ such that the closure $\ol{V}$ in $Y$ is contained in $W \cap Y$. Condition \labelcref{itm:open3} implies that $f^{-1}(\ol{V}) \cap (X \setminus W)$ is compact in $X$. Let $B \coloneq \unif{f}^{-1}(V) \cap ((X \sqcup Y) \setminus (f^{-1}(\ol{V}) \cap (X \setminus W))) \in \Bb$. Then $x \in B$ and  
	\begin{align*}
		B&= \unif{f}^{-1}(V) \cap \Big(\big((X \sqcup Y) \setminus f^{-1}(\ol{V})\big)
		\cup \big((X \sqcup Y)\setminus (X \setminus W)\big)
		\Big)\\
		&= \big(
		\unif{f}^{-1}(V) \cap ((X \sqcup Y)\setminus f^{-1}(\ol{V}))
		 \big)
		 \cup
		 \big(
		 \unif{f}^{-1}(V) \cap W
		 \big)
		 \cup
		 \big(
		 f^{-1}(V) \cap Y
		 \big)\\
		 &= V \cup (\unif{f}^{-1}(V) \cap W) \cup V
		 \subseteq W.
	\end{align*}
	So $\Bb$ is a base for the topology generated by sets satisfying \labelcref{itm:open1,itm:open2,itm:open3}. 
	
		Parts \labelcref{itm:closed} and \labelcref{itm:nets} follow from \labelcref{itm:open}. 
\end{proof}

We summarise some basic properties of the unified space. Most of these properties can be found throughout {\cite[\S 3]{Why66}}, but we reprove them in contemporary language.  

\begin{prop}
	\label{prop:unified_properties}
	Let $f \colon X \to Y$ be a continuous map between locally compact Hausdorff spaces.
	\begin{enumerate}
	
		\item\label{itm:unified_ix_open} The inclusion $\iota_X \colon X \hookrightarrow \unis{X}{f}{Y}$ is continuous and open.
		\item \label{itm:unified_iy_closed} The inclusion $\iota_Y \colon Y \hookrightarrow \unis{X}{f}{Y}$ is continuous and closed. 
		\item \label{itm:unified_1} The map $\unif{f} \colon \unis{X}{f}{Y} \to Y$ is perfect.
		\item \label{itm:compactness_characterisation} A set $K \subseteq \unis{X}{f}{Y}$ is compact if and only if $K$ is closed in $\unis{X}{f}{Y}$ and there is a compact set $K' \in \kappa_Y$ such that $K \subseteq \unif{f}^{-1}(K')$.
		\item \label{itm:local_compactness} The space $\unis{X}{f}{Y}$ is locally compact. 
		
		\item \label{itm:unified_6}  The pair $(\unis{X}{f}{Y}, \unif{f})$ is a perfection of $f$ and $\unif{f}(\unis{X}{f}{Y} \setminus \ol{\iota_X(X)}) = \pr(f)$. 
		
		\item \label{itm:unified_4} If $f$ is proper, then $\unis{X}{f}{Y}$ is $X \sqcup Y$ with the disjoint union topology.
		\item \label{itm:unified_5} If both $X$ and $Y$ are second countable, then $\unis{X}{f}{Y}$ is metrisable. 
	\end{enumerate}
\end{prop}

\begin{proof}
	\labelcref{itm:unified_ix_open} and  \labelcref{itm:unified_iy_closed} follow directly from \cref{lem:unified_topology}. 
	
	\labelcref{itm:unified_1} 
	The map $\unif{f}$ is surjective by construction. For properness, fix a compact set $K \subseteq Y$ and let $(x_{\lambda})$ be a net in $\unif{f}^{-1}(K)$. We show that $(x_{\lambda})$ admits a convergent subnet. Since $(\unif{f}(x_\lambda))$ is a net in the compact set $K$, we may pass to a subnet to guarantee that $(\unif{f}(x_{\lambda}))$ converges to some $y \in K$.
	If, for every $C \in \kappa_X$, there exists $\lambda_C$ such that $\lambda \ge \lambda_C$ implies $x_{\lambda} \notin C$, then \cref{lem:unified_topology}~\labelcref{itm:nets} says that $x_{\lambda} \to y \in \unif{f}^{-1}(K)$. Otherwise, there exists $C \in \kappa_X$ and a subnet $(x_{\lambda'})$ contained in $C$. Since $C$ is compact in $X$, we can pass to a subnet to guarantee that $(x_{\lambda'})$ converges in $C$. Since the limit lies in $X$, which is open in $\unis{X}{f}{Y}$, the convergence also occurs in $\unis{X}{f}{Y}$. So $\unif{f}^{-1}(K)$ is compact, and hence $\unif{f}$ is perfect.

	\labelcref{itm:compactness_characterisation} Suppose that $K \subseteq \unis{X}{f}{Y}$ is closed and there is a compact set $K' \subseteq Y$ such that $K \subseteq \unif{f}^{-1}(K')$. Since $\unif{f}$ is perfect, $\unif{f}^{-1}(K')$ is compact, so $K$ is compact. On the other hand, if $K \subseteq \unis{X}{f}{Y}$ is compact, then it is closed and contained in the compact set $\unif{f}^{-1}(\unif{f}(K))$.
	
	\labelcref{itm:local_compactness} Fix $x \in \unis{X}{f}{Y}$. If $x \in X$, then by local compactness of $X$ there exists an open neighbourhood $U \in \tau_X$ of $x$ such that the closure $\ol{U}$ in $X$ is compact in $X$. Then $\unif{f}^{-1}(\unif{f}(\ol{U}))$ is a compact set containing $U$. If $x \in Y$, then local compactness of $Y$ yields an open neighbourhood $V \in \tau_Y$ of $x$ such that the closure $\ol{V}$ in $Y$ is compact in $Y$. Then $\unif{f}^{-1}(\ol{V})$ is compact and contains the open neighbourhood $\unif{f}^{-1}(V)$ of $x$.
	
		\labelcref{itm:unified_6} Part \labelcref{itm:unified_1} says that $\unif{f}$ is perfect. Since $\iota_X \colon X \to \unis{X}{f}{Y}$ is open and $\unif{f} \circ \iota_X = f$, \labelcref{perf:commuting} is satisfied. 
	Since $\unis{X}{f}{Y} \setminus \ol{\iota_X(X)} \subseteq \iota_Y(Y)$, the restriction of $\unif{f}$ to $\unis{X}{f}{Y} \setminus \ol{\iota_X(X)}$ is a homeomorphism onto its image. We show that $\unif{f}(\unis{X}{f}{Y} \setminus \ol{\iota_X(X)}) = \pr(f)$ from which \labelcref{perf:density} follows.
	
	First suppose that $y \in \pr(f)$. Fix a precompact open neighbourhood $W \subseteq \pr(f)$ of $y$ such that $f^{-1}(\ol{W})$ is compact.  If $K \subseteq W$ is compact, then $f^{-1}(K) 
	\subseteq  f^{-1}(\ol{W})$ is compact in $X$. By \cref{lem:unified_topology}, the set $\iota_Y(W)$ is open in $\unis{X}{f}{Y}$. Since $\iota_Y(W)$ is an open neighbourhood of $\iota_Y(y)$ that is disjoint from $\iota_X(X)$, we have $ \pr(f) \subseteq \unif{f}(\unis{X}{f}{Y} \setminus \ol{\iota_X(X)})$.
	
	Now suppose that $y \in \unif{f}(\unis{X}{f}{Y} \setminus \ol{\iota_X(X)})$. Then there is an open set $W \subseteq \unis{X}{f}{Y}$  such that $\iota_Y(y) \in W$ and $W \cap X = \varnothing$. By \cref{lem:unified_topology}, the set $W \cap Y$ is open in $Y$ and for any compact set $K \subseteq W$ the set  $f^{-1}(K) \cap (X \setminus W) = f^{-1}(K)$ is compact in $X$. Since $Y$ is locally compact there is a precompact open neighbourhood $U \subseteq \ol{U} \subseteq W \cap Y$ of $y$ in $Y$. Then $f^{-1}(\ol{U})$ is compact, and so $y \in \pr(f)$. 		
	
	\labelcref{itm:unified_4} If $f$ is proper, then $Y = \unis{X}{f}{Y} \setminus f^{-1}(Y)$ is open in $\unis{X}{f}{Y}$. By \cref{lem:unified_topology}, the subspace topologies on $X$ and $Y$ agree with $\tau_X$ and $\tau_Y$. 
	
	\labelcref{itm:unified_5} Choose a countable base $\Bb_X$ for $X$ and $\Bb_Y$ for $Y$. Since second-countable locally compact Hausdorff spaces are $\sigma$-compact, there is an increasing sequence of compact subsets $(K_i)_{i \in \NN}$ of $X$ such that $X = \bigcup_{i \in \NN} K_i$. For any compact $K \subseteq X$, there exists $i \in \NN$ such that $K \subseteq K_i$, so $\Bb_X \cup \{ (\unif{f})^{-1}(V) \cap ((X \sqcup Y) \setminus K_i) \mid V \in \Bb_Y ,\, i \in \NN\}$ is a countable base for $\unis{X}{f}{Y}$. Hence, $\unis{X}{f}{Y}$ is second countable, and since second-countable locally compact Hausdorff spaces are metrisable, $\unis{X}{f}{Y}$ is metrisable. 
\end{proof}

\begin{dfn}
	Let $f \colon X \to Y$ be a continuous map between locally compact Hausdorff spaces. With the notation established in \cref{dfn:closure_subfw}, we call the strict fibrewise compactification $([\unis{X}{f}{Y}],[\unif{f}])$ the \emph{minimal fibrewise compactification} of $f$. 
\end{dfn}

\begin{rmk}\label{rmk:minimal_fw}
	The term ``minimal'' is justified for $([\unis{X}{f}{Y}],[\unif{f}])$. For each $y \in f(X)$ there is precisely one point added to compactify the fibre $f^{-1}(y)$. For each $y \in \ol{f(X)} \setminus f(X)$ there is precisely one point in $[\unif{f}]^{-1}(y)$ which is required in light of \cref{lem:boundary_issues}. The set $\iota_Y(Y \setminus \ol{f(X)})$ is open in $\unis{X}{f}{Y}$ and disjoint from $\ol{\iota_X(X)}$, so for every $y \in Y \setminus \ol{f(X)}$ the set $[\unif{f}]^{-1}(y)$ is empty. Thus, in forming $[\unis{X}{f}{Y}]$ from $X$ we have included the fewest points required to ensure a proper extension of $f$. 
\end{rmk}

\begin{example}
	If $X$ is a locally compact Hausdorff space and $f \colon X \to \{\infty\}$ is the map given by $f(x) = \infty$ for all $x \in X$, then $\unis{X}{f}{\{\infty\}}$ is the one-point compactification of $X$, including the isolated point $\{\infty\}$ if $X$ is already compact. The minimal fibrewise compactification corresponds to the one-point compactification where we do not include an isolated point if $f$ is compact. 
\end{example}

\begin{example}\label{ex:fun_pictures}
	Let $X = ((0,1] \times [0,1)) \cup ((1,2] \times [0,1])$ with the subspace topology from $\RR^2$. Let $Y = \RR$ and  $f \colon X \to Y$ be the projection onto the first component. In the left-hand diagram of 
	
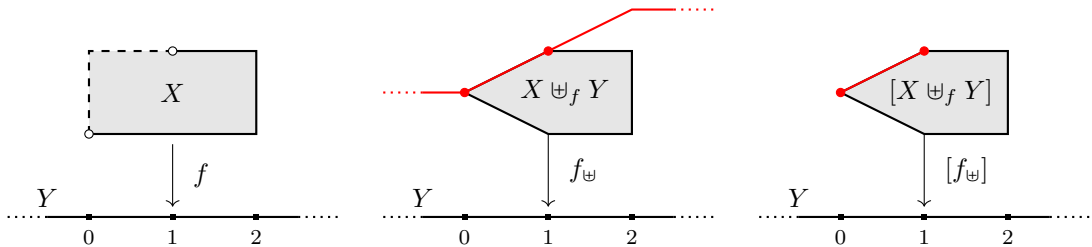
\begin{center}
			\begin{tikzpicture}[scale=1.1]
				\begin{scope}[font=\small]
					
					\node (shift) at (0,1.5) {};
					
					\coordinate (AL) at (-1,0);
					\node (BL) at (0,0) {};
					\coordinate (CL) at (1,0);
					\coordinate (AU) at (-1,1);
					\coordinate (BU) at (0,1);
					\coordinate (CU) at (1,1);
					
					\coordinate[label={[label distance=0cm]90:$Y$}] (LL) at (-1.5,-1);
					\coordinate (LLE) at (-2,-1);
					\coordinate (LR) at (1.5,-1);
					\coordinate (LRE) at (2,-1);
					\node (LM) at (0,-1) {};

					\fill[black!10!white] (AL) rectangle (CU);
					
					\draw [dashed,thick] (AL) -- (AU) -- (BU);
					\draw [thick] (AL) -- (CL) -- (CU) -- (BU);
					
					\draw [thick] (LL) -- (LR);
					\draw [thick,dotted] (LL) -- (LLE) (LR) -- (LRE);
					
					\node (M) at (0,0.5) {$X$};
					
					\draw[fill=white]
					(AL) circle (0.5mm) (BU) circle (0.5mm);
					
					\draw [->] (BL.south) -- (LM.north) node[midway, label={[label distance=0cm]0:$f$}] {};
					
					\node[fill,inner sep=1pt,label=270:{\scriptsize$0$}] (0) at (-1,-1) {};
					\node[fill,inner sep=1pt,label=270:{\scriptsize$1$}] (0) at (0,-1) {};
					\node[fill,inner sep=1pt,label=270:{\scriptsize$2$}] (0) at (1,-1) {};
					
				\end{scope}
				
			\end{tikzpicture}
\quad   
			\begin{tikzpicture}[scale=1.1]
				\begin{scope}[font=\small]
					
					\node (BN) at (0,0) {};
					\coordinate (BL) at (0,0) {};
					\coordinate (CL) at (1,0);
					\coordinate (A) at (-1,0.5);
					\coordinate (AL) at (-1.5,0.5);
					\coordinate (ALE) at (-2,0.5);
					\coordinate (BU) at (0,1);
					\coordinate (CU) at (1,1);
					\coordinate (CUU) at (1,1.5);
					\coordinate (CUUR) at (1.5,1.5);
					\coordinate (CUURE) at (2,1.5);
					
					\coordinate[label={[label distance=0cm]90:$Y$}] (LL) at (-1.5,-1);
					\coordinate (LLE) at (-2,-1);
					\coordinate (LR) at (1.5,-1);
					\coordinate (LRE) at (2,-1);
					\node (LM) at (0,-1) {};
					
					
					\draw [thick,fill=black!10!white] (A) -- (BL) -- (CL) -- (CU) -- (BU) -- cycle;
					
					\draw [thick,red] (A) -- (AL);
					\draw [thick,dotted,red] (AL) -- (ALE);
					\draw [thick,red] (BU) -- (CUU) -- (CUUR);
					\draw [thick,dotted,red] (CUUR)  -- (CUURE);
					\draw [thick,red] (A) -- (BU);
					
					\draw [thick] (LL) -- (LR);
					\draw [thick,dotted] (LL) -- (LLE) (LR) -- (LRE);
					
					\node (M) at (0.2,0.5) {$\unis{X}{f}{Y}$};
					
					\draw[red,fill=red]
					(A) circle (0.5mm)
					(BU) circle (0.5mm);

					\draw [->] (BL.south) -- (LM.north) node[midway, label={[label distance=0cm]0:$\unif{f}$}] {};
					
					\node[fill,inner sep=1pt,label=270:{\scriptsize$0$}] (0) at (-1,-1) {};
					\node[fill,inner sep=1pt,label=270:{\scriptsize$1$}] (0) at (0,-1) {};
					\node[fill,inner sep=1pt,label=270:{\scriptsize$2$}] (0) at (1,-1) {};
				\end{scope}
				
			\end{tikzpicture}
			\quad
			\begin{tikzpicture}[scale=1.1]
				\begin{scope}[font=\small]
					
					\node (BN) at (0,0) {};
					\coordinate (BL) at (0,0) {};
					\coordinate (CL) at (1,0);
					\coordinate (A) at (-1,0.5);
					\coordinate (AL) at (-1.5,0.5);
					\coordinate (ALE) at (-2,0.5);
					\coordinate (BU) at (0,1);
					\coordinate (CU) at (1,1);
					\coordinate (CUU) at (1,1.5);
					\coordinate (CUUR) at (1.5,1.5);
					\coordinate (CUURE) at (2,1.5);
					
					\coordinate[label={[label distance=0cm]90:$Y$}] (LL) at (-1.5,-1);
					\coordinate (LLE) at (-2,-1);
					\coordinate (LR) at (1.5,-1);
					\coordinate (LRE) at (2,-1);
					\node (LM) at (0,-1) {};
					
					
					\draw [thick,fill=black!10!white] (A) -- (BL) -- (CL) -- (CU) -- (BU) -- cycle;
					
					\draw [thick,red] (A) -- (BU);
					
					\draw [thick] (LL) -- (LR);
					\draw [thick,dotted] (LL) -- (LLE) (LR) -- (LRE);
					
					\node (M) at (0.2,0.5) {$[\unis{X}{f}{Y}]$};
					
					\draw[red,fill=red]
					(A) circle (0.5mm)
					(BU) circle (0.5mm);
					
					\draw [->] (BL.south) -- (LM.north) node[midway, label={[label distance=0cm]0:$[\unif{f}]$}] {};
					
					\node[fill,inner sep=1pt,label=270:{\scriptsize$0$}] (0) at (-1,-1) {};
					\node[fill,inner sep=1pt,label=270:{\scriptsize$1$}] (0) at (0,-1) {};
					\node[fill,inner sep=1pt,label=270:{\scriptsize$2$}] (0) at (1,-1) {};
				\end{scope}
				
			\end{tikzpicture}
			
	\end{center}
	the map $f$ is given by projecting $X$ orthogonally onto the line $Y$.	
	The centre diagram illustrates the unified space $(\unis{X}{f}{Y},\unif{f})$ of $f$, up to isomorphism. The closed copy of $Y$ embedded in $\unis{X}{f}{Y}$ is highlighted in red. The right-hand diagram illustrates the minimal fibrewise compactification $([\unis{X}{f}{Y}],[\unif{f}])$.  In each diagram, open circles represent the exclusion of a point, while filled-in circles represent the inclusion of a point. 	
\end{example}

\begin{example}
	As in \cref{ex:boundary_issues}, let $X = (0,1)^2 \subseteq \RR^2$, let $Y = \RR$, and define $f \colon X \to Y$ by $f(x_1,x_2) = x_1$. Let $S^2 \subseteq \RR^3$ denote the $2$-sphere of radius $1/2$ centred at $(1/2,0,0)$. Then $\unis{X}{f}{Y}$ can be identified with the topological subspace of $\RR^3$ given by
	\[
\big(	(-\infty,0] \times \{0\} \times \{0\}\big) \cup S^2 \cup \big(	[1,\infty) \times \{0\} \times \{0\}\big) 
	\]
	consisting of two infinite rays attached to the sphere $S^2$, and $\unif{f}$ can be identified with the projection onto the first component. The closed embedding of $Y$ in $\unis{X}{f}{Y}$ can be identified with the subspace
	\[
	\big(	(-\infty,0] \times \{0\} \times \{0\}\big) \cup \big\{(x,y,0) \in S^2 \colon y \ge  0 \big\} \cup \big(	[1,\infty) \times \{0\} \times \{0\}\big).
	\]
	The minimal fibrewise compactification consists of the sphere $S^2$ and the projection onto the first component. 
\end{example}

	For the remainder of this section we are exclusively concerned with sub-fibrewise compactifications of the unified space; however larger fibrewise compactifications do exist. If $f \colon X \to Y$ is surjective then it is shown in \cite[Remarks~1.3]{AnDe14} that there is an analogous fibrewise compactification to the Stone--\v{C}ech compactification. 
\begin{rmk}\label{rmk:maximalish_fw}
	We outline the construction of a ``large'' fibrewise compactification. For each locally compact Hausdorff space $X$, let $\beta X$ denote its Stone--\v{C}ech compactification. Let $f \colon X \to Y$ be a continuous map between locally compact Hausdorff spaces and consider the extension $\ol{f} \colon X \to \beta \ol{f(X)}$. By the universal property of the Stone--\v{C}ech compactification, $\ol{f}$ extends to a unique continuous map $\beta {\ol{f}} \colon \beta X \to \beta \ol{f(X)}$ such that $(\beta \ol{f})|_X = \ol{f}$. Let $Z = (\beta\ol{f})^{-1}(\ol{f(X)})$ and let $g$ be the composition $Z \overset{\beta \ol{f}}{\to} \ol{f(X)} \hookrightarrow Y$. 
	
	Since $\ol{f(X)}$ is open in $\beta \ol{f(X)}$, $Z$ is open in $\beta X$. Since $X$ is open in $\beta X$ and $X \subseteq Z$, $X$ is open in $Z$. Moreover, $g|_X = f$. If $K \subseteq Y$ is compact, then $g^{-1}(K) = (\beta \ol{f})^{-1}(K \cap \ol{f(X)})$ is compact in $Z$, so $g$ is proper. Since $X$ is dense in $\beta X$ it is also dense in $Z$, and so $(Z,g)$ is a strict fibrewise compactification of $f$. 
	
	By \cite[Theorem~10.13]{GM60}, the space $Z$ is the largest subspace of $\beta X$ for which $f$ has a continuous proper extension into $Y$. Thus, we suspect that $(Z,g)$ is universal among strict fibrewise compactifications, but we leave this question and its correct formulation open. 
\end{rmk}

\subsection{Sub-fibrewise compactifications of the unified space}
\label{sec:sub-fibrewise}

Although the unified space construction yields a perfection of $f \colon X \to Y$, it is by no means efficient. There are typically many sub-fibrewise compactifications and sub-perfections. To classify these we introduce the following notion. 

\begin{dfn}\label{dfn:f-perfect}
	Let $f \colon X \to Y$ be a continuous map between locally compact Hausdorff spaces. An open set $U \subseteq Y$ is said to be \emph{$f$-proper} if the restriction $\restr{f}{f^{-1}(U)} \colon f^{-1}(U) \to U$ is proper, and 
	 \emph{$f$-perfect} if the restriction is perfect.
\end{dfn}
 Given $f \colon X \to Y$ we consider the open set
\[
\per(f)
\coloneqq \pr(f) \cap \intt(\ol{f(X)}).
\]

\begin{lem}[{cf.~\cite[Remark~1.7]{AnDe14}~and~\cite[Proposition~2.8]{Kat04}}] \label{lem:f-reg_characterisation}Let $f \colon X \to Y$ be a continuous map between locally compact Hausdorff spaces and let $U \subseteq Y$ be open. Then
	\begin{enumerate}
		\item $U$ is $f$-proper if and only if $U \subseteq \pr(f)$; and
		\item $U$ is $f$-perfect if and only if $U \subseteq \per(f)$.
	\end{enumerate}
In particular, $\pr(f)$ is the maximal $f$-proper subset of $Y$, and $\per(f)$ is the maximal $f$-perfect subset of $Y$. 
\end{lem}

\begin{proof}
	First suppose that $U \subseteq \pr(f)$ is open. Fix $y \in U$. Then $y$ has a precompact open neighbourhood $W$ such that $f^{-1}(\ol{W})$ is compact.
	 Since $U \cap W$ is precompact and $f^{-1}(\ol{U \cap W}) \subseteq f^{-1}(\ol{W})$ is compact, we may assume---without loss of generality---that $W \subseteq U$. 
	 If $K \subseteq W$ is compact, then its preimage $f^{-1}(K) \subseteq f^{-1}(W) \subseteq f^{-1}(\ol{W})$ is compact in $f^{-1}(W)$, and is therefore compact in $f^{-1}(U)$. Hence, $U$ is $f$-proper. 
	 Conversely, suppose that $U$ is $f$-proper and fix $y \in U$.  Since $Y$ is a locally compact Hausdorff space, there is a precompact open neighbourhood $W$ of $y$ such that $\ol{W} \subseteq U$. Then $f^{-1}(\ol{W})$ is compact.
	 
	 Now suppose that $U \subseteq \per(f)$ is open and fix $y \in U$. Choose a precompact open neighbourhood $W \subseteq \ol{W} \subseteq U$ of $y$ such that $f^{-1}(\ol{W})$ is compact. 
	  Since $y \in \intt(\ol{f(X)})$, there is a net $y_{\lambda} \to y$ such that $y_{\lambda} \in W$ and $y_{\lambda} = f(x_{\lambda})$ for some $x_{\lambda} \in f^{-1}(W)$. 
	 Since $f^{-1}(\ol{W})$ is compact, we can pass to a convergent subnet to assume that $x_{\lambda} \to x$ for some $x \in f^{-1}(\ol{W})$. Then $f(x) = y$, so $U$ is $f$-perfect.
	Conversely, suppose that $U$ is $f$-perfect. Then $U \setminus f(X) = \varnothing$. Hence, $U \cap (Y \setminus \ol{f(X)}) = U \setminus \ol{f(X)} = \varnothing$. Since $U$ is open, $U \cap (Y \setminus \intt(\ol{f(X)})) = U \cap \ol{Y \setminus \ol{f(X)}} = \varnothing$, so $U \subseteq \intt(\ol{f(X)})$. Hence, $U \subseteq \pr(f) \cap \intt(\ol{f(X)}) =  \per(f)$. 	
\end{proof}

If $U$ is $f$-proper, then excising $U$ from $\unis{X}{f}{Y}$ produces a sub-fibrewise compactification.

\begin{dfn}
	\label{dfn:other_perfections}
	Let $f \colon X \to Y$ be a continuous map between locally compact Hausdorff spaces and suppose that $U \subseteq Y$ is $f$-proper. Define
	\[
	\pers{X}{f}{U}{Y} \coloneqq  X \sqcup (Y \setminus
	U),
	\]
	equipped with the subspace topology of $\unis{X}{f}{Y}$, and let $\perf{f}{U}\colon\pers{X}{f}{U}{Y} \to Y$ be the restriction
	of $\unif{f}$ to ${\pers{X}{f}{U}{Y}}$.
\end{dfn}

\begin{rmk}\label{rmk:empty_set}
	Since the empty set is vacuously $f$-perfect, we have
	$\pers{X}{f}{\varnothing}{Y} = \unis{X}{f}{Y}$ and $\perf{f}{\varnothing} = \unif{f}$. 
\end{rmk}

\begin{prop}\label{lem:subperfection_characterisation}
	Let $f \colon X \to Y$ be a continuous map between locally compact Hausdorff spaces and let $U$ be $f$-proper ($f$-perfect). Then $\iota_Y(U)$ is open in $\unis{X}{f}{Y}$ and $(\pers{X}{f}{U}{Y}, \perf{f}{U})$ is a fibrewise compactification (perfection) of $f$ with $\perf{f}{U}(\unis{X}{f}{Y} \setminus \ol{\iota_X(X)}) = \pr{f} \setminus U$.
	Moreover, every sub-fibrewise compactification (sub-perfection) of $(\unis{X}{f}{Y},\unif{f})$ is of the form $(\pers{X}{f}{U}{Y}, \perf{f}{U})$ for some $f$-proper ($f$-perfect) set $U \subseteq Y$.
\end{prop}

\begin{proof}
	Fix an $f$-proper set $U$. By \cref{lem:f-reg_characterisation}, each $y \in U$ has a precompact open neighbourhood $W_y \subseteq U$  such that $f^{-1}(\ol{W_y})$ is compact. In $\unis{X}{f}{Y}$ we have
	\[
	\iota_Y(U) = \bigcup_{y \in U} \unif{f}^{-1}(U) \cap (\unis{X}{f}{Y} \setminus f^{-1}(\ol{W_y})),
	\]
	which is open in $\unis{X}{f}{Y}$ by \cref{lem:unified_topology}. It follows that $\pers{X}{f}{U}{Y}$ is a closed subspace of $\unis{X}{f}{Y}$, and is therefore locally compact and Hausdorff. It also follows that the closure of $\iota_X(X)$ in both $\pers{X}{f}{U}{Y}$ and $\unis{X}{f}{Y}$ coincides. By \cref{prop:unified_properties}, $\unif{f}$ restricts to an homeomorphism from $\unis{X}{f}{Y} \setminus \ol{\iota_X(X)}$ to $\pr(f)$, so $\perf{f}{U}$, being a restriction of $\unif{f}$, restricts to a homeomorphism from $\pers{X}{f}{U}{Y} \setminus \ol{\iota_X(X)}$ to $\per(f) \setminus U$. So \labelcref{perf:density} holds.	
	The restriction $\perf{f}{U}$ of $\unif{f}$ is proper since the intersection of a compact set with a closed subspace is compact, and $\perf{f}{U} \circ \iota_X= f$, so \labelcref{perf:commuting} holds. 
	Hence, $(\pers{X}{f}{U}{Y}, \perf{f}{U})$ is a fibrewise compactification. 
	
	If $U$ is also $f$-perfect, then $\perf{f}{U}$ surjects onto $Y \setminus U$. By \cref{lem:boundary_issues}, $\perf{f}{U}$ is closed, so $\ol{\perf{f}{U}(X)} =\perf{f}{U}(\ol{\iota_X(X)})$. It follows that
	\[
	U \subseteq \ol{f(X)} \subseteq \ol{\perf{f}{U}(X)} =\perf{f}{U}(\ol{\iota_X(X)}),
	\]
	and so $\perf{f}{U}$ is surjective and $(\pers{X}{f}{U}{Y}, \perf{f}{U})$ is a perfection. 
	
	Now suppose that $(Z,g)$ is a sub-fibrewise compactification of $(\unis{X}{f}{Y},\unif{f})$. By \cref{lem:subperfections_closed},   $U \coloneqq (\unis{X}{f}{Y}) \setminus Z$ is an open subset of $Y$. Since $Z$ is closed \cref{lem:unified_topology} implies that for any compact set $K \subseteq Y \setminus Z =U$, the set $f^{-1}(K) \cap Z = f^{-1}(K)$ is compact in $X$. So $U$ is $f$-proper. 
	
	Suppose that $g$ is also surjective and fix $y \in U$. Then there exists $z \in Z$ such that $g(z) = y$. We must have $z \in \iota_X(X)$, for otherwise $U \cap Z \ne \varnothing$. So $f(z)  = g(z) = y$. Thus, $z \in f^{-1}(U)$, and so $U$ is $f$-perfect.  
\end{proof}

By \cref{lem:f-reg_characterisation}, $\per(f)$ is the maximal $f$-perfect open subset of $Y$. Since $f$-perfect subsets correspond to sub-perfections of the unified space, $(\pers{X}{f}{\per(f)}{Y},\perf{f}{\per(f)})$ is the unique minimal sub-perfection of the unified space.

\begin{dfn}
	Let $f \colon X \to Y$ be a continuous map between locally compact Hausdorff spaces. 
	We call $(\mins{X}{f}{Y}, \minf{f}) \coloneqq (\pers{X}{f}{\per(f)}{Y}, \perf{f}{\per(f)})$ the \emph{minimal perfection} of $f$. 
\end{dfn}

\begin{rmk}
The term ``minimal'' in minimal perfection is justified since \cref{prop:strict_fw} implies that $[\mins{X}{f}{Y}]$ is the minimal fibrewise compactification, and for each $y \in Y \setminus \ol{f(X)}$ there is precisely one element in $\minf{f}^{-1}(y)$, which is required for $\minf{f}$ to be surjective. 
\end{rmk}

If $f \colon X \to Y$ is surjective (or has dense range), then the minimal fibrewise compactification and the minimal perfection coincide. They are typically distinct though, as the following example highlights. 

\begin{example}
	Let $f \colon X \to Y$ be the map of \cref{ex:fun_pictures}. Then $\per(f) = (1,2)$. The minimal perfection $(\mins{X}{f}{Y},\minf{f})$ corresponds diagrammatically to the subspace
	
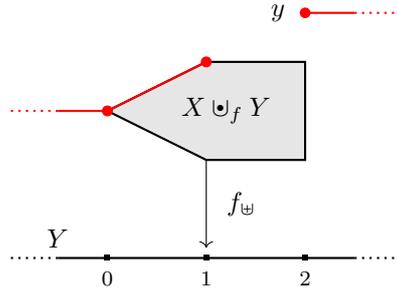
\begin{center}
		\begin{tikzpicture}[scale=1.3]
			\begin{scope}[font=\small]
				
				\node (BN) at (0,0) {};
				\coordinate (BL) at (0,0) {};
				\coordinate (CL) at (1,0);
				\coordinate (A) at (-1,0.5);
				\coordinate (AL) at (-1.5,0.5);
				\coordinate (ALE) at (-2,0.5);
				\coordinate (BU) at (0,1);

				\coordinate (CU) at (1,1);
				\coordinate (CUU) at (1,1.5);
				\coordinate (CUUR) at (1.5,1.5);
				\coordinate (CUURE) at (2,1.5);
				
				\coordinate[label={[label distance=0cm]90:$Y$}] (LL) at (-1.5,-1);
				\coordinate (LLE) at (-2,-1);
				\coordinate (LR) at (1.5,-1);
				\coordinate (LRE) at (2,-1);
				\node (LM) at (0,-1) {};

				
				\draw [thick,fill=black!10!white] (A) -- (BL) -- (CL) -- (CU) -- (BU) -- cycle;
				
				\draw [thick,red] (A) -- (AL);
				\draw [thick,dotted,red] (AL) -- (ALE);
				\draw [thick,red] (CUU)--(CUUR);
				\draw [thick,dotted,red] (CUUR)  -- (CUURE);
				\draw [thick,red] (A) -- (BU);
				
				\draw [thick] (LL) -- (LR);
				\draw [thick,dotted] (LL) -- (LLE) (LR) -- (LRE);
				
				\node (M) at (0.2,0.5) {$\mins{X}{f}{Y}$};
				
				\draw[red,fill=red]
				(A) circle (0.5mm)
				(BU) circle (0.5mm);

				\draw [->] (BL.south) -- (LM.north) node[midway, label={[label distance=0cm]0:$\unif{f}$}] {};
				
				\node[label={[label distance=0.5mm]180:$y$},circle, fill=red, inner sep=0.5mm] at (CUU) {};
				
				\node[fill,inner sep=1pt,label=270:{\scriptsize$0$}] (0) at (-1,-1) {};
				\node[fill,inner sep=1pt,label=270:{\scriptsize$1$}] (0) at (0,-1) {};
				\node[fill,inner sep=1pt,label=270:{\scriptsize$2$}] (0) at (1,-1) {};
				
			\end{scope}
			
		\end{tikzpicture}
	\end{center} 
	of the unified space. Again, the parts of $Y$ that are ``glued'' to $X$ are highlighted in red. Although the point labelled $y$ (corresponding to $2 \in Y$) is not necessary for surjectivity of $\minf{f}$, it is needed to ensure that $\minf{f}$ is proper.
\end{example}

\begin{example}
	Suppose that $f \colon X \to Y$ is continuous and proper, but not surjective. Since $f$ is proper, $\unis{X}{f}{Y} \simeq X \sqcup Y$, and since $\per(f) = \intt(\ol{f(X)})$ we have $\mins{X}{f}{Y} \simeq X \sqcup (Y \setminus \intt(\ol{f(X)}))$.
\end{example}

\subsection{Composing sub-fibrewise compactifications}
\label{sec:commutative_composition}
Let $X,Y$, and $Z$ be locally compact Hausdorff spaces and suppose that $f \colon X \to Y$ and $g \colon Y \to Z$ are continuous. There are at least three ways that the unified space construction interacts with the composition of $f$ and $g$.

 The simplest approach is to take the unified space $(\unis{X}{g \circ f}{Z}, \unif{(g \circ f)})$ of $g \circ f$. This loses information about the intermediate space $Y$, making it undesirable when we introduce regulated limits in \cref{subsec:lch_inverse_limits}.

The second approach is to first form the unified space $(\unis{Y}{g}{Z}, \unif{g})$ of $g$. Considering the induced map $\iota_Y \circ f \colon X \to \unis{Y}{g}{Z}$, we form $(\unis{X}{\iota_Y \circ f}{ (\unis{Y}{g}{Z})}, \unif{(\iota_Y \circ f)})$.
The second construction is distinct from the first: as sets $\unis{X}{g \circ f}{Z} = X \sqcup Z$, while $\unis{X}{\iota_Y \circ f}{ (\unis{Y}{g}{Z})} = X \sqcup Y \sqcup Z$. So the second construction remembers the intermediate space $Y$.
The composition $\unif{g} \circ \unif{ (\iota_Y \circ f) }$ is perfect.

The third approach is to first form the unified space $(\unis{X}{f}{Y},\unif{f})$ of $f$. Considering the composition $g \circ \unif{f} \colon \unis{X}{f}{Y} \to Z$, we form $(\unis{(\unis{X}{f}{Y})}{g \circ \unif{f}}{Z},\unif{(g \circ \unif{f})})$. As sets  $\unis{(\unis{X}{f}{Y})}{g \circ \unif{f}}{Z}=X \sqcup Y \sqcup Z$, like in the second construction.

The commuting diagram
\begin{equation}\label{eq:big_diagram}
\begin{tikzcd}[ampersand replacement=\&,cramped,column sep=35pt]
	\& {\unis{Y}{g}{Z}} \& {\unis{X}{\iota_Y \circ f}{ (\unis{Y}{g}{Z})}} \\
	Z \& Y \& X \\
	{\unis{(\unis{X}{f}{Y})}{g \circ \unif{f}}{Z}} \& {\unis{X}{f}{Y}}
	\arrow["{\unif{g}}"',curve={height=15pt}, two heads, from=1-2, to=2-1]
	\arrow["{\unif{(\iota_Y \circ f)}}"', two heads, from=1-3, to=1-2]
	\arrow["{\iota_Y}"', hook, from=2-2, to=1-2]
	\arrow["g"', from=2-2, to=2-1]
	\arrow["{\iota_X}"', hook, from=2-3, to=1-3]
	\arrow["f"', from=2-3, to=2-2]
	\arrow["{\iota_X}"',curve={height=-15pt}, hook, from=2-3, to=3-2]
	\arrow["{\unif{(g \circ \unif{f})}}", two heads, from=3-1, to=2-1]
	\arrow["{\unif{f}}"', two heads, from=3-2, to=2-2]
	\arrow["{\iota_{\unis{X}{f}{Y}}}"', hook, from=3-2, to=3-1]
\end{tikzcd}
\end{equation}
summarises the maps involved in the second and third constructions; the inclusions are open inclusions, and the surjections are perfect.

It is not immediately apparent, but the second and third constructions result in homeomorphic spaces. In particular, unified space construction satisfies the following associativity property.

\begin{lem}\label{lem:compositions_associative}
Let $f \colon X \to Y$ and $g \colon Y \to Z$ be continuous maps between locally compact Hausdorff spaces. The identity map on $X \sqcup Y \sqcup Z$ induces a homeomorphism 
	\[
	\unis{X}{\iota_Y \circ f}{ (\unis{Y}{g}{Z})} \simeq \unis{(\unis{X}{f}{Y})}{g \circ \unif{f}}{Z}.
	\] 
For all $x \in X \sqcup Y \sqcup Z$ we have
	\begin{equation}\label{eq:compositions_same}
		\unif{g} \circ \unif{ (\iota_Y \circ f)}(x)	
			=\unif{(g \circ \unif{f})} (x) =
		\begin{cases}
			g \circ f (x) & \text{if } x \in X\\
			g(x)& \text{if } x \in Y\\
			x& \text{if } x \in Z.
		\end{cases}
	\end{equation}

\end{lem}

\begin{proof}

	Let $\tau_1$ denote the topology on $	\unis{X}{\iota_Y \circ f}{ (\unis{Y}{g}{Z})}$ and let $\tau_2$ denote the topology on $\unis{(\unis{X}{f}{Y})}{g \circ \unif{f}}{Z}$. To show that these spaces are homeomorphic we appeal to the convergence characterisation of \cref{lem:unified_topology} and show that they have the same convergent nets.

	Fix a net $(x_\lambda)$ in $X \sqcup Y \sqcup Z$. If $x \in X$, then $x_\lambda \to x$ with respect to $\tau_1$ if and only if $x_\lambda \to x$ with respect to $\tau_2$, since in both cases the net must eventually
	belong to $X$ and converge in the topology of $X$.
	
	Now suppose that $y \in Y$. If $x_\lambda \to y$ with respect to $\tau_1$, then $ \unif{(\iota_Y \circ f)}(x_\lambda) \to \unif{(\iota_Y \circ f)}(y) = y$ in $\unis{Y}{g}{Z}$, and for every $K \in \kappa_X$ there exists $\lambda_K$ such that $\lambda \ge \lambda_K$ implies $x_\lambda \notin K$. Since $\unif{(\iota_Y \circ f)}(x_\lambda) \to y \in Y$, and $Y$ is open in $\unis{Y}{g}{Z}$, by passing to a subnet we may assume that $\unif{(\iota_Y \circ f)}(x_\lambda) \in Y$ for all $\lambda$. In particular, $\unif{(\iota_Y \circ f)}(x_\lambda) = \unif{f}(x_\lambda)$. It follows that $x_\lambda \to y$ in $\unis{X}{f}{Y}$. Since $\unis{X}{f}{Y} \in \tau_2$ we have $x_\lambda \to y$ with respect to $\tau_2$. 
	
	On the other hand, suppose that $x_\lambda \to y$ with respect to $\tau_2$. Since $\unis{X}{f}{Y} \in \tau_2$, the net $(x_\lambda)$ is eventually in $\unis{X}{f}{Y}$ and converges in the topology of $\unis{X}{f}{Y}$. In particular, for any $K \in \kappa_X$, there exists $\lambda_K$ such that $\lambda \ge \lambda_K$ implies $x_\lambda \notin K$ and $\unif{f}(x_\lambda) \to \unif{f}(y) = y$ in $Y$. Since $\unif{f}(y') = \unif{(\iota_Y \circ f)}(y')$ for all $y' \in X \sqcup Y$, we have $x_\lambda \to y$ in $\tau_1$. 
	
	Now suppose that $z \in Z$. If $x_\lambda \to z$ with respect to $\tau_1$, then $\unif{(\iota_Y \circ f)}(x_\lambda) \to \unif{(\iota_Y \circ f)}(z) = z$ in $\unis{Y}{g}{Z}$, and for every $K_1 \in \kappa_X$ there exists $\lambda_{K_1}$ such that $\lambda \ge \lambda_{K_1}$ implies $x_\lambda \notin K_1$. Since $\unif{(\iota_Y \circ f)}(x_\lambda) \to  z$ in $\unis{Y}{g}{Z}$ and $z \in Z$, for any compact $K_2 \in \kappa_Y$ there exists $\lambda_{K_2}$ such that $\lambda \ge \lambda_{K_2}$ implies $\unif{(\iota_Y \circ f)}(x_\lambda) \notin K_2$ and $\unif{g} \circ\unif{(\iota_Y \circ f)}(x_\lambda) \to z $ in $Z$. It follows from \labelcref{eq:compositions_same} that $\unif{(g \circ \unif{f})}(x_{\lambda}) \to z$ in $Z$.
	
	Fix a compact set $K \in \kappa_{\unis{X}{f}{Y}}$. To show that $x_\lambda \to z$ in $\tau_2$ it remains to show that we can find $\lambda_K$ such that $\lambda \ge \lambda_K$ implies $x_\lambda \notin K$. By \cref{prop:unified_properties}, $K$ is closed in $\unis{X}{f}{Y}$ and there exists $K' \in \kappa_Y$ such that $K \subseteq \unif{f}^{-1}(K')$. Set $K_2 = K'$. By the previous paragraph, there exists $\lambda_{K'}$ such that $\lambda \ge \lambda_{K'}$ implies $\unif{(\iota_Y \circ f)}(x_\lambda) \notin K'$. In particular, if $x_\lambda \in X$ then $f(x_{\lambda}) \notin K'$, and if $x_\lambda \in Y$ then $x_\lambda \notin K'$. In either case $x_\lambda \notin K$ so it suffices to let $\lambda_K \coloneqq \lambda_{K'}$. Hence, $x_\lambda \to z$ with respect to $\tau_2$. 
	
	On the other hand, suppose that $x_\lambda \to z$ with respect to $\tau_2$. Then for any compact $K \in \kappa_{\unis{X}{f}{Y}}$ there exists $\lambda_K$ such that $\lambda \ge \lambda_K$ implies $x_\lambda \notin K$, and $\unif{(g \circ \unif{f})}(x_\lambda) \to z$ in $Z$. By \labelcref{eq:compositions_same} we have $\unif{g} \circ \unif{(\iota_Y \circ f)}(x_\lambda) \to z$ in $Z$. If $K_2 \in \kappa_Y$ is compact, then  \cref{prop:unified_properties} implies that $K_2$ is also compact in $\unis{X}{f}{Y}$. Hence, there exists $\lambda_{K_2}$ such that $\lambda \ge \lambda_{K_2}$ implies $\unif{(\iota_Y \circ f)}(x_\lambda) \notin K_2$. In particular, $\unif{(\iota_Y \circ f)}(x_\lambda) \to z$ in $\unis{Y}{g}{Z}$. Now fix $K_1 \in \kappa_X$. Since $K_1 \subseteq  \unif{f}^{-1} (f(K_1))$, \cref{prop:unified_properties} implies that $K_1$ is also compact in $\unis{X}{f}{Y}$, so there exists $\lambda_{K_1}$ such that $\lambda \ge \lambda_{K_1}$ implies $x_\lambda \notin K_1$. Consequently, $x_\lambda \to z$ in $\tau_1$. 
	
	The equality \labelcref{eq:compositions_same} follows from the definitions of $\unif{g} \circ \unif{(\iota_Y \circ f)}$ and $\unif{(g \circ \unif{f})}$.
\end{proof}

Since \cref{lem:compositions_associative} implies that $	\unis{X}{\iota_Y \circ f}{ (\unis{Y}{g}{Z})} \simeq \unis{(\unis{X}{f}{Y})}{g \circ \unif{f}}{Z}$, we abuse notation and simply write $X \uplus_f Y \uplus_g Z$ for this composition. By \labelcref{eq:unified_base}, a
 base for the topology on $X \uplus_f Y \uplus_g Z$ is given by the collection 
\begin{align} \label{eq:level_two_unified}
	\begin{split}
		\big\{\big(U_X \cap \unif{(\iota_Y \circ f)}^{-1}(U_Y) &\cap   \unif{(\iota_Y \circ f)}^{-1} \circ \unif{g}^{-1} (U_Z)\big) \setminus \big(K_X \cup \unif{(\iota_Y \circ f)}^{-1}(K_Y)\big) \bigm| \\
		& U_X \in \tau_X,\, U_Y \in \tau_Y, U_Z \in \tau_Z, K_X \in \kappa_X,\, K_Y \in \kappa_Y\big\}.
	\end{split}
\end{align}

We can also compose sub-fibrewise compactifications of the unified space. 

\begin{cor}\label{lem:compositions_associative_sub}
	Let $f \colon X \to Y$ and $g \colon Y \to Z$ be continuous maps between locally compact Hausdorff spaces. Let $U \subseteq Y$ be $f$-proper ($f$-perfect) and let $V \subseteq Y$ be $g$-proper ($g$-perfect). Then $V$ is $(g \circ \perf{f}{U})$-proper ($(g \circ \perf{f}{U})$-perfect),  $\iota_{Y}(U)$ is $(\iota_Y \circ f)$-proper ($(\iota_Y \circ f)$-perfect). As sets
	\[
	\pers{(\pers{X}{f}{U}{Y})}{g \circ \perf{f}{U}}{V}{Z} = X \sqcup (Y \setminus U) \sqcup (Z \setminus V) = \pers{X}{\iota_Y \circ f}{\iota_Y(U)}{(\pers{Y}{g}{V}{Z})},
	\]
	and the identity map induces a homeomorphism $\pers{(\pers{X}{f}{U}{Y})}{g \circ \perf{f}{U}}{V}{Z} \simeq \pers{X}{\iota_Y \circ f}{\iota_Y(U)}{(\pers{Y}{g}{V}{Z})}$. 
	
	If $U' \subseteq U$ and $V' \subseteq V$ are open, then $\pers{(\pers{X}{f}{U}{Y})}{g \circ \perf{f}{U}}{V}{Z}$ can be identified with the closed subset $X \sqcup (Y \setminus U) \sqcup (Z \setminus V)$ of $\pers{(\pers{X}{f}{U'}{Y})}{g \circ \perf{f}{U}}{V'}{Z}$.
\end{cor}

\begin{proof}
	First suppose that $U$ is $f$-proper and $V$ is $g$-proper. 
	Since $U$ is $f$-proper, $\perf{f}{U}$ is proper. So, since $V$ is $g$-proper, it is also $(g \circ \perf{f}{U})$-proper. The disjoint union description of $\pers{(\pers{X}{f}{U}{Y})}{g \circ \perf{f}{U}}{V}{Z}$ follows from \cref{dfn:other_perfections}.
	Since $\iota_Y\colon Y \to \pers{Y}{g}{V}{Z}$ is a homeomorphism onto its image, $(\iota_Y \circ f)^{-1}(\iota_Y(U)) = f^{-1}(U)$. It follows that  $\iota_{Y}(U)$ is $(\iota_Y \circ f)$-proper. The disjoint union description of $\pers{X}{\iota_Y \circ f}{\iota_Y(U)}{(\pers{Y}{g}{V}{Z})}$ again follows from \cref{dfn:other_perfections}.	
	The homeomorphism follows from \cref{lem:compositions_associative} since we have sub-fibrewise compactifications of the unified space. The final statement follows from \cref{lem:subperfection_characterisation} and \cref{lem:subperfections_closed}.
	
	The proof is identical, after making the necessary changes, if $U$ is $f$-perfect and $V$ is $g$-perfect. 
\end{proof}

We write $\pers{X}{f}{U}{\pers{Y}{g}{V}{Z}}$ for the space $ \pers{(\pers{X}{f}{U}{Y})}{g \circ \perf{f}{U}}{V}{Z} \simeq \pers{X}{\iota_Y \circ f}{\iota_Y(U)}{(\pers{Y}{g}{V}{Z})}$. By \cref{lem:compositions_associative_sub} and \cref{rmk:empty_set}, there is a sequence of closed inclusions
\[
\pers{X}{f}{\pr(f)}{\pers{Y}{g}{\pr(g)}{Z}} \subseteq \pers{X}{f}{\per(f)}{\pers{Y}{g}{\per(g)}{Z}} \subseteq\\
X \uplus_f Y \uplus_g Z.
\]

\begin{rmk}
	Despite the unified space (and its sub-fibrewise compactifications) displaying an associativity-like property, it is not clear to the author whether any of the constructions can be made functorial.
	An interesting observation is that the unified space (and other sub-fibrewise compactifications) can be used to define a category whose objects are locally compact Hausdorff spaces that differ from the usual category with continuous maps. 
	
	Let $X$ and $Y$ be locally compact Hausdorff spaces. Suppose that $(i_A, A, p_A)$ is a triple consisting of a locally compact Hausdorff space $A$, an open inclusion $i_A \colon X \to A$, and a continuous proper map $p_A \colon A \to Y$. We say that another such triple $(i_{A'},A',p_{A'})$ is isomorphic to $(i_A, A, p_A)$ if there is a homeomorphism $h \colon A \to A'$ such that $i_A \circ h = i_{A'}$ and $p_A \circ h = p_{A'}$.
	
	Let $\LCH$ be the category of locally compact Hausdorff spaces with continuous maps. 
	We define a new category $\unif{\LCH}$ 	
	whose objects are also locally compact Hausdorff spaces.	
	A morphism $X \to Y$ in $\unif{\LCH}$ is (an isomorphism class of) a triple $(i_A, A, p_A)$, which we denote by $[i_A, A, p_A]$.
	Composition of morphisms is defined using the unified space construction: if $[i_A, A, p_A] \colon X \to Y$ and $[i_B, B, p_B] \colon Y \to Z$ are morphisms in $\unif{\LCH}$, then $[i_B, B, p_B] \circ	[i_A, A, p_A]$
	is defined to be 
$[\iota_A \circ i_A, \unis{A}{i_B \circ p_A}{B}, p_B \circ \unif{(i_B \circ p_A)}]
	$. By \cref{lem:compositions_associative}, this composition is associative, so $\unif{\LCH}$ is a category. 
	
	The map taking a morphism $f \colon X \to Y$ in $\LCH$ to $[\iota_X, \unis{X}{f}{Y}, \unif{f}]$ in $\unif{\LCH}$ is not functorial since if $g \colon Y \to Z$ is another continuous map then typically $\unis{X}{g \circ f}{Z} \not\cong \unis{X}{f}{\unis{Y}{g}{Z}}$.
\end{rmk}

\subsection{Locally compactifying inverse limits}
\label{subsec:lch_inverse_limits}

Let $(X_i,f_i)_{i \in \NN}$ be an \emph{inverse sequence} of topological spaces $X_i$ and continuous maps $f_i \colon X_{i+1} \to X_i$. Its \emph{inverse limit} $\varprojlim(X_i,f_i)$ is the unique space (up to homeomorphism) with continuous maps $\pi_k \colon \varprojlim(X_i,f_i) \to X_k$ satisfying $\pi_k  \circ f_{k} = \pi_{k+1} $ for all $k \in \NN$, that is universal: if $Z$ is another topological space with continuous maps $g_k \colon Z \to X_k$ satisfying $g_k \circ f_k = g_{k+1}$ for all $k \in \NN$, then there is a unique continuous map $g \colon Z \to \varprojlim(X_i,f_i)$ such that $f_k \circ g = g_k$ for all $k \in \NN$.
  
By \cite[Proposition~2.5.1]{Eng89}, the inverse limit $\varprojlim(X_i,f_i)$ may be described explicitly as the closed subspace $\{(x_i)_{i \in \NN} \mid x_i \in X_i,\, f_{i}(x_{i+1}) = x_i\}$ of the product $\prod_{i \in \NN} X_i$ . The maps $\pi_k \colon \varprojlim (X_i,f_i) \to X_k$ are the coordinate projections, and the collection $\{\pi_k^{-1}(U) \colon k \in \NN,\, U \in \tau_{X_k}\}$ constitutes a base for the topology on $\varprojlim(X_i,f_i)$.
We record the following well-known result. 

\begin{lem}\label{lem:injectivity_limits}
	Let $(X_i,f_i)_{i \in \NN}$ and $(Y_i,g_i)_{i \in \NN}$ be inverse sequences of topological spaces and continuous maps. Let $\pi_k \colon \varprojlim (X_i,f_i) \to X_k$ and $\tau_k \colon \varprojlim(Y_i,g_i) \to Y_k$ be the universal projections. If for each $k \in \NN$, there are continuous maps $\alpha_k \colon X_k \to Y_k$ satisfying $g_k \circ \alpha_{k+1} = \alpha_{k} \circ f_k$, then there is a unique continuous map $\alpha \colon \varprojlim (X_i,f_i) \to \varprojlim(Y_i,g_i)$ satisfying $\alpha_k \circ \pi_k = \tau_k \circ \alpha$ for all $k \in \NN$.  Moreover, if each $\alpha_k$ is injective, then so is $\alpha$. 
\end{lem}

\begin{proof}
	The first statement follows directly from the universal property of $\varprojlim(Y_i,g_i)$ applied to the maps $\alpha_i \circ \pi_i \colon \varprojlim(X_i,g_i) \to Y_i$. To see that $\alpha$ is injective, suppose that $\alpha(x) = \alpha(x')$ for some $x,x'\in \varprojlim(X_i,f_i)$. Then $\alpha_k \circ \pi_k (x) = \tau_k \circ \alpha(x) = \tau_k \circ \alpha(x') = \alpha_k \circ \pi_k (x')$ for all $k \in \NN$. Since each $\alpha_{k}$ is injective, $\pi_k(x) = \pi_{k}(x')$ for all $k \in \NN$. Since the $\pi_k$ are coordinate projections in $\prod_{i \in \NN}X_i$, we have $x = x'$.  
\end{proof}

If each $X_i$ is locally compact and Hausdorff, then $\varprojlim(X_i,f_i)$ is Hausdorff, but it is not locally compact unless the maps $f_i$ are proper \cite[Theorem 3.7.13]{Eng89}. 
Using the unified space construction we produce a new inverse limit $\varprojlim(\wt{X}_i,\wt{f}_i)$ in which the maps $\wt{f}_i$ are perfect extensions of the $f_i$.

\begin{dfn}
	Let $(X_i,f_i)_{i \in \NN}$ be an inverse sequence of locally compact Hausdorff spaces with continuous maps $f_i \colon X_{i+1} \to X_{i}$. 
	Let $\wt{X}_0 \coloneqq X_0$. For each $i \ge 1$ inductively define
	\[
	\wt{X}_{i} \coloneqq \unis{X_i}{\iota_{X_{i-1}} \circ f_{i-1}}{\wt{X}_{i-1}}
	\]
	and let $\wt{f}_i \colon \wt{X}_{i+1} \to \wt{X}_{i}$ be the perfect map given by
	\[
	\wt{f}_i (x) = \unif{(\iota_{X_i} \circ f_i)}(x) =  \begin{cases}
		f_{i}(x) & \text{if } x \in X_{i+1},\\
		x & \text{if } x \in \wt{X}_{i}.
	\end{cases}
	\]
	We call the space $\unilim (X_i,f_i) \coloneqq \varprojlim (\wt{X}_{i},\wt{f}_i)$ the \emph{unified limit} of $(X_i,f_i)_{i \in \NN}$.
\end{dfn}

As sets we have $\wt{X}_{i} = \bigsqcup_{k=0}^i X_k$, however the topology on $\wt{X}_i$ is quite nontrivial. We extend the notation established in \cref{sec:commutative_composition} and write 
\[
 \wt{X}_{i} \eqqcolon \unis{X_{i}}{f_{i-1}}{
\unis{X_{i-1}}{f_{i-2}}{\unis{\cdots}{f_0}{X_0}}}.
\]
The next result follows inductively from \cref{lem:compositions_associative}.

\begin{lem}\label{lem:composition_big}
	Let $(X_i,f_i)_{i \in \NN}$ be an inverse sequence of locally compact Hausdorff spaces and continuous maps. Fix $p,q \in \NN$. Let $f_{p,q} \colon \unis{X_{p+q}}{f_{p+q-1}}{ \unis{\cdots}{f_{q+1}}{X_{q+1}}} \to  \wt{X}_q$ be the continuous map given by
	\[
	f_{p,q}(x) = \begin{cases}
		f_q \circ \cdots \circ  f_{i-1}(x) & \text{if } x \in X_i \text{ and } i > q,\\
		x & \text{if } x \in \wt{X}_q.
	\end{cases}
	\]
	Then $\wt{X}_{p+q} \simeq \unis{(\unis{X_{p+q}}{f_{p+q-1}}{ \unis{\cdots}{f_{q+1}}{X_{q+1}}})}{f_{p,q}}{\wt{X}_q}$. In particular, $\bigsqcup_{i=q+1}^{p+q} X_i$ is open in $\wt{X}_{p+q}$, and $ \wt{X}_q = \bigsqcup_{i=0}^{q} X_i$ is closed in $\wt{X}_{p+q}$.
\end{lem}

 The unified limit $\unilim(X_i,f_i)$ can be thought of as a ``local compactification'' of the limit $\varprojlim(X_i,f_i)$.

\begin{prop}\label{prop:unified_limit}
	Let $(X_i,f_i)_{i \in \NN}$ be an inverse sequence of locally compact Hausdorff spaces and continuous maps. The unified limit $\unilim({X}_{i},{f}_i)$ is a locally compact Hausdorff space. Moreover,
	\begin{enumerate}
		\item the universal projections $\wt{\pi}_k \colon \unilim({X}_i,{f}_i) \to \wt{X}_k$ are perfect;
		\item\label{itm:base_unifiedlim}  the sets
		\begin{align*}
			\Zz(U_0,\ldots,U_m; K_0, \ldots, K_n) \coloneqq \big(\wt{\pi}_0^{-1}(U_0) &\cap \cdots \cap \wt{\pi}_m^{-1}(U_m)\big)\\ 
			&\setminus  \big(\wt{\pi}_0^{-1}(K_0) \cup \cdots \cup \wt{\pi}_n^{-1}(K_n)\big)
		\end{align*}
		indexed by $m ,n \in \NN$ with $m \ge n$,  $(U_i)_{i=0}^m$ with $U_i \in \tau_{X_i}$, and $(K_j)_{j=0}^{n}$ with $K_j \in \kappa_{X_j}$, constitute a base for the topology on $\unilim({X}_{i},{f}_i)$;
	
		\item 
		the inclusions $\iota_{X_k} \colon X_k \hookrightarrow \wt{X}_k$ for each $k \in \NN$ induce a continuous inclusion $\varprojlim(X_i,f_i) \hookrightarrow \varprojlim^{\Uu}(X_{i},f_i)$; and
		\item
		 if each $X_i$ is second countable, then $\unilim ({X}_{i},{f}_i)$ is metrisable.
		\end{enumerate}
\end{prop}

\begin{proof}

By construction, $(\wt{X}_i, \wt{f}_i)_{i \in \NN}$ is an inverse sequence of locally compact Hausdorff spaces with perfect maps. By \cite[Theorem 3.7.13]{Eng89}, the limit $\varprojlim (\wt{X}_{i},\wt{f}_i)$ is a locally compact Hausdorff space in the initial topology induced by the $\wt{\pi}_k$, and the $\wt{\pi}_k$ are proper. The $\wt{\pi}_k$ are surjective since each $\wt{f}_k$ is surjective. 
That the sets in part \labelcref{itm:base_unifiedlim} form a base follows inductively from the base \labelcref{eq:level_two_unified} and the definition of the initial topology induced by the $\wt{\pi}_k$. 
By \cref{lem:injectivity_limits}, the inclusions $\iota_{X_k} \colon X_k
 \hookrightarrow \wt{X}_k$ induce a continuous inclusion $\iota \colon \varprojlim(X_i,f_i) \hookrightarrow \unilim(\wt{X}_i, \wt{f}_i)$ such that $\wt{\pi}_k \circ \iota = \iota_{X_k} \circ \pi_k$ for all $k \in \NN$. 
If each $X_i$ is second countable, then \cref{prop:unified_properties} implies that each $\wt{X}_i$ is second countable. The countable inverse limit of second-countable spaces is second countable. Since $\unilim(X_i,f_i)$ is second countable, locally compact, and Hausdorff, it is metrisable.   
\end{proof}

Since $\wt{X}_i = \bigsqcup_{k=0}^{i} X_k$ and $\wt{f}_i|_{\wt{X}_i} = \id_{\wt{X}_i}$ for each $i \in \NN$, we may identify $\unilim ({X}_{i},{f}_i)$ with $\varprojlim (X_{i},f_i) \sqcup \bigsqcup_{i=0}^{\infty} X_i$. If $\pi_i \colon \varprojlim (X_i,f_i) \to X_i$ denote the universal projections, then, under this identification, the universal projections $\wt{\pi}_i \colon \varprojlim (X_{i},f_i) \sqcup \bigsqcup_{i=0}^{\infty} X_i \to \wt{X}_i$ satisfy
\begin{equation}\label{eq:pi_tilde}
	\wt{\pi}_i (x) 
	= \begin{cases}
		\pi_i(x) & \text{if } x \in \varprojlim(X_i,f_i),\\
		f_{i} \circ \cdots \circ  f_{j-1}(x) & \text{if } x \in X_j \text{ and } j > i,\\
		x & \text{if } x \in X_j \text{ and } j \le i.
	\end{cases}
\end{equation}

Although \cref{prop:unified_limit} implies that $\varprojlim(X_i,f_i)$ may be identified as a subspace of $\unilim (X_i,f_i)$, it is typically neither open nor closed for otherwise $\varprojlim(X_i,f_i)$ would be locally compact. The basic open sets $\Zz(U_0,\ldots,U_m; K_0, \ldots, K_n)$ contain elements from both $\varprojlim(X_i,f_i)$ and $\bigsqcup_{i=0}^{\infty} X_i$. Limit points of $\varprojlim(X_i,f_i)$, as a subspace of $\unilim(X_i,f_i)$, can belong to $\bigsqcup_{i=0}^{\infty} X_i$.

One application of the unified limit is in extending a countably infinite product of locally compact spaces (which is typically not locally compact) to a locally compact space. 

\begin{example}	\label{ex:product}
	Let $(X_i)_{i \in \NN}$ be a sequence of locally compact Hausdorff spaces.
	Let $p_m \colon \prod_{i=0}^{m+1} X_i \to \prod_{i=0}^{m} X_{i}$ be the projection onto the first $m$ components. Then $\prod_{i=0}^{\infty} X_i \simeq \varprojlim ( \prod_{i=0}^{m} X_{i},p_m)$. By \cref{prop:unified_limit}, the space $\varprojlim^{\uplus}  ( \prod_{i=0}^{m} X_{i},p_m)$ is locally compact and contains $\prod_{i=0}^{\infty} X_i$ as a subspace.
\end{example}

In the next example looks at a more specific case of \cref{ex:product} to highlight the nature of convergence in the unified limit. 

\begin{example}\label{ex:integers_product}
	Suppose that $X_i = \ZZ$ for all $i \in \NN$. Then $\ZZ^{\infty} \coloneqq \prod_{i =0}^{\infty} \ZZ$ is the space of all infinite sequences of integers. Let $\pi_k \colon \ZZ^{\infty} \to \prod_{i=0}^k \ZZ$ be the projection onto the first $k$ terms. The family  $\{\pi_k^{-1}(n) \mid k \in \NN,\, n \in \ZZ^k\}$ constitutes a base for the topology on $\ZZ^{\infty}$. For $k \in \NN$ and $n = (n_1,\ldots,n_k)$, the basic open set $\pi_k^{-1}(n)$ consists of all infinite sequences that start with the finite sequence $n$. 
	If $a \in \ZZ^{\infty}$ we write $a_i$ for its $i$-th term. A sequence $(a^j)_{j \in \NN}$ of sequences $a^j \in \ZZ^{\infty}$ converges to $a \in \ZZ^{\infty}$ if, for every $I \in \NN$, there exists $J \in \NN$ such that $a_i = a_i^j$ for all $0 \le i \le I$ and $j \ge J$. 
	
	As a set, we identify the unified limit $\varprojlim^{\uplus}  ( \prod_{i=0}^{m} \ZZ,p_m)$ with $\ZZ^{\infty} \sqcup \bigsqcup_{k \in \NN} \ZZ^k$, the collection of all finite and infinite sequences in $\ZZ$. Let $k \in \NN$, $n \in \ZZ^k$, and fix a finite set $F \subseteq \bigsqcup_{i \in \NN} \ZZ^i$ of finite sequences. Let
	\[
	\Zz(n,F) \coloneqq \Big\{ a \in \ZZ^{\infty} \sqcup \bigsqcup_{i \in \NN} \ZZ^i \mid (a_1,\ldots,a_k) = n \text{ and } (a_1,\ldots,a_i) \notin F \text{ for all } i \in \NN\Big\}.
	\]
	 Adapting \cref{prop:unified_limit}~\labelcref{itm:base_unifiedlim} to this setting, a base for the topology on $\varprojlim^{\uplus}  ( \prod_{i=0}^{m} \ZZ,p_m)$ is given by the sets $\Zz(n,F)$, ranging over all finite sequences $n$ and finite sets $F \subseteq \bigsqcup_{k \in \NN} \ZZ^k$.
	 
	Fix a finite sequence $d = (d_1,\ldots,d_l) \in \ZZ^{l}$. For each $j \ge 1$, let $b^j = (d_1,\ldots,d_l,j,j,\cdots) \in \ZZ^{\infty}$. The sequence $(b^j)_{j \in \NN}$ does not converge in $\ZZ^{\infty}$. We show, however, that $b^j \to d$ in $\varprojlim^{\uplus}   ( \prod_{i=0}^{m} \ZZ,p_m)$. Fix $k \in \NN$, $n \in \ZZ^k$ and a finite set $F$ of finite sequences such that $d \in \Zz(n,F)$. That is, $(d_1,\ldots,d_k) = n$ (so $k \le l$) and no initial segment of $d$ belongs to $F$. By definition, $(b^j_1,\ldots,b^j_k) = n$, and since $F$ is finite, for large $j$ we have $b^j \not\in F$. Thus, $b^j \to d$ in $\varprojlim^{\uplus}   ( \prod_{i=0}^{m} \ZZ,p_m)$. We have found a sequence of infinite sequences in $\ZZ^{\infty}$ that converges to a finite sequence in the unified limit. 
\end{example}

For the next example we require the following definition. 

\begin{dfn}[{\cite{Kat04}}]
	A \emph{topological graph} $E = (E^0,E^1,r,s)$ consists of locally compact Hausdorff spaces $E^0$ and $E^1$ of \emph{vertices} and  \emph{edges}, together with a continuous map $r \colon E^1 \to E^0$ called the \emph{range}, and a local homeomorphism $s \colon E^1 \to E^0$ called the \emph{source}. 
\end{dfn}

Topological graphs were introduced by Katsura~\cite{Kat04} as a class of topological dynamical systems that generalise directed graphs. Indeed, directed graphs correspond to topological graphs where $E^0$ and $E^1$ are countable and discrete. 

\begin{example}\label{ex:top_graph_1}
	Let $E = (E^0,E^1,r,s)$  be a topological graph.
	For each $n \in \NN$, the space of \emph{paths of length $n$} in $E$ is 
	\[
	E^n \coloneq \{e_1e_2\cdots e_n \mid e_i \in E^1, \, s(e_i) = r(e_{i+1}) \text{ for all } 1 \le i < n\} \subseteq \prod_{i=1}^n E^1
	\]
	with subspace topology inherited from the product. Each $E^n$ is closed in $\prod_{i=1}^n E^1$, and is therefore locally compact and Hausdorff. 
	For each $i \ge 1$, define $r_i \colon E^{i+1} \to E^{i}$ by $r_i(e_1\cdots e_{i+1}) = e_1\cdots e_{i}$ and $r_0 \coloneqq r$. The  \emph{(one-sided) infinite path space} of $E$ is the inverse limit 
	\[
	E^{\infty} \coloneqq \varprojlim (E^i,r_i) =  \{e_1e_2 e_3 \cdots \mid e_i \in E^1,\, s(e_i) = r(e_{i+1}) \text{ for all } i \in \NN \} \subseteq \prod_{i=1}^{\infty} E^1
	\]
	with the inverse limit topology. 
			
	The range map $r \colon E^1 \to E^0$ is typically not proper nor surjective. As such, the infinite path space $E^{\infty}$ is typically not locally compact. 
	By \cref{prop:unified_limit}, the unified limit $\varprojlim^{\uplus} (E^{i},r_i)$ is a locally compact Hausdorff space containing $E^\infty$ as a subspace. 
	The new sequence $(\wt{E}^i,\wt{r}_i)_{i \in \NN}$ is given explicitly by
	\[
	\wt{E}^i = \bigsqcup_{k=0}^i E^k
	\qquad \text{and} \qquad
	\wt{r}_i (x) =
	\begin{cases}
		r_i(x) & \text{if } x \in E^{i+1},\\
		x & \text{if } x \in \wt{E}^{i}.
	\end{cases}
	\]
	The unified limit is
	\[
	E^{\le \infty} \coloneqq {\textstyle\varprojlim^{\uplus}} (E^{i},r_i) =  E^{\infty} \sqcup \bigsqcup_{i=1}^{\infty} E^i,
	\]
	the collection of all finite and infinite paths in $E$. By  \cref{prop:unified_limit}~\labelcref{itm:base_unifiedlim}, the topology on $E^{\le \infty}$ is generated by the basic open sets 
	\[
	\big\{\Zz(U_0,\ldots,U_m; K_0, \ldots, K_n)  \bigm| U_i \in \tau_{E^i} , K_j \in \kappa_{E^k}, m \ge n \in \NN\big\}.
	\]
	
	The topology on $E^{\le \infty}$ agrees with the patch topology of \cite[Proposition~3.8]{deCa21}, which in turn agrees with the topologies on $E^{\le \infty}$ previously considered in \cite{Exe08,PW05,Web14,Yen07} (see \cite{deCa21}).
	 Our description complements these previous analyses of $E^{\le \infty}$. Importantly, our construction of $E^{\le \infty}$ only relies on elementary topological tools. It also helps to explain \emph{why} the topology on $E^{ \le \infty}$ is what it is: it ensures that the maps appearing in the inverse limit decomposition of $E^{\le \infty}$ are perfect.
\end{example}

By considering sub-fibrewise compactifications of the unified space, we can also construct smaller extensions of $\varprojlim(X_i,f_i)$. To describe these extensions, we introduce the following terminology. 
\begin{dfn}\label{dfn:regulated_limit} Let $(X_i,f_i)_{i \in \NN}$ be an inverse sequence of locally compact Hausdorff spaces and continuous maps. 
	A \emph{regulating sequence} for $(X_i,f_i)_{i \in \NN}$ is a collection $\Uu = (U_i)_{i \in \NN}$ of open sets $U_i \subseteq X_i$, such that each $U_i$ is $f_{i}$-proper. If each $U_i$ is $f_{i}$-perfect, then we call $\Uu$ a \emph{perfect} regulating sequence. 
	
	Given a regulating sequence $\Uu$ for $(X_i,f_i)_{i \in \NN}$ let $X^{\Uu}_0 \coloneqq X_0$. For $i \ge 1$ inductively define
	\[
	X^\Uu_i = \pers{X_{i}}{\iota_{X_{i-1}} \circ f_{i-1}}{U_{i-1}}{X^{\Uu}_{i-1}}.
	\]
	Let $f^{\Uu}_i \colon X^{\Uu}_{i+1} \to X^{\Uu}_i$ be the perfect map given by 
	\[
	f^{\Uu}_i(x) = \perf{(\iota_{X_i} \circ f_i)}{U_i} (x)= \begin{cases}
		f(x) & \text{if } x \in X_{i+1}\\
		x & \text{if } x \in X^\Uu_i\setminus U_i.
	\end{cases}
	\]
	 We call the space $\varprojlim^{\Uu}(X_i,f_i) \coloneqq \varprojlim(X^{\Uu}_{i},f^{\Uu}_i)$ the \emph{$\Uu$-regulated limit} of $(X_i,f_i)_{i \in \NN}$.
\end{dfn}

\begin{rmk}
	If $\Uu = (\varnothing)_{i \in \NN}$, then $\varprojlim^{\Uu}(X_i,f_i) = \unilim(X_i,f_i)$.
\end{rmk}

As sets we have $X^{\Uu}_{i} = X_i \sqcup \bigsqcup_{k=0}^{i-1} (X_k \setminus U_k)$. We extend the notation established in \cref{sec:commutative_composition} and write
\[
X^{\Uu}_{i} \eqqcolon \pers{X_{i}}{f_{i-1}}{U_{i-1}}{
			\pers{X_{i-1}}{f_{i-2}}{U_{i-2}}{\pers{\cdots}{f_0}{U_0}{X_0}}}.
\]
We have the following generalisation of \cref{prop:unified_limit}.

\begin{cor}\label{cor:regulated_limit}
		Let $(X_i,f_i)_{i \in \NN}$ be an inverse sequence of locally compact Hausdorff spaces and continuous maps, and let $\Uu = (U_i)_{i \in \NN}$ be a regulating sequence for $(X_i,f_i)_{i \in \NN}$. 
		\begin{enumerate}
			\item\label{itm:ureg_lch} The $\Uu$-regulated limit $\varinjlim^{\Uu}({X}_{i},{f}_i)$ is a locally compact Hausdorff space.
			
			\item\label{itm:genlim_prop} The universal projections $\pi^{\Uu}_k \colon  \varprojlim^{\Uu}(X_i,f_i) \to X^{\Uu}_{k}$ are proper, and perfect if $\Uu$ is perfect.

			\item\label{itm:genlim_inclusion} 
			The inclusions $\iota_{X_k} \colon X_k \hookrightarrow X^{\Uu}_k$ for each $k \in \NN$ induce a continuous inclusion $\varprojlim(X_i,f_i) \hookrightarrow \varprojlim^{\Uu}(X_{i},f_i)$.
			\item\label{itm:genlim_ctble}
			If each $X_i$ is second countable, then $\varprojlim^{\Uu} ({X}_{i},{f}_i)$ is metrisable.
			
			\item\label{itm:genlim_sub} If $\Vv = (V_i)_{i \in \NN}$ is another regulating sequence for $(X_i,f_i)_{i \in \NN}$ such that $U_i \subseteq V_i$ for all $i \in \NN$, then $\varprojlim^{\Vv}(X_i,f_i)$ can be identified with the closed subset $\varprojlim(X_i,f_i) \sqcup \bigsqcup_{i=0}^{\infty} (X_i \setminus V_i)$ of $\varprojlim^{\Uu}(X_i,f_i)$.  In particular, $\varprojlim^\Uu(X_i,f_i)$ can be identified with the closed subset $\varprojlim(X_i,f_i) \sqcup \bigsqcup_{i=0}^{\infty} (X_i \setminus U_i)$ of $\unilim(X_i,f_i)$. 
		\end{enumerate}
\end{cor}

\begin{proof}
	Items \labelcref{itm:ureg_lch,itm:genlim_prop,itm:genlim_inclusion,itm:genlim_ctble} follow from arguments analogous to those in the proof of \cref{prop:unified_limit}.
	
 For \labelcref{itm:genlim_sub}, let $\Uu = (U_i)_{i \in \NN}$ and $\Vv = (V_i)_{i \in \NN}$ be regulating sequences for $(X_i,f_i)_{i \in \NN}$ with $U_i \subseteq V_i$. By \cref{lem:composition_big} and
 \cref{lem:compositions_associative_sub}, for each $i \in \NN$  the space $X_{i}^{\Vv} = X_i \sqcup \bigsqcup_{k=0}^{i-1} (X_k \setminus V_i)$ is closed in $X_{i}^{\Uu} = X_i \sqcup \bigsqcup_{k=0}^{i-1} (X_k \setminus U_i)$ which is, in turn, closed in $\wt{X}_{i}= \bigsqcup_{k=0}^{i}X_k$. For each $i \in \NN$ let,  $\alpha_i \colon X_{i}^{\Vv} \hookrightarrow  X_{i}^{\Uu}$ be the corresponding closed inclusion. Then $ f_{i}^{\Vv} \circ \alpha_{i+1} = \alpha_{i} \circ f_{i}^{\Uu}$.
 So, \cref{lem:injectivity_limits} gives a continuous inclusion $\lambda \colon \varprojlim^{\Vv}(X_i,f_i) \hookrightarrow \varprojlim^{\Uu}(X_i,f_i)$ such that $\alpha_i \circ \pi^{\Vv}_i = \pi^{\Uu}_i  \circ \alpha$ for all $i \in \NN$. 

	To see that $\alpha$ is closed, fix a net $(x_{\lambda})$ in $\varprojlim^{\Vv}(X_i,f_i)$ such that $\alpha(x_{\lambda}) \to x$ for some $x \in  \varprojlim^{\Uu}(X_i,f_i)$. Fix $k \in \NN$. Then $\alpha_k \circ \pi^{\Vv}_k (x_{\lambda}) = \pi^{\Uu}_k \circ \alpha (x_{\lambda}) \to \pi^{\Uu}_k (x)$. Since each $\alpha_k$ is a closed map, $\alpha_k(X_k^{\Vv})$ is closed. So $\pi^{\Uu}_k(x) = \alpha_k(y_k)$ for some $y_k \in X^{\Vv}_k$. Injectivity of $\alpha_k$ implies that $\pi^{\Vv}_k(x_{\lambda}) \to y_k$. Since each $X_k^{\Vv}$ is Hausdorff, continuity of $f_{k}^{\Vv}$ and the fact that $f_k^{\Vv} \circ \pi_{k+1}^{\Vv} = \pi_k^{\Vv}$  gives $f_{k}^{\Vv}(y_{k+1}) = y_k$ for all $k \in \NN$. That is, $y \coloneqq (y_0,y_1,\ldots) \in \prod_{i=0}^{\infty} X^{\Vv}_i$ belongs to $\varprojlim(X^{\Vv}_{i},f^{\Vv}_i)$, and $\alpha(y) = x$. So $\alpha$ is closed. 

	Taking $\Uu = (\varnothing)_{i \in \NN}$, $\Vv = (U_i)_{i \in \NN}$, and applying \cref{lem:compositions_associative_sub}
	gives the identification of $\varprojlim^{\Uu}(X_i,f_i)$ with $\varprojlim(X_i,f_i) \sqcup \bigsqcup_{i=0}^{\infty} (X_i \setminus U_i)$ in $\unilim(X_i,f_i)$. 
\end{proof}

\begin{dfn}
	Let $(X_i,f_i)_{i \in \NN}$ be an inverse sequence of locally compact Hausdorff spaces.
	\begin{enumerate}
		\item If $\Uu = (\pr(f_i))_{i \in \NN}$, then we call $\minlim(X_i,f_i) \coloneqq \varprojlim^{\Uu}(X_i,f_i)$ the \emph{minimal regulated limit} of $(X_i,f_i)_{i \in \NN}$.
		\item If $\Uu = (\per(f_i))_{i \in \NN}$, then we call $\perlim(X_i,f_i) \coloneqq \varprojlim^{\Uu}(X_i,f_i)$ the \emph{minimal perfect regulated limit} of $(X_i,f_i)_{i \in \NN}$.
	\end{enumerate}
\end{dfn}

By \cref{cor:regulated_limit}, there is a sequence of inclusions
\[
\varprojlim(X_i,f_i) 
\subseteq \minlim(X_i,f_i)
\subseteq \perlim(X_i,f_i)
\subseteq \unilim (X_i,f_i)
\]
with the final two inclusions being closed. 

\begin{example}
	In the setup of \cref{ex:product} we had locally compact Hausdorff spaces $X_i$ and projections $p_m \colon \prod_{i=0}^{m+1} X_i \to \prod_{i=0}^{m} X_{i}$. Each projection $p_m$ is surjective, and is proper if and only if $X_{m+1}$ is compact. In fact, we have
	\[
	\per(p_m) = \pr(p_m) = \begin{cases}
		\prod_{i=0}^{m} X_{i} & \text{if } X_{m+1} \text{ is compact},\\
		\varnothing & \text{if } X_{m+1} \text{ is not compact}.
	\end{cases}
	\]
	So $\minlim(X_i,f_i)
	= \perlim(X_i,f_i)$. These limits are equal to $\varprojlim(X_i,f_i)$ if every $X_i$ is compact, and equal to $\unilim(X_i,f_i)$ if every $X_i$ is not compact (as in \cref{ex:integers_product}). 
\end{example}

\begin{example}\label{ex:top_graph_2}
	Let $E = (E^0,E^1,r,s)$ be a topological graph as in \cref{ex:top_graph_1}. There are two important subsets of vertices in the analysis of topological graphs (see \cite[Definition~2.6]{Kat04}). These are the set of \emph{finite receivers}
	\begin{align*}
		E^0_{\mathrm{fin}} \coloneqq \{v \in E^0 &\mid v \text{ has precompact open neighbourhood } V \\
		&\quad \text{ such that } r^{-1}(\ol{V}) \text{ is compact}\}
	\end{align*}
	and the set of \emph{sources} $
	E^0_{\mathrm{src}} \coloneqq E^0 \setminus \ol{r(E^1)}
	$.
	Elements of $E^0_{\mathrm{sing}} \coloneqq 	\ol{E^0_{\mathrm{src}}} \cup (E^0 \setminus E^0_{\mathrm{fin}})$ are called  \emph{singular vertices}, and elements of $E^0_{\mathrm{reg}} \coloneqq E^0 \setminus E^0_{\mathrm{sing}}$ are called \emph{regular vertices}.
	In our established terminology $E^0_{\mathrm{fin}} = \pr(r)$ and 
	\begin{align*}
		E_{\mathrm{reg}}^0 &= E^0 \setminus \Big(\ol{E^0 \setminus \ol{r(E^1)}} \cup (E^0 \setminus E^0_{\mathrm{fin}})\Big)  
		=\Big(E^0 \setminus \ol{E^0 \setminus \ol{r(E^1)}}\Big) \cap E^0_{\mathrm{fin}}\\
		&= \intt(\ol{r(E^1)}) \cap \pr(r)
		= \per(r).
	\end{align*}

	 Let $V \subseteq E^0$ be an open subset of $E^0_{\mathrm{fin}}$, so $V$ is $r$-proper by \cref{lem:f-reg_characterisation}. 	As in \cref{ex:top_graph_1}, let $r_i \colon E^{i+1} \to E^{i}$ be given by $r_i(e_1\cdots e_{i+1}) = e_1\cdots e_{i}$ for $i \ge 1$, and let $r_0 \coloneqq r$.  For each $i \ge 1$, the set $E^{i}V\coloneqq \{e_1 \cdots e_i \in E^i \mid s(e_i) \in V\}$ is $r_i$-proper. Setting $U_0 = V$, and $U_i = E^i V$ gives a regulating sequence $\Uu = (U_i)_{i \in \NN}$ for $(E^i,r_i)_{i \in \NN}$. The regulating sequence $\Uu$ is perfect if and only if $V \subseteq E^0_{\mathrm{reg}}$. 
	
	In the case where $V = E^{0}_{\mathrm{reg}}$ we write $E_{\mathrm{sing}}^i \coloneqq E^i \setminus E^iV$ for $i \ge 0$. Since $\unilim(X_i,f_i) \simeq E^{\le \infty}$ (see \cref{ex:top_graph_1}), $\perlim(X_i,f_i)$ is homeomorphic to the \emph{boundary path space}
	\[
	\partial E \coloneqq E^{\infty} \sqcup \bigsqcup_{i=0}^{\infty} E^i_{\mathrm{sing}}
	\]
	of \cite{deCa21,Web14,Yen06}. If $V \coloneqq E^0_{\mathrm{fin}}$, then we write $E_{\mathrm{inf}}^i = E^i \setminus E^iV$ for $i \ge 0$. We have
	\[
	\minlim(X_i,f_i) \cong E^{\infty} \sqcup \bigsqcup_{i=0}^{\infty} E^i_{\mathrm{inf}},
	\]
	which is a closed subspace of $\partial E$. As far as the author is aware, this space does not appear in the topological or directed-graph literature. Other choices of $V$ (or even $\Uu$) correspond to ``relative'' path spaces.
\end{example}

\section{\texorpdfstring{The $C^*$-algebraic setting}{The C*-algebraic setting}}
\label{sec:CStar}

In this section, we move from classical point-set topology to noncommutative $C^*$-algebras. Briefly, a $C^*$-algebra is a norm-closed $*$-subalgebra of the bounded linear operators on some complex Hilbert space.  We refer the reader to \cite{Arv76,Dav96,Mur90} for a standard introduction to $C^*$-algebras.

 By Gelfand duality, every \emph{commutative} $C^*$-algebra $A$ is isomorphic to $C_0(X)$, the algebra of continuous $\CC$-valued functions vanishing at infinity on some locally compact Hausdorff space $X$. 
The space $X$ is compact if and only if $A$ is unital (has a multiplicative identity). In fact, Gelfand duality provides an equivalence of categories between the category of compact Hausdorff spaces with continuous maps, and the category of unital commutative $C^*$-algebras with unital $*$-homomorphisms. A continuous map $f \colon X \to Y$ is dual to the $*$-homomorphism $f^* \colon C(Y) \to C(X)$  given by $f^*(a) = a \circ f$. As such, noncommutative unital $C^*$-algebras can be thought of as ``noncommutative compact Hausdorff spaces''. Results about the category of compact Hausdorff spaces and continuous maps often have analogous in the category of \emph{noncommutative} unital $C^*$-algebras and $*$-homomorphisms. This is referred to as the \emph{noncommutative topology} paradigm. 

 In the noncompact setting things are more subtle, requiring us to introduce some additional terminology.
By an ideal in $C^*$-algebra we mean a norm-closed 2-sided ideal. Ideals in $C_0(X)$ are isomorphic to $C_0(U)$ for some open set $U \subseteq X$, and the quotient $C_0(X)/C_0(U)$ is isomorphic to $C_0(X \setminus U)$, with the quotient map $C_0(X) \twoheadrightarrow  C_0(X \setminus U)$ given by restriction of functions to $X \setminus U$. 
 If $I$ is an ideal in a $C^*$-algebra $A$ then the \emph{annihilator} of $I$ is the ideal of $A$ given by
\begin{equation}\label{eq:annihilator}
	I^{\perp} = \{a \in A \colon ab = 0 \text{ for all } b \in I\}.
\end{equation}
 An ideal $I \trianglelefteq A$ is \emph{essential} if $I^{\perp} = \{0\}$.
 If $U \subseteq X$ is open, then $C_0(U)^{\perp} \cong C_0(X \setminus \overline{U})$, so $C_0(U)\trianglelefteq C_0(X)$ is essential if and only if $U$ is dense in $X$. 

The \emph{multiplier algebra} $\Mm(A)$ of a $C^*$-algebra $A$ is the largest unital $C^*$-algebra containing $A$ as an essential ideal (see \cite[Ch.~2]{Lan95}). The \emph{strict topology} on $\Mm(A)$ is generated by the seminorms $t \mapsto \|ta\|$ and $t \mapsto \|t^*a\|$ for all $a \in A$. For a commutative $C^*$-algebra $C_0(X)$ there are isomorphisms $\Mm(C_0(X)) \cong C_b(X) \cong C(\beta X)$, where $C_b(X)$ is the continuous bounded functions on $X$, and $\beta X$ is the Stone--\v{C}ech compactification of $X$. We usually identify $\Mm(C_0(X))$ with $C_b(X)$ and $C(\beta X)$. 

Let $B$ be another $C^*$-algebra.
 Following \cite{Lan95}, we call a $*$-homomorphism $\varphi \colon A \to \Mm(B)$ a \emph{morphism} from $A$ to $B$ if $\varphi(A)B \coloneqq {\spaan} \{ \varphi(a)b \colon a \in A, b \in B\}$ is norm-dense in $B$. We write $\Mor(A,B)$ for the collection of all morphisms from $A$ to $B$. By \cite[Proposition~2.5]{Lan95}, a $*$-homomorphism $\varphi \colon A \to \Mm(B)$ belongs to $\Mor(A,B)$ if and only if for some approximate unit $(a_{\lambda})$ of $A$, we have $\varphi(a_\lambda) \to 1_{\Mm(B)}$ in the strict topology on $\Mm(B)$.  If $I \trianglelefteq A$, then a morphism $\varphi \in \Mor(I,B)$ extends to a unique morphism $\ol{\varphi} \in
 \Mor(A,B)$ satisfying $\ol{\varphi}(a)b = \lim_{\lambda} \varphi(a u_{\lambda})b$ for all $a \in A$, $b \in B$, and approximate identities $(u_{\lambda})$ for $I$ \cite[Proposition~2.1]{Lan95}. If $\varphi \in \Mor(A,B)$, then the extension $\ol{\varphi} \colon \Mm(A) \to \Mm(B)$, induced by $A \trianglelefteq \Mm(A)$, is unital. 
  
  Following \cite{ELP99}, we say that a morphism $\varphi \in \Mor(A,B)$ is \emph{proper} if $\varphi(A) \subseteq B$. 
  We denote the collection of proper morphisms from $A$ to $B$ by $\Morp(A,B)$. 
  By post-compositing with the inclusion $B \hookrightarrow \Mm(B)$, a $*$-homomorphism $\varphi \colon A \to B$ corresponds to a $*$-homomorphism $ A \to \Mm(B)$ whose image is contained in $B$. Thus, a $*$-homomorphism $\varphi \colon A \to B$
  corresponds to a proper morphism in $\Morp(A,B)$ if and only if for any approximate identity $(a_{\lambda})$ for $A$, the net $(\varphi(a_{\lambda}))$ is an approximate identity for $B$. We identify $\Morp(A,B)$ with such $*$-homomorphisms.  
  
It is important to recognise that if $\varphi \colon A \to \Mm(B)$ belongs to $\Mor(A,B)$, then it is not usually the case that $\varphi \in \Morp(A,\Mm(B))$; as this requires $\ol{\varphi(A)\Mm(B)} = \Mm(B)$ (see \cref{rmk:nondegen_distinction} below). 
In the literature there is some confusion; the maps that we call ``morphisms'' and ``proper morphisms'' are both referred to a \emph{nondegenerate $*$-homomorphisms}, despite their subtle distinction. We refrain from using the term ``nondegenerate'' to avoid confusion, except when referring to representations of $C^*$-algebras on Hilbert spaces and Hilbert modules.

Both morphisms and proper morphisms can be used to define categories whose objects are $C^*$-algebras. If $\phi \in \Mor(A,B)$ and $\psi\in \Mor(B,C)$, then the composition of $\psi$ with $\phi$ is defined to be $\ol{\psi} \circ \phi \in \Mor(A,C)$, where $\ol{\psi}\colon \Mm(B) \to \Mm(C)$ is the extension induced by the ideal $B \trianglelefteq \Mm(B)$.
The category of $C^*$-algebras and proper morphisms is a wide subcategory of the category of $C^*$-algebras and morphisms. 
Composition of proper morphisms corresponds to the usual composition of $*$-homomorphisms.

If $f \colon X \to Y$ is a continuous map between locally compact Hausdorff spaces, and $a \in C_0(Y)$, then $f^*(a) \coloneqq a \circ f$ is continuous and bounded, but does not necessarily vanish at infinity. We instead get a $*$-homomorphism $f^* \colon C_0(Y) \to C_b(X) \cong \Mm(C_0(X))$ such that $f^*(C_0(Y))C_0(X)$ is dense in $C_0(X)$. That is, $f^* \in \Mor(C_0(Y),C_0(X))$. In fact, there is an equivalence of categories between the category of locally compact Hausdorff spaces and continuous maps and the category of commutative $C^*$-algebras with morphisms \cite{aHRW10,Wor80}. 
As such, the ``correct'' category of \emph{noncommutative} locally compact Hausdorff spaces is the category of $C^*$-algebras with morphisms. 

  If $f \colon X \to Y$ is also assumed to be proper, then the map $f^*$ instead takes values in $C_0(X)$. The equivalence of categories descends to one between the category of locally compact Hausdorff spaces with proper continuous maps and the category of commutative $C^*$-algebras with proper morphisms. The category of $C^*$-algebras and proper morphisms is the noncommutative analogue of the category of locally compact Hausdorff spaces and proper continuous maps. 

We summarise the preceding discussion with the following noncommutative dictionary:

	\begin{center}
		\def\arraystretch{1.1}
		\begin{tabular}{c | c }
			Topological spaces & $C^*$-algebras\\
			\hline
			\hline
			Compact Hausdorff spaces & Unital $C^*$-algebras\\
			Continuous maps & Unital $*$-homomorphisms\\
			\hline
			Locally compact Hausdorff spaces & Nonunital $C^*$-algebras\\
			 Continuous maps & Morphisms\\
			\hline
		Locally compact Hausdorff spaces & Nonunital $C^*$-algebras \\
			Continuous proper maps  & Proper morphisms\\
		\end{tabular}
	\end{center}
	
	\noindent For commutative $C^*$-algebras, the category on the left of each row is equivalent to the category on the right.

	We record some examples of morphisms that appear in the $C^*$-algebraic literature.
	
	\begin{example}\label{ex:reps}
		Let $A$ be a $C^*$-algebra 
		and suppose that $\pi \colon A \to \Bb(\Hh)$ is a representation on a Hilbert space $\Hh$. Since the compact operators $\Kk(H)$ are the unique nontrivial ideal of $\Bb(H)$, they are essential, and we have $\Bb(\Hh) \cong \Mm(\Kk(\Hh))$. 		
		If $\pi$ is a nondegenerate representation (in the sense that $\ol{\pi(A)\Hh} = \Hh$), then  \cite[p.~20]{Lan95} implies that $\pi \in \Mor(A,\Kk(\Hh))$. 
	\end{example}

	\begin{example}\label{ex:C0(X)-algebras}
		Let $A$ be a $C^*$-algebra and let $X$ be a locally compact Hausdorff space. The algebra $A$ is said to be a \emph{$C_0(X)$-algebra} if there is a $*$-homomorphism $\varphi$ from $C_0(X)$ into the center of $\Mm(A)$ such that $\ol{\varphi(C_0(X)) A} = A$. In particular, $\varphi \in \Mor(C_0(X),A)$.
	\end{example}

\begin{example}\label{ex:correspondences}
	Let $B$ be a $C^*$-algebra and suppose that $X_B$ is a right Hilbert $B$-module (see~\cite{Lan95,RW98}) with $B$-valued inner product $\langle \, \cdot \mid \cdot \, \rangle \colon X_B \times X_B \to B$. We denote $X_B$ by $X$ if the algebra $B$ is clear from context. Let $\Ll_B(X)$ denote the collection of adjointable $B$-linear operators on $X_B$, and let $\Kk_B(X) = \ol{\spaan}\{\Theta_{x,y} \mid x,y \in X\}$ denote its ideal of generalised compact operators, where $\Theta_{x,y}(z) = x \, \langle y | z\rangle$ for all $x,y,z \in X$. 
	
	Suppose $A$ is another $C^*$-algebra and that $\varphi \colon A \to \Ll_B(X)$ is a $*$-homomorphism. We say that $(\varphi,{}_A X_B)$ is an \emph{$A$--$B$-correspondence}. 
	If $\varphi$ nondegenerate (in the sense that $\ol{\varphi(A)X_B }= X_B$) then we say that $(\varphi,{}_A X_B)$ is a \emph{nondegenerate correspondence}.  It follows from  \cite[p.~20]{Lan95} that under the isomorphism $\Ll_B(X) \cong \Mm(\Kk_B(X))$, nondegeneracy of $\varphi \colon A \to \Ll_B(X)$ is equivalent to $\varphi \in \Mor(A,\Kk_B(X))$.
\end{example}

\subsection{Noncommutative fibrewise compactifications and perfections}\label{sec:nc_fw}
For our noncommutative analogue of fibrewise compactifications our starting datum is a morphism $\varphi \in \Mor(A,B)$. The idea is that we want to modify the codomain of $\varphi$ to attain a proper morphism. 

In \cite[Definition~3.8]{Pim97}, Pimsner introduced the following ideal. 

\begin{dfn}
Let $\varphi \in \Mor(A,B)$. We call the ideal
	 \[
	\pim(\varphi) \coloneqq \varphi^{-1}(B) 
	\]
	of $A$ the \emph{Pimsner ideal of $\varphi$}.
\end{dfn}
 The Pimsner ideal of $\varphi \in \Mor(A,B)$ is the largest ideal $J$ of $A$ such that $\varphi|_J \in \Morp(J,B)$.  
For commutative $C^*$-algebras, the Pimsner ideal is related to the open set $\pr(f)$ of \labelcref{eq:pimsner_set}.

The following result translates the notions of $f$-proper and $f$-perfect sets in topological spaces to ideals in $C^*$-algebras. In the language of topological graphs the result is known to experts. 

\begin{lem}[{cf.~\cite{Kat04}}]\label{lem:from_sets_to_ideals_and_back}
	Let $f \colon X \to Y$ be a continuous map between locally compact Hausdorff spaces and let $f^* \colon C_0(Y) \to C_b(X)$ be the $*$-homomorphism given by $f^*(a) = a \circ f$. 
	An open set $U \subseteq Y$ is $f$-proper if and only if $f^*(C_0(U)) \subseteq C_0(X)$, and $f$-perfect if and only if $U$ is $f$-proper and $f^*|_{C_0(U)}$ is injective. In particular, $\pim(f^*) = C_0(\pr(f))$. 
\end{lem}

\begin{proof}
	First suppose that $U \subseteq Y$ is $f$-proper. Identify $C_0(U)$ with the ideal of functions in $C_0(Y)$ that are identically zero on $Y \setminus U$. Since $U$ is $f$-proper, for any compactly supported $a \in C_c(U)$ the function $f^*(a)$ is compactly supported. So $f^*(C_0(U)) \subseteq C_0(X)$. If $a \in \ker(f^*) \cap C_0(U)$, then $f^*(a)(x) = a(f(x)) = 0$ for all $x \in U$. So, if $U$ is also $f$-perfect, then $a = 0$. That is, $f^*|_{C_0(U)}$ is injective.
	
	For the converse, suppose that $U \subseteq Y$ is open but not $f$-proper. Since $f|_{f^{-1}(U)}$ is not proper, there exists a compact set $K \subseteq U$ such that $f^{-1}(K)$ is not compact. Take a precompact open set $V \subseteq U$ such that $K \subseteq V \subseteq \ol{V} \subseteq U$. Since $\ol{V}$ is compact, it is normal, so the Tietze Extension Theorem gives a compactly supported function $a \in C_c(U)$ such that $a(y) = 1$ for all $y \in K$, and $a(y) = 0$ for all $y \in U \setminus \ol{V}$. So, $f^*(a)(x) = 1$ for all $x \in f^{-1}(K)$. Now fix a compact set $K_0 \subseteq X$. Since $f^{-1}(K)$ is closed and not compact, there exists $x_{K_0} \in  f^{-1}(K) \setminus K_0$. As $f^*(a)(x_{K_0}) = 1$ the function $f^*(a)$ does not vanish at infinity, so $f^*(C_0(U)) \not\subseteq C_0(X)$.
	
	Now suppose $f|_{f^{-1}(U)}$ does not surject onto $U$. Then there exists $y_0 \in U \setminus f(X)$. If $y_0 \in \ol{f(X)} \setminus f(X)$, then \cref{lem:boundary_issues} implies that $f|_{f^{-1}(U)}$ is not proper, so we can assume that $y_0 \in U \setminus \ol{f(X)}$. Take a precompact neighbourhood $V$ of $y_0$ such that $\ol{V} \subseteq   U \setminus \ol{f(X)}$. Using Tietze's Extension Theorem on $\ol{V}$, there exists $a \in C_c(U)$ such that $a(y_0) = 1$ and $a(y) = 0$ for all $y \in U \setminus \ol{V}$. Then $f^*(a) = 0$, so $f^*$ is not injective. 
	
	The final statement follows from the definition of $\pim(f^*)$ and the fact that $\pr(f)$ is the maximal $f$-proper set. 
\end{proof}

\cref{lem:from_sets_to_ideals_and_back} indicates the importance of the Pimsner ideal in our formulation of noncommutative fibrewise compactifications.

If $I$ is an ideal in a $C^*$-algebra $A$, then there is an induced $*$-homomorphism $\alpha_I \colon A \to \Mm(I)$ satisfying $\alpha(a)b = ab$ for all $a \in A$ and $b \in I$. Since $I$ is an ideal $\ol{\alpha_I(A)I }= I$, so $\alpha_I \in \Mor(A,I)$. Moreover, $\alpha_I$ restricts to an isomorphism between the copies of $I$ in $A$ and in $\Mm(I)$. 

\begin{dfn}\label{dfn:perfection_nc}
	Let $A$ and $B$ be $C^*$-algebras and let $\varphi \in \Mor(A,B)$. A \emph{(noncommutative) fibrewise compactification} of $\varphi$ is a pair $(C,\psi)$ consisting of a $C^*$-algebra $C$ and $\psi \in \Morp(A,C)$ such that
	\begin{enumerate}[labelindent=0pt,labelwidth=\widthof{\ref{itm:NCP1}},label=(NF\arabic{enumi}), ref=(NF\arabic*),leftmargin=!]
		\item\label{itm:NCP1} the algebra $B$ is isomorphic to an ideal of $C$ and if $\alpha_C \in \Mor(C,B)$ is the induced morphism, then $\alpha_C \circ \psi = \varphi$; and
		
		\item\label{itm:NCP2}
	the image of the $*$-homomorphism $\ol{\psi} \colon \pim(\varphi) \to C/B$, given by $\ol{\psi}(a) = \psi(a) + B$, contains $\{c + B \colon c \in B^{\perp} \trianglelefteq C\}$.

	\end{enumerate}
	If $\psi$ is also injective, then we call $(C,\psi)$ a \emph{(noncommutative) perfection}. If $B$ is essential in $C$, then we call $(C,\psi)$ a \emph{strict fibrewise compactification}.
\end{dfn}

\begin{rmk}\label{rmk:nondegen_distinction}
Given $\varphi \in \Mor(A,B)$ it is tempting to think that $(\Mm(B),\varphi)$ is a fibrewise compactification of $\varphi$. Indeed, $\alpha_{\Mm(B)} = \id_{\Mm(B)}$ and $B$ is an essential ideal in $\Mm(B)$. 
 However, for $\varphi \in \Mor(A,B)$ we typically have $\ol{\varphi(A)\Mm(B)} \ne \Mm(B)$, so  $\varphi \notin \Morp(A,\Mm(B))$. 
 For example, suppose that $B$ is nonunital and $\varphi \colon B \to \Mm(B)$ is the inclusion of $B$ as a proper ideal. Then, $\ol{\varphi(B)\Mm(B)} = B \ne \Mm(B)$. 
\end{rmk}

\begin{rmk}\label{rmk:strict_nice}
	To check that $(C,\psi)$ is a \emph{strict} fibrewise compactification we only need to verify that \labelcref{itm:NCP1} holds and that $B^{\perp} = \{0\}$, since \labelcref{itm:NCP2} holds automatically. 
\end{rmk}

We record the following characterisation of the annihilator of an ideal. 
\begin{lem}\label{lem:annihilator_kernel}
	Let $C$ be a $C^*$-algebra, let $B \trianglelefteq C$, and let $\alpha_C \in \Mor(C,B)$ be the induced morphism. Then $\ker(\alpha_C) =  B^{\perp}$.  
\end{lem}

\begin{proof}
	Let $B'$ denote the canonical copy of $B$ in $\Mm(B)$ and note that $\alpha_C$ restricts to an isomorphism from $B$ to $B'$.
	Fix $c \in \ker(\alpha_C)$. For all $b \in B$ we have $cb = \alpha_C^{-1}(\alpha_C(cb)) = 0$, and similarly $bc = 0$, so $c \in B^{\perp}$. Now suppose that $c \in B^{\perp}$. Then for all $b' \in B'$ we have $\alpha_C(c) b' = \alpha_C(c \alpha_C^{-1}(b')) = 0$, so $c \in \ker(\alpha_C)$. 
\end{proof}

The next lemma shows that the condition \labelcref{itm:NCP2} can be strengthened to an equality.

\begin{lem}\label{rmk:NCP2}
	Let $A$ and $B$ be $C^*$-algebras and let $\varphi \in \Mor(A,B)$. Suppose that $(C,\psi)$ is a fibrewise compactification of $\varphi$. Then
	\begin{enumerate}
		\item\label{itm:nc_equal} $\ol{\psi}(\pim(\varphi)) = \{c + B \colon c \in B^{\perp} \}$; and
		\item\label{itm:nc_quotient}  $\pim(\varphi)/\psi^{-1}(B) \cong B^{\perp}$.
	\end{enumerate}
\end{lem}

\begin{proof}
	 \labelcref{itm:nc_equal} Let $q \colon C \to  C/B$ be the quotient map. Since $B \cap B^{\perp} = \{0\}$, the algebra $q(B^{\perp}) = \{ c + B \colon c \in B^{\perp}\}$ is isomorphic to $B^{\perp}$. 	By \labelcref{itm:NCP2} we have $\ol{\psi}(\pim(\varphi)) \supseteq q(B^{\perp})$. Fix $a \in \pim(\varphi)$. By \labelcref{itm:NCP1} we have $\alpha_C \circ \psi(a) = \phi(a) \in B \subseteq \Mm(B)$. Hence, there exists $b \in B$ such that $\alpha_C(b - \psi(a)) = 0$. Since $\ker(\alpha_C) = B^{\perp}$, there exists $b' \in B^{\perp}$ such that $\psi(a) = b + b' \in B + B^{\perp}$. It follows that $\ol{\psi}(a) = b' + B \in q(B^{\perp})$, so $\ol{\psi}(\pim(\varphi)) = q(B^{\perp})$.
	 
	 \labelcref{itm:nc_quotient} Observe that $\ker(\ol{\psi}) = \psi^{-1}(B) \trianglelefteq \pim(\varphi)$. Since $\ol{\psi}$ maps onto $q(B^{\perp})$, it induces an isomorphism $\pim(\varphi)/\psi^{-1}(B) \cong B^{\perp}$. 
\end{proof}

Noncommutative fibrewise compactifications directly generalise fibrewise compactifications of continuous maps between locally compact Hausdorff spaces. 
To see this we require the following lemma, known to experts.
\begin{lem}\label{lem:ideal+perp_commutative}
	Let $I$ be an ideal of a $C^*$-algebra $A$. 
	\begin{enumerate}
		\item\label{itm:comm_ideals_1}	The image of $I^{\perp}$ in $A/I$ is an essential ideal.
		\item\label{itm:comm_ideals_2} If $I$ and $A/I$ are commutative, then so is $A$.
		\item\label{itm:comm_ideals_3} If $I$ and $I^{\perp}$ are commutative, then so is $A$.
	\end{enumerate} 
\end{lem}

\begin{proof}
	\labelcref{itm:comm_ideals_1} Let $q \colon A \to A/I$ denote the quotient map. If $a \in A$ is such that $q(a I^{\perp}) = 0$, then $a I^{\perp} \subseteq I$, but this can only happen if $a = 0$. So, $q(I^{\perp})$ is essential in $A/I$. 
	
	\labelcref{itm:comm_ideals_2}  Let $I$ and $A/I$ be commutative. Suppose for contradiction that $a,b \in A$ are such that $[a,b] \ne 0$. Since $A/I$ is commutative, $q([a,b]) = 0$, so $[a,b] \in I$. Fix an irreducible representation $\pi \colon I \to \CC$ such that $\pi([a,b]) \ne 0$. Then $\pi$ extends to an irreducible representation $\widetilde{\pi} \colon A \to \CC$ by \cite[Lemma I.9.14]{Dav96} and $0 \ne \widetilde{\pi}([a,b]) = [\widetilde{\pi}(a), \widetilde{\pi}(b)]$, contradicting the commutativity of $\CC$. 
	
	\labelcref{itm:comm_ideals_3} Suppose that $I$ and $I^{\perp}$ are commutative. Part \labelcref{itm:comm_ideals_1} and the universal property of the multiplier algebra implies that $A/I$ is a subalgebra of the commutative algebra $\Mm(q(I^{\perp}))$. So, $A/I$ is commutative, and by \labelcref{itm:comm_ideals_2}, so is $A$.
\end{proof}

 Recall that the spectrum $\widehat{A}$ of a commutative $C^*$-algebra $A$ consists of all nonzero $*$-homomorphisms $A \to \CC$ with the weak* topology. If $B$ is another commutative $C^*$-algebra and $\varphi \in \Mor(A,B)$, then Gelfand duality \cite[Theorem~2]{aHRW10} yields a continuous map $\varphi_* \colon \widehat{B} \to \widehat{A}$ such that $(\varphi_*)^* = \varphi$. The morphism $\varphi$ is proper if and only if $\varphi_*$ is proper. 
 
\begin{prop}\label{prop:fw_commutative_same}
Let $ f \colon X \to Y$ be a continuous map between locally compact Hausdorff spaces. If $(Z,g)$ is a fibrewise compactification of $f$, then $(C_0(Z),g^*)$ is a fibrewise compactification of $f^* \in \Mor(C_0(Y),C_0(X))$. Conversely, if $(C,\psi)$ is a fibrewise compactification of $f^*$, then $C$ is commutative and $(\widehat{C}, \psi_*)$ is a fibrewise compactification of $f$. The same is true when restricting to perfections.
\end{prop}

\begin{proof}
	First suppose that $(Z,g)$ is a fibrewise compactification of $f$. Let $A \coloneqq C_0(Y)$, $B \coloneqq C_0(X)$, $C \coloneqq C_0(Z)$, and $\varphi \coloneqq f^* \in \Mor(A,B)$.
	 	 Since $g$ is proper, \cref{lem:from_sets_to_ideals_and_back} implies that $\psi \coloneqq g^* \in \Morp(A,C)$. Since $\iota_X \colon X \to Z$ is open, $B$ is isomorphic to an ideal of $C$. Let $\alpha_{C} \in \Mor(C,C_0(X))$ be the induced $*$-homomorphism.
	  Using \labelcref{perf:commuting} at the second equality we have 
	\[
	(\alpha_C \circ \psi) (a)(x) = a(g \circ \iota_X(x)) = a(f(x)) = \varphi(a)(x)
	\] 
	for all $a \in A$ and $x \in X$. 
	So, \labelcref{itm:NCP1} is satisfied.
	
	For~\labelcref{itm:NCP2}, note that  $C/B \cong C_0(Z \setminus \iota_X(X))$ and $\{c + B \mid c \in B^{\perp}\} = C_0(Z \setminus \ol{\iota_X(X)})$. By \cref{lem:from_sets_to_ideals_and_back}, we have $\pim(\varphi) = C_0(\pr(f))$. For $a \in \pim(\varphi)$ we have $\ol{\psi}(a) = (a \circ g)|_{Z \setminus \iota_X(X)}$.  We use~\labelcref{perf:density} to see that $g$ restricts to a homeomorphism from $Z \setminus \ol{\iota_X(X)}$ to a closed subset $S$ of $\pr(f)$ in the subspace topology. Fix a compactly supported function $a_0 \in C_c(Z \setminus \ol{\iota_X(X)})$ and let $a \in C_c(S)$ be the corresponding function induced by the homeomorphism. By the Tietze Extension Theorem we can extend $a$ to a function $a' \in \pim(\varphi)$ with $\ol{\psi}(a') = a_0$. Density of compactly supported functions in functions vanishing at infinity gives \labelcref{itm:NCP2}. That is, $(C_0(Z),g^*)$ is a fibrewise compactification of $f^*$.  If $g$ is surjective, then $g^*$ is injective. 
	 
	Now suppose that $(C,\psi)$ is a fibrewise compactification of $f^*  \in \Mor(C_0(Y),C_0(X))$. Let $A \coloneqq C_0(Y)$, $B \coloneqq C_0(X)$, and $\varphi \coloneqq f^*$. By~\cref{rmk:NCP2}, there is an isomorphism $\pim(\varphi)/\psi^{-1}(B) \cong B^{\perp}$. In particular, $B^{\perp}$ is commutative, so by \cref{lem:ideal+perp_commutative} the algebra $C$ is commutative. Since $B \trianglelefteq C$ there is an open inclusion $\iota_X \colon X \to \widehat{C}$. 
	By Gelfand duality $\psi \colon A \to C$ is dual to a proper map $g \coloneqq \psi_* \colon \widehat{C} \to Y$. 
	
	By \labelcref{itm:NCP1}, for all $a \in A$ and $b \in B$ we have $\psi(a)b = \phi(a)b \in B$. Fix $x \in X$ and take $b \in B$ such that $b(x) = 1$. Then 
	\[
	a(g \circ \iota_X(x)) = \psi(a)b(x) =\phi(a)b(x) = a(f (x))
	\] 
	for all $a \in A$, so $g \circ \iota_X = f$. That is, \labelcref{perf:commuting} is satisfied. 
	  
	For \labelcref{itm:NCP2}, note that $B^{\perp} = C_0(\widehat{C} \setminus \ol{\iota_X(X)})$. The quotient $\pim(\varphi)/\psi^{-1}(B)$ corresponds to a closed subset $S$ of $\pr(f)$ in the subspace topology. The isomorphism $B^{\perp} \cong \pim(\varphi)/\psi^{-1}(B)$ of \cref{rmk:NCP2} is induced by the restriction of $g$ to $\widehat{C} \setminus \ol{\iota_X(X)}$. In particular, $g \colon \widehat{C} \setminus \ol{\iota_X(X)} \to S$ is a homeomorphism. So, $(\widehat{C},\varphi_*)$ is a fibrewise compactification of $f$. If $\varphi$ is injective, then $\varphi_*$ is surjective. 
\end{proof}

\begin{dfn}
	Suppose that $(C,\psi)$ is a fibrewise compactification of $\varphi \in \Mor(A,B)$. We say that a fibrewise compactification $(C',\psi')$ of $\varphi$ is a \emph{quotient fibrewise compactification} if there is a surjective $*$-homomorphism $q \colon C \to C'$ such that $\psi' \circ q = \psi$ and $\alpha_{C'} \circ q = \alpha_C$. That is, the diagram
\[\begin{tikzcd}[row sep = 0pt,ampersand replacement=\&,cramped]
	\& C \\
	A \&\& {\Mm(B)} \\
	\& {C'}
	\arrow["{\alpha_C}", from=1-2, to=2-3]
	\arrow["q"', from=1-2, to=3-2]
	\arrow["\psi", from=2-1, to=1-2]
	\arrow["{\psi'}"', from=2-1, to=3-2]
	\arrow["{\alpha_C'}"', from=3-2, to=2-3]
\end{tikzcd}\]
	commutes.
	 If $\psi$ and $\psi'$ are both injective we say that $(C',\psi')$ is a \emph{quotient perfection}. If $q$ is an isomorphism we say that $(C,\psi)$ and $(C',\psi')$ are \emph{isomorphic}. 
\end{dfn}

\begin{rmk}\label{rmk:B_preservation}
	Suppose $(C',\psi')$ is a quotient fibrewise compactification of $(C,\psi)$. Since $\alpha_{C'}$ and $\alpha_C$ restrict to isomorphisms on their respective ideals that are isomorphic to $B$, and $\alpha_{C'} \circ q = \alpha_C$, the map $q$ also restricts to an isomorphism between the copies of $B$.
\end{rmk}

\begin{dfn}\label{dfn:essential_subperfection}
	Let $(C,\psi)$ be a fibrewise compactification of $\varphi$ and let $q_{\perp} \colon C \to C/B^{\perp}$ be the quotient map. Then
	\[
	([C],[\psi]) \coloneqq (C/{B^{\perp}},q_{\perp} \circ \psi)
	\]
	is a strict fibrewise compactification.
\end{dfn}

\begin{lem}
	Let $(C,\psi)$ be a fibrewise compactification of $\varphi \in \Mor(A,B)$. Then $([C],[\psi]) \cong (\alpha_C(C),\alpha_C \circ \psi)$. If $(C',\psi')$ is a quotient fibrewise compactification of $(C,\psi)$, then $([C],[\psi]) \cong ([C'],[\psi'])$.
\end{lem}

\begin{proof} 
Let $B^{\perp}$ denote the annihilator of $B$ in $C$. 
For the first statement we use \cref{lem:annihilator_kernel} to see that $C/B^{\perp} = C/\ker(\alpha_C) \cong \alpha_C(C)$. The isomorphism takes $c + B^{\perp}$ to $\alpha_C(c)$, so $\alpha_C \circ \psi$ and $q_{\perp} \circ \psi$ coincide under the isomorphism. 

For the second statement let $q \colon C \to C'$ denote the quotient map. Since $q$ surjects and $\alpha_{C'} \circ q = \alpha_C$,  we have $\alpha_{C'} (C') = \alpha_C(C)$. Since $\alpha_C \circ \psi = \alpha_{C'} \circ \psi'$, the previous paragraph gives $([C],[\psi]) \cong ([C'],[\psi'])$.
\end{proof}

\subsection{The unified algebra}\label{sec:unified_alg}
The unified space of \cref{dfn:unified_space} fits naturally into the noncommutative topology paradigm via the language of split extensions.
 We recall some notions from the theory of  extensions of $C^*$-algebras and refer to \cite{Bla98} and \cite{W-O93} for further details. 

Let $A$ and $B$ be $C^*$-algebras. An \emph{extension of $A$ by $B$} is a short exact sequence $0 \to B \to E \to A \to 0$ of $C^*$-algebras and $*$-homomorphisms. Often, $E$ itself is referred to as an extension of $A$ by $B$. Since $B$ is an ideal in $E$, there is an induced morphism $\alpha \in \Mor(E,B)$ that descends to a $*$-homomorphism $\beta \colon A \to \Qq(B) \coloneqq \Mm(B)/B$. With $q_B \colon \Mm(B) \to \Qq(B)$ denoting the quotient map, there is a commuting diagram
\[
\begin{tikzcd}
	0 \arrow[r]
	& B \arrow[r] \arrow[d,equals]
	& E \arrow[r] \arrow[d,"\alpha"]
	& A \arrow[r] \arrow[d,"\beta"]
	& 0\\
	0 \arrow[r]
	& B \arrow[r]
	& \Mm(B)\arrow[r,"q_B"]
	& \Qq(B) \arrow[r] 
	& 0
\end{tikzcd}
\]
with exact rows. 
The map $\beta$ is a complete isomorphism invariant of the extension called the \emph{Busby invariant}. Indeed, the algebra $E$ is isomorphic to the pullback algebra
\[
\Mm(B) \oplus_{q_B,\beta} A \coloneqq \{(m,a) \in \Mm(B) \oplus A \mid q_B(m) = \beta(a) \}.
\]
An extension $0 \to B \to E \overset{q}{\to} A \to 0$ is \emph{split} if there is a $*$-homomorphism $s \colon A \to E$ such that $q \circ s = \id_A$. For a split extension, the Busby invariant lifts to a $*$-homomorphism $\ol{\beta} \coloneqq \alpha \circ s \colon A \to \Mm(B)$. In this case, $E$ is isomorphic to
\begin{equation}\label{eq:unified_algebra}
	\Mm(B) \oplus_{\ol{\beta}} A \coloneqq \{ (m,a) \in \Mm(B) \oplus A \mid m - \ol{\beta}(a) \in B\}.
\end{equation}

Unified spaces are dual to split extensions in the following sense. 

\begin{thm}\label{prop:unified_cstar}
	Suppose that $f \colon X \to Y$ is a continuous map between locally compact Hausdorff spaces with unified space $(\unis{X}{f}{Y},\unif{f})$. There is a split extension
	\begin{equation}\label{eq:split_commutative}
		\begin{tikzcd}[ampersand replacement=\&,cramped]
			0 \arrow[r]
			\& C_0(X) \arrow[r]
			\& C_0(\unis{X}{f}{Y}) \arrow[r,"q"']
			\& C_0(Y) \arrow[r] \arrow[l,"\unif{f}^*"', bend right=33]
			\& 0
		\end{tikzcd},
	\end{equation}
	where $q$ is the restriction of functions to $Y$, and $\unif{f}^*$ is dual to the proper map $\unif{f}$. The map $f^* \in \Mor(C_0(Y),C_0(X))$ is a lift of the Busby invariant of the split extension and
	\begin{equation}\label{eq:pullback_unified}
		C_0(\unis{X}{f}{Y}) \cong \{(g,h) \in C_b(X) \oplus C_0(Y) \mid g - f^*(h) \in C_0(X) \}.
	\end{equation}
	Moreover, every split extension of $C_0(Y)$ by $C_0(X)$ is isomorphic to $C_0(\unis{X}{f}{Y})$ for some continuous map $f \colon X \to Y$.
\end{thm}

\begin{proof}
	Since $Y$ is closed in $\unis{X}{f}{Y}$, the restriction of functions in $C_0(\unis{X}{f}{Y})$ to $Y$ gives a surjective $*$-homomorphism $q \colon C_0(\unis{X}{f}{Y}) \to C_0(Y)$ with kernel isomorphic to $C_0(X)$. For each $a \in C_0(X)$ and $y \in Y$ we have
	\[
	q \circ \unif{f}^* (a) (y) = \unif{f}^* (a)(\iota_Y(y)) = a(\unif{f} \circ \iota_Y(y)) = a(y),
	\]
	so $\unif{f}^*$ splits $q$.
	Let $\alpha \in \Mor(C_0(\unis{X}{f}{Y}), C_0(X))$ be the morphism induced by $C_0(X) \trianglelefteq C_0(\unis{X}{f}{Y})$. Thinking of $\alpha$ as a $*$-homomorphism $C_0(\unis{X}{f}{Y}) \to C_b(X)$ we have $\alpha(a)(x) = a \circ \iota_X(x)$ for all $a \in C_0(\unis{X}{f}{Y})$ and $x \in X$. Since $\unif{f} \circ \iota_X = f$, we have $\alpha \circ \unif{f}^* = f^*$. Hence, the diagram
	\[
	\begin{tikzcd}[ampersand replacement=\&,cramped]
		0 \arrow[r]
		\& C_0(X) \arrow[r] \arrow[d,equals]
		\& C_0(\unis{X}{f}{Y}) \arrow[r,"q"'] \arrow[d, "\alpha"]
		\& C_0(Y) \arrow[r] \arrow[r] \arrow[l,"\unif{f}^*"', bend right=33] \arrow[dl, "f^*"] \arrow[d, "\beta"]
		\& 0\\
		0 \arrow[r]
		\& C_0(X) \arrow[r]
		\& \Mm(C_0(X)) \arrow[r]
		\& \Qq(C_0(X)) \arrow[r] 
		\& 0
	\end{tikzcd}
	\]
	commutes, where $\beta$ being the Busby invariant for \labelcref{eq:split_commutative}. Hence, $f^*$ is a lift of the Busby invariant, so \labelcref{eq:unified_algebra} gives \labelcref{eq:pullback_unified}. 
	
	For the final statement, \cref{lem:ideal+perp_commutative} implies that extensions of commutative $C^*$-algebras are commutative. 
	Suppose that we have a split extension of commutative $C^*$-algebras
	\begin{equation*}\label{eq:split_commutative2}
		\begin{tikzcd}
			0 \arrow[r]
			& C_0(X) \arrow[r]
			& C_0(Z) \arrow[r,"q"',start anchor={[yshift=1.5ex]south east}, end anchor ={[yshift=1.5ex]south west}]
			& C_0(Y) \arrow[r] \arrow[l,"s"', start anchor={[yshift=-1.5ex]north west}, end anchor ={[yshift=-1.5ex]north east}]
			& 0,
		\end{tikzcd}
	\end{equation*}
	and without loss of generality assume that $X\subseteq Z$ is open and $Y = Z \setminus X$.
	Let $\alpha \in \Mor(C_0(Z),C_0(X))$ be the morphism induced by the extension.
	 Let $\varphi \coloneqq \alpha \circ s \in \Mor(C_0(Y),C_0(X))$. For each $x \in X$ let $\epsilon_x \in C(\beta X)^* = \Mm(C_0(X))^*$ denote the evaluation functional at $x$. By Gelfand duality, the map $\epsilon_x \mapsto \epsilon_x \circ \varphi$ induces a continuous map $\varphi_* \colon X \to Y$ such that $(\varphi_*)^* = \varphi$, \cite[Theorem~2]{aHRW10}. It follows that
	\[
	C_0(Z) \cong \{ (g,h) \in C(\beta X) \oplus C_0(Y) \mid g - (\varphi_*)^*(h) \in C_0(X)\}.
	\]
	The Busby invariant is a complete invariant of extensions, so $C_0(Z) \cong C_0(\unis{X}{\varphi_*}{Y})$. By Gelfand duality, $Z \simeq \unis{X}{\varphi_*}{Y}$. 
\end{proof}
In light of \cref{prop:unified_cstar} we introduce the following terminology. 

\begin{dfn}
	Let $A$ and $B$ be $C^*$-algebras and suppose that $\varphi \in \Mor(A,B)$. Let 
	\begin{align*}
	\unis{B}{\varphi}{A} \coloneq	
	\Mm(B) \oplus_{\varphi} A
		= \{(m,a) \in \Mm(B) \oplus A \mid m - \varphi(a) \in B\}
	\end{align*}
	and let $\unif{\varphi} \colon A \to \unis{B}{\varphi}{A}$ be given by $\unif{\varphi} (a) = (\varphi(a),a)$. We refer to both $\unis{B}{\varphi}{A}$ and $(\unis{B}{\varphi}{A}, \unif{\varphi})$ as the \emph{unified algebra of $\varphi$}.
\end{dfn}

Although it is not standard in the $C^*$-algebraic literature,  the name \emph{unified algebra} is justified by \cref{prop:unified_cstar} as whenever $f \colon X \to Y$ is a continuous map between locally compact Hausdorff spaces,
$
C_0(\unis{X}{f}{Y}) \cong \unis{C_0(X)}{f^*}{ C_0(Y)}.
$

For a general $\varphi \in \Mor(A,B)$, there is an extension $0 \to B \overset{\iota}{\to} \unis{B}{\varphi}{A} \overset{q}{\to} A \to 0$ with $\iota(b) = (b,0)$ and $q(b,a) = a$ that splits via $\unif{\varphi}$. The discussion preceding \cref{prop:unified_cstar} implies that $\varphi$ corresponds to the lift $\ol{\beta}$ of the Busby invariant for the split extension. Let $\alpha \in \Mor(\unis{B}{\varphi}{A}, B)$ be the morphism induced by $B \trianglelefteq \unis{B}{\varphi}{A}$. Then $\alpha(m,b)b' = mb'$ for all $b' \in B$, so $\alpha(m,b) = m$.

\begin{rmk}\label{rmk:unified_alg_alternate}
Observe that if $(m,a) \in \unis{B}{\varphi}{A}$, then there is some $b \in B$ such that $(m,a) = (\varphi(a) + b, a) = \unif{\varphi}(a) + (b,0)$. So, $\unis{B}{\varphi}{A}$ is isomorphic to the internal sum $\unif{\varphi}(A) + B \oplus 0$ in  $\Mm(B) \oplus A$.
\end{rmk}

\begin{prop}\label{prop:unified_alg_perf}
		Let $A$ and $B$ be $C^*$-algebras and let $\varphi \in \Mor(A,B)$. Then $(\unis{B}{\varphi}{A}, \unif{\varphi})$ is a perfection of $\varphi$ and $B^{\perp} = 0 \oplus \pim(\varphi)$.
\end{prop}

\begin{proof}
	Since $\varphi \in \Mor(A,B)$, for any approximate identity $(a_{\lambda})$ for $A$, $a \in A$, and $b \in B$, we have
\[
\unif{\varphi}(a_{\lambda}) (b + \varphi(a),a) = (\varphi(a_{\lambda})b + \varphi(a_{\lambda}a), a_{\lambda}a) \to (b + \varphi(a),a).
\]
So, $\unif{\varphi} \in \Morp(A,\unis{B}{\varphi}{A})$.
	
For \labelcref{itm:NCP1}, observe that $B \oplus 0$ is an ideal in $\unis{B}{\varphi}{A}$, and that $\unif{\varphi}(a) (b,0) = (\varphi(a)b,0)$ for all $a \in A$ and $b \in B$. Let $\alpha \in \Mor(\unis{B}{\varphi}{A} , B)$ be the induced map. Since $\alpha(m,a) = m$, we have $\alpha \circ \unif{\varphi} = \varphi$. 

For \labelcref{itm:NCP2}, we first show that $B^{\perp} = 0 \oplus \pim(\varphi)$. If $a \in \pim(\varphi)$, then $(0,a) \in \unis{B}{\varphi}{A}$. So, for all $b \in B$, we have $(0,a)(b,0) = 0$. That is, $0 \oplus \pim(\varphi) \subseteq B^{\perp}$. On the other hand, if $(\varphi(a) - b,a) \in \unis{B}{\varphi}{A}$ is such that $0 = (\varphi(a) - b,a)(b',0) =(\varphi(a)b'- b b',0)$ for all $b' \in B$, then $\varphi(a)b' = b b'$ for all $b' \in B$. Hence, $\varphi(a) = b$ and so $a \in \pim(\varphi)$. That is, $B^{\perp} = 0 \oplus \pim(\varphi)$. 
 
The map $\ol{\unif{\varphi}} \colon \pim(\varphi) \to (\unis{B}{\varphi}{A}) / (B \oplus 0)$ of \labelcref{itm:NCP2} satisfies
\[
\ol{\unif{\varphi}}(a) = (\varphi(a),a) + B \oplus 0 = (0,a) + B \oplus 0
\] 
for all $a \in \pim(a)$, so \labelcref{itm:NCP2} is satisfied. Since $\unif{\varphi}$ is injective, $(\unis{B}{\varphi}{A}, \unif{\varphi})$ is a perfection of $\varphi$.
\end{proof}

The strict fibrewise compactification of \cref{dfn:essential_subperfection}, associated to the unified algebra, also admits an explicit description. 

\begin{prop}\label{lem:minimal_alg}
	Let $A$ and $B$ be $C^*$-algebras and let $\varphi \in \Mor(A,B)$. Let $B + \varphi(A)$ be the subalgebra of $\Mm(B)$ generated by $B$ and $\varphi(A)$. 
	Then $([\unis{B}{\varphi}{A}],[\unif{\varphi}]) \cong (B + \varphi(A), \varphi)$.
\end{prop}

\begin{proof}
	Let $\alpha \in \Mor(\unis{B}{\varphi}{A},B)$ be the morphism induced by $B \trianglelefteq \unis{B}{\varphi}{A}$. Since $\alpha(m,a) = m$ for all $(m,a) \in \unis{B}{\varphi}{A}$,  \cref{lem:minimal_alg} gives
	$
	[\unis{B}{\varphi}{A}] 
	\cong \alpha(\unis{B}{\varphi}{A}) = B +\varphi(A).
	$
	Under this isomorphism $[\unif{\varphi}] = \alpha \circ \unif{\varphi}$ takes $a \in A$ to $\varphi(a)$. 
\end{proof}

\begin{rmk}
	The map $\varphi$ of \cref{lem:minimal_alg} that appears in $(B + \varphi(A), \varphi)$ is, strictly speaking, different from the original morphism $\varphi \in \Mor(A,B)$, since its codomain is $B + \varphi(A)$. We abuse notation and write $\varphi$ for both morphisms. 
\end{rmk}

\begin{dfn}
	Let $A$ and $B$ be $C^*$-algebras and let $\varphi \in \Mor(A,B)$. We call the strict fibrewise compactification $([\unis{B}{\varphi}{A}],[\unif{\varphi}])$ the \emph{minimal fibrewise compactification} of $\varphi$. 
\end{dfn}

The term ``minimal'' is justified since $B + \varphi(A)$ is the smallest $C^*$-algebra containing $B$ as an ideal and $\varphi(A)$.

\begin{example}\label{ex:reps_unified}
	 Following \cref{ex:reps}, let $\pi \colon A \to \Bb(\Hh)$ be a nondegenerate representation. So, $\pi \in \Mor(A,\Kk(\Hh))$. Then $\unis{\Kk(\Hh)}{\pi}{A} = \{(T + \pi(a),a) \mid T \in \Kk(\Hh), a \in A \}$ is a subalgebra of $\Bb(\Hh) \oplus A$. By \cref{lem:minimal_alg}, $[\unis{\Kk(\Hh)}{\pi}{A}] = \pi(A) + \Kk(\Hh)$ is the subalgebra of compact perturbations of $\pi(A)$ in $\Bb(\Hh)$.
\end{example}

\begin{example}\label{ex:toeplitz}

	As in \cref{ex:correspondences}, let $(\varphi,X_B)$ be a nondegenerate $A$--$B$-correspondence. By \cref{lem:minimal_alg}, we have $[\unis{\Kk_B(X)}{\varphi}{A}] \cong \varphi(A) + \Kk_B(X)$, the algebra of compact perturbations of $\varphi(A)$ in $\Ll_B(X)$.
	
	If $A = B$, then the unified algebra $\unis{\Kk_A(X)}{\varphi}{A}$ admits an explicit description in terms of Toeplitz algebras of correspondences. 
	Following \cite{Kat04cor,Pim97}, the \emph{Fock space} of $(\varphi,X_A)$ is the right Hilbert $A$-module
	\[
	F_{X} \coloneqq \bigoplus_{n \ge 0} X^{\ox_A n} = \bigoplus_{n \ge 0} \underbrace{X \ox_A\cdots \ox_A X}_{n\text{-terms}},
	\] 
	where $X^{\ox_A 0} \coloneqq A_A$ is a right $A$-module under right multiplication. There is an injective nondegenerate $*$-homomorphism $\varphi_{\infty} \colon A \to \Ll_A(F_X)$ satisfying $\varphi_{\infty}(a)(x_1 \ox \cdots \ox x_n) = (\varphi(a)x_1) \ox \cdots \ox x_n$ for all $x_1 \ox \cdots \ox x_n \in X^{\ox n}$ and $\varphi(a)a' = aa'$ for all $a' \in A = X^{\ox 0}$.
	
	To each $x \in X$, we associate a \emph{creation operator} $T_x \in \Ll_A(F_X)$ that satisfies $T_x(x_1 \ox \cdots \ox x_n) = x \ox x_1 \ox \cdots \ox x_n$ and $T_x(a) = xa$ for $a \in X^{\ox 0}$. The \emph{Toeplitz algebra} $\Tt_X$ of $(\varphi,X_A)$ is the $C^*$-subalgebra of $\Ll_A(F_X)$ generated by $\varphi_{\infty}(A)$ and $\{T_x \mid x \in X\}$. 

Consider the subalgebra $C^*(\varphi_{\infty}(A), \{T_xT_y^* \mid x,y \in X\}) \subseteq \Tt_X$.
In \cref{prop:cuntz-pimsner_characterisation} below, we show that
	\[
	(\unis{\Kk_A(X)}{\varphi}{A},\unif{\varphi}) \cong \Big(C^*\big(\varphi_{\infty}(A), \{T_xT_y^* \mid x,y \in X\}\big), \varphi_{\infty}\Big) .
	\]
\end{example}

For the remainder of this article we work exclusively with quotient fibrewise compactifications of the unified algebra. Like in the topological setting, larger fibrewise compactifications do exist. 

\begin{rmk}
	We outline a ``large'' fibrewise compactification analogous to that of \cref{rmk:maximalish_fw}. 	
	Fix $\varphi \in \Mor(A,B)$. Since $\varphi(A)$ is a subalgebra of $\Mm(B)$, we treat $\Mm(\varphi(A))$ as a subalgebra of $\Mm(B)$. 
	Let $D \coloneqq \Mm(\varphi(A)) +B 
\subseteq \Mm(B)$ and define $\delta \colon A \to D$ by $\delta(a) = \varphi(a)$ for all $a \in A$ (so $\delta$ and $\varphi$ only differ by codomain). Since $\varphi$ is a morphism, $\delta \in \Morp(A,C)$.
 The algebra $B$ is an ideal in $D$ and induces a morphism $\alpha_D \in \Mor(D,B)$ that coincides with the inclusion $D \hookrightarrow \Mm(B)$. It follows that $\alpha_D \circ \delta = \varphi$, and so \labelcref{itm:NCP1} is satisfied. Since $D \subseteq \Mm(B)$, we have $B^{\perp} = \{0\}$. By \cref{rmk:strict_nice}, $(D,\delta)$ is a strict fibrewise compactification of $\varphi$. 

We suspect that $(D,\delta)$ is universal among strict fibrewise compactifications of $\varphi$, but leave the question and its precise formulation open. 
\end{rmk}

\subsection{Quotient fibrewise compactifications of the unified algebra}\label{sec:quotient_fw}

Like in \cref{sec:sub-fibrewise} where we classified sub-fibrewise compactifications of the unified space, in this section we classify quotient fibrewise compactifications of the unified algebra. 
 \cref{lem:from_sets_to_ideals_and_back} motivates the following definition.  

\begin{dfn}
	Let $A$ and $B$ be $C^*$-algebras and let $\varphi \in \Mor(A,B)$. An ideal $J \trianglelefteq A$ is \emph{$\varphi$-proper} if $\varphi(J) \subseteq B$. We say that $J$ is \emph{$\varphi$-perfect} if $\varphi|_J$ is also injective.
\end{dfn}

In \cite[Definition~3.2]{Kat04cor}, Katsura  introduced the following ideal. 

\begin{dfn}
	Let $\varphi \in \Mor(A,B)$. We call the ideal
	\[
	\kat(\varphi) \coloneqq \pim(\varphi) \cap \ker(\varphi)^{\perp}
	\]
	of $A$ the \emph{Katsura ideal of $\varphi$}.
\end{dfn}

The following result is analogous to \cref{lem:f-reg_characterisation} and follows directly from the definitions of $\pim(\varphi)$ and $\kat(\varphi)$. 

\begin{lem}\label{lem:phi_proper_characterisation}
 Let $A$ and $B$ be $C^*$-algebras and let $\varphi \in \Mor(A,B)$. Suppose that $J \trianglelefteq A$. Then
 \begin{enumerate}
 	\item $J$ is $\varphi$-proper if and only if $J \subseteq \pim(\varphi)$; and
 	\item $J$ is $\varphi$-perfect if and only if $J \subseteq \kat(\varphi)$. 
 \end{enumerate}
 In particular, $\pim(\varphi)$ is the maximal $\varphi$-proper ideal and $\kat(\varphi)$ is the maximal $\varphi$-perfect ideal. 
\end{lem}

\begin{example}\label{ex:proper_the_same}
	If $f \colon X \to Y$ is a continuous map between locally compact Hausdorff spaces, then maximality of $\kat(f^*)$ and $\per(f)$ together with \cref{lem:from_sets_to_ideals_and_back} imply that $\kat(f^*) \cong C_0(\per(f))$. 
	More generally, \cref{lem:from_sets_to_ideals_and_back} implies that $U$ is $f$-proper if and only if $C_0(U)$ is $f^*$-proper, and $U$ is $f$-perfect if and only if $C_0(U)$ is $f^*$-perfect. 
\end{example}

Suppose that $J \trianglelefteq A$ is a $\varphi$-proper ideal. Since $\varphi(J) \subseteq B$, the subalgebra $0 \oplus J$ of $ \unis{B}{\varphi}{A}$ is an ideal. This allows for the following definition.

\begin{dfn}
Let $A$ and $B$ be $C^*$-algebras and let $\varphi \in \Mor(A,B)$. Suppose that $J \trianglelefteq A$ is a $\varphi$-proper ideal. Define
\[
\pers{B}{\varphi}{J}{A} \coloneqq \frac{\unis{B}{\varphi}{A}}{0 \oplus J},
\]
and define $\varphi_J \colon A \to \pers{B}{\varphi}{J}{A}$ by $\varphi_J  (a) = \unif{\varphi}(a) + 0 \oplus J = (\varphi(a),a + J)$.
\end{dfn}

 Analogously to \cref{rmk:unified_alg_alternate}, we may write
\begin{equation}\label{eq:quotient_description}
	\pers{B}{\varphi}{J}{A} = \{\varphi_J (a) + (b,J) \mid a \in A,\, b \in B \}.
\end{equation}
We show in \cref{prop:quotient_perfections_unified} below, that $(\pers{B}{\varphi}{J}{A},\varphi_J)$ is a fibrewise compactification of $\varphi$. First, we identify a copy of $B$ as an ideal in $\pers{B}{\varphi}{J}{A}$. 

\begin{lem}\label{lem:BJ_ideal}
Let $A$ and $B$ be $C^*$-algebras, let $\varphi \in \Mor(A,B)$, and let $J$ be a $\varphi$-proper
ideal. Then $ B_{J} \coloneqq \{(b,J)\in \pers{B}{\varphi}{J}{A}\mid b \in B\}$ is an ideal in $\pers{B}{\varphi}{J}{A}$ and $B_J \cong B$. 
\end{lem}

\begin{proof}
	The set $B_J$ is clearly a subalgebra of $\pers{B}{\varphi}{J}{A}$. Since
		 $B \trianglelefteq \Mm(B)$ and $J \trianglelefteq A$, we have $(m,a+J)(b,J) = (mb,J) \in B_J$ and $(b,J)(m,a+J) = (bm,J) \in B_J$ for all $(m,a+J) \in  \pers{B}{\varphi}{J}{A}$ and $(b,J) \in B_J$. So, $B_J$ is an ideal. If $(b,J) = 0$, then $b = 0$, so the map $b \mapsto (b,J)$ induces an isomorphism $B \cong B_J$.
\end{proof}

\begin{prop}\label{prop:quotient_perfections_unified}
	Let $A$ and $B$ be $C^*$-algebras, let $\varphi \in \Mor(A,B)$, and let $J$ be a $\varphi$-proper ($\varphi$-perfect)
	 ideal. 
	Then $(\pers{B}{\varphi}{J}{A},\varphi_J)$ is a fibrewise compactification (perfection) of $\varphi$. Moreover, every quotient fibrewise compactification (perfection) of $(\unis{B}{\varphi}{A},\unif{\varphi})$ is isomorphic to $(\pers{B}{\varphi}{J}{A},\varphi_J)$ for some $\varphi$-proper ($\varphi$-perfect) ideal $J \trianglelefteq A$.
\end{prop}

\begin{proof}
	By \cref{lem:BJ_ideal}, $B_J\cong B$ is an ideal in  $\pers{B}{\varphi}{J}{A}$. Let $\alpha \in \Mor (\pers{B}{\varphi}{J}{A},B)$ be the induced morphism. 	Since $\varphi_J$ is the composition of $\unif{\varphi}$ with the quotient map $\unis{B}{\varphi}{A} \twoheadrightarrow  \pers{B}{\varphi}{J}{A}$, we have $\varphi_J \in \Morp(A,\pers{B}{\varphi}{J}{A})$. For each $a \in A$, we have $\perf{\varphi}{J}(a)(b,J) = (\varphi(a)b,J)$ for all $b \in B$. Hence, $\alpha \circ \perf{\varphi}{J} = \varphi$, and so $(\pers{B}{\varphi}{J}{A},\varphi_J)$ satisfies \labelcref{itm:NCP1}. 
	
	For \labelcref{itm:NCP2}, we first show that $B_J^{\perp} = \{(0,a + J) \mid a \in \pim(\varphi) \}$. If $a \in \pim(\varphi)$, then for all $(b',J) \in B_J$ we have $(0,a+J)(b',J) = (0,J) =(b',J)( 0,a+J)$, so $(0,a + J) \in B_J^{\perp}$. On the other hand, if $(\varphi(a)-b,a +J) \in B_J^{\perp}$ for some $a \in A$ and $b \in B$, then for each $(b',J) \in B_J$ we have $0 = (\varphi(a)-b,a +J)(b',J)$. It follows that $\varphi(a)b' = bb'$ for all $b' \in B$, so $\varphi(a) = b$. Hence, $a \in \pim(\varphi)$.
	
	The map $\ol{\varphi_J} \colon \pim(\varphi) \to \pers{B}{\varphi}{J}{A}/B_J$ satisfies
	\[
	\ol{\varphi_J}(a) = (\varphi(a),a +J) + B_J  = (0, a + J) + B_J
	\]
	for all $a \in \pim(\varphi)$, so by our characterisation of $B_J^{\perp}$, \labelcref{itm:NCP2} is satisfied. Hence $(\pers{B}{\varphi}{J}{A},\varphi_J)$ is a fibrewise compactification of $\varphi$.
	If $J$ is also $\varphi$-perfect and $\varphi_J(a) = 0$, then $a \in J$ and $\varphi(a) = 0$, so $a = 0$. Hence, $\varphi_J$ is injective. 
	
	Now suppose that $(C,\psi)$ is a quotient fibrewise compactification of $(\unis{B}{\varphi}{A},\unif{\varphi})$, with quotient map $q \colon \unis{B}{\varphi}{A} \to C$. Let $B_C$ denote the ideal of $C$ that is isomorphic to $B$, and let $\alpha_C \in \Mor(C,B)$ be the induced morphism. Let $J \coloneqq \psi^{-1}(B_C) = \pim(\psi)$ and observe that $J \trianglelefteq A$. Then $\varphi(J) = \alpha_C \circ \psi(J) \subseteq B \subseteq \Mm(B)$, so $J \subseteq \pim(\varphi)$. By \cref{lem:phi_proper_characterisation}, $J$ is $\varphi$-proper. 
	
	We show that $\ker(q) = 0 \oplus J$. Fix $(m,a) \in \ker(q)$. 
	As $q \circ \unif{\varphi} = \psi$, we have
	\[
	0 = q(m,a) = q\big((m - \varphi(a),0) + \unif{\varphi}(a)\big) = q(m - \varphi(a),0) + \psi(a).
	\]
	By \cref{rmk:B_preservation}, $q$ induces an isomorphism $q|_B \colon B \oplus 0 \to B_C$, so $\psi(a) =  q(\varphi(a)-m,0) \in B_C$. That is, $a \in J$. Since $\alpha_C \circ q = \alpha \in \Mor(\unis{B}{\varphi}{A},B)$, 
	\[
	0 = \alpha_C \circ q(m,a) = \alpha(m,a) = m.
	\]
	That is, $(m,a) \in 0 \oplus J$, and so $\ker(q) \subseteq 0 \oplus J$. 

	Now, fix $(0,j) \in 0 \oplus J$. Since $q|_B \colon B \to B_C$ is an isomorphism and $q \circ \unif{\varphi} = \psi$,
	\[
	q(0,j) = q\big(\unif{\varphi}(j) - (\varphi(j),0)\big) = \psi(j) - q(\varphi(j),0) \in B_C
	\]
	and so
	\[
	 \alpha_C \circ q(0,j) = \alpha_C \circ \psi(j) - \alpha_C \circ q(\varphi(j),0) = \varphi(j) - \alpha(\varphi(j),0) = \varphi(j) - \varphi(j) = 0.
	\]
	By \cref{lem:annihilator_kernel}, we have $q(0,j) \in B_C^{\perp}$. Since $q(0,j) \in B_C \cap B_C^{\perp} = \{0\}$, we have $(0,j) \in \ker(q)$.  Hence, $\ker(q) = 0 \oplus J$, and so, $C \cong \pers{A}{\varphi}{J}{B}$.
	
	Now suppose that $\psi$ is injective. Fix $j \in J$ and $a \in \ker(\varphi)$. Since $aj \in J \cap \ker(\varphi)$, we have $\psi(aj) \in B_C$ and $\alpha_C (\psi(aj)) = \varphi(aj) = 0$, so by \cref{lem:annihilator_kernel} we have $\psi(aj) \in B_C^{\perp}$. Hence $\psi(aj) = 0$, and so $aj = 0$. A symmetric argument shows $ja= 0$, so $j \in \ker(\varphi)^{\perp}$. Hence, $J \trianglelefteq \kat(\varphi)$, so by \cref{lem:phi_proper_characterisation}, $J$ is $\varphi$-perfect.   
	
	Since $q \circ \unif{\varphi} = \psi$, the isomorphism between $C$ and $\pers{B}{\varphi}{J}{A}$ intertwines $\psi$ and $\perf{\varphi}{J}$.
\end{proof}

In the commutative setting we have the following consequence of \cref{prop:unified_cstar}.

\begin{cor}\label{cor:commutative_quotients_same}
	Let $f \colon X \to Y$ be a continuous map between locally compact Hausdorff spaces and suppose that $U \subseteq Y$ is $f$-proper. Then
	\[
	(\pers{C_0(X)}{f^*}{C_0(U)}{C_0(Y)}, \perf{(f^*)}{C_0(U)})\cong (C_0( \pers{X}{f}{U}{Y}), (\perf{f}{U})^*).
	\]
\end{cor}

\begin{proof}
	\cref{ex:proper_the_same} says that $C_0(U)$ is $f^*$-proper. 
	By \cref{lem:subperfections_closed}, the set $\iota_Y(U)$ is open in $\unis{X}{f}{Y}$, and $\pers{X}{f}{U}{Y} = \unis{X}{f}{(Y \setminus U)}$ is closed. \cref{prop:unified_cstar} implies that the diagram
	\[\begin{tikzcd}[ampersand replacement=\&]
		0 \& {C_0(U)} \& {\unis{C_0(X)}{f^*}{C_0(Y)}} \& {\pers{C_0(X)}{f^*}{C_0(U)}{C_0(Y)}} \& 0 \\
		0 \& {C_0(U)} \& {C_0(\unis{X}{f}{Y})} \& {C_0(\pers{X}{f}{U}{Y})} \& 0
		\arrow[from=1-2, to=2-2, equals]
		\arrow[from=1-1, to=1-2]
		\arrow[from=1-2, to=1-3]
		\arrow[from=1-3, to=1-4]
		\arrow[from=1-4, to=1-5]
		\arrow[from=2-1, to=2-2]
		\arrow[from=2-2, to=2-3]
		\arrow[from=2-3, to=2-4]
		\arrow[from=2-4, to=2-5]
		\arrow[from=1-3, to=2-3, "\cong"]
	\end{tikzcd}\]
	commutes. The isomorphism $\unis{C_0(X)}{f^*}{C_0(Y)} \cong C_0(\unis{X}{f}{Y})$ descends to the quotients. 
\end{proof}

Just as the unified algebra can be described in terms of extensions, so can quotient perfections of the unified algebra.

\begin{thm}\label{thm:unified_quotient}
	Let $A$ and $B$ be $C^*$-algebras, let $\varphi \in \Mor(A,B)$, and suppose that $J$ is $\varphi$-proper. The quotient of $\pers{B}{\varphi}{J}{A}$ by $B_J = \{(b,J) \mid b \in B\}$ is isomorphic to $A/J$. Moreover, the diagram
	\begin{equation}\label{eq:quotient_unified}
		\begin{tikzcd}[ampersand replacement=\&, column sep={6em,between origins}]
			0 \arrow[r] 
			\& J\arrow[r,"\iota_A"] \arrow[d,"\restr{\varphi}{J}"] 
			\& A \arrow[r, "q_A"] \arrow[d, "\varphi_J "]
			\& A/J \arrow[r] \arrow[d, equal] 
			\& 0 
			\\
			0 \arrow[r] 
			\& B \arrow[r] 
			\& \pers{B}{\varphi}{J}{A} \arrow[r]  
			\& A/J \arrow[r]                              
			\& 0
		\end{tikzcd}
	\end{equation}
	of extensions, commutes. 
	The $C^*$-algebra $\pers{B}{\varphi}{J}{A}$ is universal for \labelcref{eq:quotient_unified} in the following sense: if $C$ is a $C^*$-algebra containing $B$ as an ideal, $\alpha_C \in \Mor(C,B)$ is the induced morphism, and $\pi \colon A \to C$ is a $*$-homomorphism such that $\alpha_C \circ \pi = \varphi$ and $\pi|_J = \varphi|_J$, then there is a unique $*$-homomorphism $\sigma_C \colon \pers{B}{\varphi}{J}{A} \to C$, satisfying
	\begin{equation}\label{eq:pi_plus}
		\sigma_C(b + \varphi(a),a+J) = b + \pi(a)
	\end{equation}
	for all $b \in B$ and $a \in A$, that makes the diagram
\begin{equation}\label{eq:quotient_unified_universal}
	\begin{tikzcd}[ampersand replacement=\&,column sep={6em,between origins}]
		\& 0 \& J \& A \& {A/J} \& 0 \\
		\& 0 \& B \& {\pers{B}{\varphi}{J}{A}} \& {A/J} \& 0 \\[-0.7em]
		0 \& B \& C \& {C/B} \& 0
		\arrow[from=1-2, to=1-3]
		\arrow[from=1-3, to=1-4, "\iota_A"]
		\arrow[from=1-4, to=1-5, "q_A"]
		\arrow[from=1-5, to=1-6]
		\arrow[from=1-3, to=2-3,"\varphi|_J"]
		\arrow[from=1-4, to=2-4,"\varphi_J "]
		\arrow[from=1-5, to=2-5, equals]
		\arrow[from=2-5, to=3-4]
		\arrow[from=2-4, to=3-3,"\sigma_C"]
		\arrow[from=2-3, to=3-2, equals]
		\arrow[from=3-1, to=3-2]
		\arrow[from=3-2, to=3-3,"\iota_C"' ]
		\arrow[from=3-3, to=3-4, "q_C"']
		\arrow[from=3-4, to=3-5]
		\arrow[from=2-2, to=2-3]
		\arrow[from=2-3, to=2-4]
		\arrow[from=2-4, to=2-5]
		\arrow[from=2-5, to=2-6]
		\arrow[from=1-3, to=3-2, "\pi|_J"', crossing over]
		\arrow[from=1-4, to=3-3, "\pi"', crossing over]
		\arrow[from=1-5, to=3-4, crossing over]
	\end{tikzcd}
\end{equation}
commute. In particular, $\pers{B}{\varphi}{J}{A}$ is the unique $C^*$-algebra making \labelcref{eq:quotient_unified} commute. 
\end{thm}

\begin{proof}
 Consider the surjection $q \colon \pers{B}{\varphi}{J}{A} \to A/J$ given by $q(m,a +J) = a +J$. If $(m,a+J) \in \ker(q)$, then $a \in J$. Since $m - \varphi(a) \in B$, we have $m \in B$. It follows that $\ker(q) = \{(b,J) \mid b \in B\} = B_J$. By \cref{prop:quotient_perfections_unified}, we have $B \cong B_J$, yielding the bottom exact sequence of \labelcref{eq:quotient_unified}. It is routine to verify that \labelcref{eq:quotient_unified} commutes. 
	
	For the universal property, we note that since $\varphi|_J = \pi|_J$, the linear map $\sigma_C \colon \pers{B}{\varphi}{J}{A} \to C$ given by \labelcref{eq:pi_plus} is well-defined. 
	 The map $\sigma_C$ is clearly $*$-preserving. Using that $\alpha_C \circ \pi = \varphi$ at the third equality, for all $a_1,a_2 \in A$ and $b_1,b_2 \in B$ we have
	\begin{align*}
	&\sigma_C\big((b_1 + \varphi(a_1) , a_1 + J)(b_2 + \varphi(a_2) , a_2 + J)\big) \\
	&\quad= \sigma_C(b_1b_2 + \varphi(a_1)b_2 + b_1 \varphi(a_2) + \varphi(a_1a_2), a_1a_2 + J) \\
	&\quad= b_1b_2 + \varphi(a_1)b_2 + b_1 \varphi(a_2) + \pi(a_1a_2)\\
	&\quad= b_1b_2 + \pi(a_1)b_2 + b_1 \pi(a_2) + \pi(a_1a_2)\\
	&\quad=(b_1 + \pi(a_1))(b_2 + \pi(a_2))\\
	&\quad= \sigma_C(b_1 + \varphi(a_1) , a_1 + J) \sigma_C(b_2 + \varphi(a_2) , a_2 + J).
	\end{align*}
	So $\sigma_C$ is a $*$-homomorphism. 
	The map $\sigma_C$ restricts to an isomorphism from $B_J$ to $B \subseteq C$, so $\sigma_C$ descends to a $*$-homomorphism 
	$\ol{\sigma_C} \colon A/J \to C/B$. For each $a \in A$, we have
	\[
	{\sigma_C}  \circ \varphi_J (a) = {\sigma_C} (\varphi(a), a + J) = \pi (a),
	\]
	and it follows that the diagram \labelcref{eq:quotient_unified_universal} commutes.

	For uniqueness, suppose that $\tau \colon \pers{B}{\varphi}{J}{A} \to C$ is another map that makes the analogous diagram to \labelcref{eq:quotient_unified_universal} commute. Commutativity implies that
$
	\tau(\varphi_J (a)) = \pi(a)$ and $\tau(b,J) = b
$
	for all $a \in A$ and $b \in B$. Consequently, $\tau(\varphi_J (a) + (b,J)) = \pi(a) + b  = \sigma_C (\varphi_J (a) + (b,J))$, so \labelcref{eq:quotient_description} gives $\tau = \sigma_C$. The final uniqueness statement follows from the universality of $\pers{B}{\varphi}{J}{A}$. 
\end{proof}

\begin{rmk}
The setup of \cref{thm:unified_quotient} is similar to Diagram~III of \cite[Section~4.2]{ELP99}: however, we do not assume that $\varphi|_J \in \Morp(J,B)$. That is, we do not assume that $\varphi(J)B$ is dense in $B$. In \cite[Remark~4.4]{ELP99}, there is a counterexample showing that the properness assumption is critical to their universal characterisation. For us, this is made up for by additional assumptions, including the fact $\varphi|_J$ is the restriction of a morphism $\varphi \in \Mor(A,B)$. 
\end{rmk}

Since $\kat(\varphi)$ is the maximal $\varphi$-perfect ideal, we make the following definition in light of \cref{prop:quotient_perfections_unified}.

\begin{dfn}
	We call $(\mins{B}{\varphi}{A}, \minf{\varphi}) \coloneq (\pers{B}{\varphi}{\kat(\varphi)}{A}, \perf{\varphi}{J_{\kat(\varphi)}})$ the \emph{minimal perfection} of $\varphi$. 
\end{dfn}

Although examples of the algebras $\pers{B}{\varphi}{J}{A}$ may seem obscure in the noncommutative setting, they appear naturally when representing $C^*$-correspondences.

\begin{dfn}[{\cite[Definition~2.1]{Kat04cor}}]\label{dfn:corr_reps}
	Let $A$ be a $C^*$-algebra and let $(\varphi,X_A)$ be an $A$--$A$-correspondence (see \cref{ex:correspondences}). A \emph{representation} of $(\varphi,X_A)$ in a $C^*$-algebra $B$ is a pair $(\rho_A,\rho_X)$ consisting of a $*$-homomorphism $\rho_A \colon A \to B$ and a linear map $\rho_X\colon X \to B$ satisfying
	\[
	\rho_X(\varphi(a_1)x_1a_2) = \rho_A(a_1) \rho_X(x_1) \rho_A(a_2) \quad \text{ and } \quad \rho_A(\langle x_1 | x_2 \rangle) = \rho_X(x_1)^* \rho_X(x_2)
	\]
	for all $a_1,a_2 \in A$ and $x_1,x_2 \in X_A$. 		A representation $\rho$ is \emph{faithful} if $\rho_A$ is injective, in which case so is $\rho_X$. 
\end{dfn}

	Fix a representation $\rho = (\rho_A, \rho_X)$ of $(\varphi,X_A)$ in a $C^*$-algebra $B$.
	By \cite[Lemma~2.2]{KPW98}, there is a well-defined $*$-homomorphism $\rho_A^{(1)} \colon \Kk_A(X) \to B$ satisfying \begin{equation} \label{eq:compact_formula}
		\rho_A^{(1)}(\Theta_{x,y}) = \rho_X(x)\rho_X(y)^*
	\end{equation} for all rank-1 operators $\Theta_{x,y} \in \Kk_A(X)$. If $\rho$ is faithful, then so is $\rho_A^{(1)}$. 
 If $\rho$ is faithful, then by \cite[Proposition~3.3]{Kat04cor},
	\begin{equation}\label{eq:rep_ideal}
		I_{\rho} \coloneqq \big\{ a \in A \bigm| \rho_A(a) \in \rho_A^{(1)} (\Kk_A(A)) \big\}
	\end{equation}
	is an ideal contained in $\kat(\varphi)$ (so $I_{\rho}$ is $\varphi$-perfect by \cref{lem:phi_proper_characterisation}) and $\rho_A(a) = \rho_A^{(1)} \circ \varphi(a)$ for all $a \in I_{\rho}$.
	In \cite[Proposition 5.12]{Kat04cor}, it is noted that there is a commuting diagram 
	\begin{equation}\label{eq:katsura_diagram_1}
		\begin{tikzcd}[ampersand replacement=\&,column sep=small]
			0 \arrow[r] 
			\& I_{\rho} \arrow[r] \arrow[d,"\varphi"] 
			\& A \arrow[r] \arrow[d, "\rho_A "]
			\& A/I_{\rho} \arrow[r] \arrow[d, equal] 
			\& 0 
			\\
			0 \arrow[r] 
			\& \Kk_A(X) \arrow[r, "\rho_A^{(1)}"] 
			\& \rho_A^{(1)} (\Kk_A(X)) + \rho_A(A) \arrow[r]  
			\& A/I_{\rho} \arrow[r]                              
			\& 0
		\end{tikzcd}
	\end{equation}
	with exact rows. 
	The universal property of \cref{thm:unified_quotient} implies that
	\[
 (\pers{\Kk_A(X)}{\varphi}{I_{\rho}}{A},\varphi_{I_{\rho}}) \cong 	\big(\rho_A^{(1)} (\Kk_A(X)) + \rho_A(A) , \rho_A \big).
	\]
	We show in \cref{prop:cuntz-pimsner_characterisation} that every quotient perfection of $(\unis{\Kk_A(X)}{\varphi}{A},\unif{\varphi})$ arises from a faithful representation of $(\varphi,X_A)$. To prove this we recall the following definition. 
	
	\begin{dfn}[{cf. \cite[Definition~3.4]{Kat04cor}}]
		Let $J$ be a $\varphi$-proper ideal of a $C^*$-algebra $A$. A representation $(\rho_A,\rho_X)$ of an $A$--$A$-correspondence $(\varphi,X)$ in a $C^*$-algebra $B$ is said to be \emph{$J$-covariant} if $\rho_A^{(1)} \circ \varphi(a) = \rho_A(a)$ for all $a \in J$.
	\end{dfn}

	 By \cite[Theorem~2.19]{MS98}, if $J$ is a $\varphi$-proper ideal of $A$, then there is a universal $J$-covariant representation $\tau^J = (\tau_A^J,\tau_X^J)$ of $(\varphi,X_A)$ in a $C^*$-algebra $\Oo_{X,J}$, called the \emph{$J$-relative Cuntz--Pimsner algebra} of $(\varphi,X)$ (cf. \cite{Kat04cor,Pim97}). 
	 If $J = \{0\}$, then $\Oo_{X,\{0\}}$ is isomorphic to the Toeplitz algebra $\Tt_X$ of \cref{ex:toeplitz}. In the language of \cref{ex:toeplitz}, we have	 
	  $\tau_A^J =\varphi_{\infty}$ and $\tau_X^J(x) = T_x$ for all $x \in X$. 
	 If $J = \kat(\varphi)$, then $\Oo_X \coloneqq \Oo_{X,\kat(\varphi)}$ is called the \emph{Cuntz--Pimsner algebra} of $(\varphi,X)$.

	\begin{prop}\label{prop:cuntz-pimsner_characterisation}
		Let $(\varphi,X_A)$ be a $C^*$-correspondence over $A$ and let $J \trianglelefteq A$ be a $\varphi$-perfect ideal. Let $\tau^J$ be the universal representation of $(\varphi,X_A)$ in $\Oo_{X,J}$. 		
		 Then
		 \[
		 (\pers{\Kk_A(X)}{\varphi}{J}{A}, \perf{\varphi}{J}) 
		 \cong 
		 \big( (*\tau_A^J)^{(1)} (\Kk_A(X)) + \tau_A^J(A), \tau_A^J \big).
		 \]
		 In particular,  the unified algebra $\unis{\Kk_A(X)}{\varphi}{A}$ is isomorphic to a subalgebra of the Toeplitz algebra $\Tt_X$, and the minimal perfection $\mins{\Kk_A(X)}{\varphi}{A}$ is isomorphic to a subalgebra of the Cuntz--Pimsner algebra $\Oo_X$.
	\end{prop}

	\begin{proof}
		By \cite[Theorem~2.19]{MS98}, the ideal		
		 $I_{\tau^J}$ of \labelcref{eq:rep_ideal} is equal to $J$. It follows from \labelcref{eq:katsura_diagram_1} and \cref{thm:unified_quotient} that $\pers{\Kk_A(X)}{\varphi}{J}{A}$ is isomorphic to the subalgebra $(\tau_A^J)^{(1)} (\Kk_A(X)) + \tau_A^J(A)$ of $\Oo_{X,J}$. 
	\end{proof}

	\begin{rmk}
	If $J$ is a $\varphi$-proper ideal of $A$ which is not $\varphi$-perfect then we can no longer use \cite[Proposition~3.3]{Kat04cor}. It is unclear how to characterise $(\pers{\Kk_A(X)}{\varphi}{J}{A},\varphi_J)$ in this setting. 
	\end{rmk}

\subsection{Composing quotient perfections of the unified algebra}	
\label{sec:nc_composition}

Suppose that $\varphi \in \Mor(A,B)$ and $\psi  \in \Mor(B,C)$. As in the topological setting of \cref{sec:commutative_composition}, there are at least three ways to compose the unified algebra construction. In what follows we fix an approximate identity $(b_{\lambda})$ for $B$.

 For the simplest composition, we extend $\psi$ to the morphism $\ol{\psi} \in \Mor(\Mm(B),C)$ satisfying $\ol{\psi}(m)c = \lim_{\lambda}\psi(mb_{\lambda})c$ for all $m \in \Mm(B)$ and $c \in C$. We then form the unified algebra $(\unis{C}{\ol{\psi} \circ \varphi}{A}, \unif{(\ol{\psi} \circ \varphi)})$. Like in the topological setting, this construction loses information about the intermediate algebra $B$.

For the second approach, we first form $(\unis{B}{\varphi}{A}, \unif{\varphi})$. Since $B$ is an ideal in $\unis{B}{\varphi}{A}$ it induces a morphism $\alpha_B \in \Mor(\unis{B}{\varphi}{A},B)$ satisfying $\alpha_B(m,a) = m$. Then $\ol{\psi} \circ \alpha_B \in \Mor(\unis{B}{\varphi}{A},C)$, so we may form $(\unis{C}{\ol{\psi} \circ \alpha_B}{(\unis{B}{\varphi}{A})}, \unif{(\ol{\psi} \circ \alpha_B)})$.
 The composition $\unif{(\ol{\psi} \circ \alpha_B)} \circ \unif{\varphi} \in \Morp(A,\unis{C}{\ol{\psi} \circ \alpha_B}{(\unis{B}{\varphi}{A})})$ is injective. 

For the third approach, we first form $(\unis{C}{\psi}{B},\unif{\psi})$. Since $\unif{\psi} \in \Morp(B,\unis{C}{\psi}{B})$, it extends to a unique unital $*$-homomorphism $\ol{\unif{\psi}} \colon \Mm(B) \to \Mm(\unis{C}{\psi}{B})$ such that for all $m \in \Mm(B)$, we have
\begin{equation}\label{eq:map3}
	\ol{\unif{\psi}}(m) (n,b) = \lim_{\lambda} (\psi(m b_{\lambda}) n, m b_{\lambda} b) 
	=(\ol{\psi}(m) n, m b)
\end{equation}
for all $(n,b) \in \unis{C}{\psi}{B}$. We then form the composition $\ol{\unif{\psi}} \circ \varphi \in \Mor(A, \unis{C}{\psi}{B})$ and its unified algebra $(\unis{(\unis{C}{\psi}{B})}{\ol{\unif{\psi}} \circ \varphi}{A},\unif{(\ol{\unif{\psi}} \circ \varphi)})$.

Using thick arrows to denote morphisms, and thin arrows to denote proper morphisms, the algebras involved in the second and third constructions can be assembled into a diagram
\[\begin{tikzcd}[ampersand replacement=\&,cramped,column sep=35pt]
	\& {\unis{B}{\varphi}{A}} \& {\unis{C}{\ol{\psi} \circ \alpha_B}{(\unis{B}{\varphi}{A}})} \\
	A \& B \& C \\
	{\unis{(\unis{C}{\psi}{B})}{\ol{\unif{\psi}} \circ \varphi}{A}} \& {\unis{C}{\psi}{B}}
	\arrow["{\unif{(\ol{\unif{\psi}} \circ \varphi)}}", from=1-2, to=1-3]
	\arrow["\alpha_B", Rightarrow, from=1-2, to=2-2]
	\arrow["\alpha_C", Rightarrow, from=1-3, to=2-3]
	\arrow["\unif{\varphi}",curve={height=-15pt}, from=2-1, to=1-2]
	\arrow["\varphi", Rightarrow, from=2-1, to=2-2]
	\arrow["{\unif{(\ol{\unif{\psi}} \circ \varphi)}}"', from=2-1, to=3-1]
	\arrow["\psi", Rightarrow, from=2-2, to=2-3]
	\arrow["\unif{\psi}", from=2-2, to=3-2]
	\arrow["{\alpha_{\unis{C}{\psi}{B}}}", Rightarrow, from=3-1, to=3-2]
	\arrow["\alpha_C",curve={height=15pt}, Rightarrow, from=3-2, to=2-3]
\end{tikzcd}
\]
that commutes with respect to composition of morphisms. 
The maps labelled by $\alpha$ are induced by inclusions of ideals. This diagram should be compared to \labelcref{eq:big_diagram}.

We have the following noncommutative analogue of \cref{lem:compositions_associative} which says that the second and third constructions, outlined above, are actually the same up to isomorphism. 

\begin{lem}\label{lem:alg_unified_composition}
	Let $A$, $B$, and $C$ be $C^*$-algebras, let $\varphi \in \Mor(A,B)$, and let $\psi \in \Mor(B,C)$. Let $\ol{\psi} \colon \Mm(B) \to \Mm(C)$ be the unital $*$-homomorphism induced by $\psi$. Both $\unis{C}{\ol{\psi} \circ \alpha_B}{(\unis{B}{\varphi}{A})}$ and $\unis{(\unis{C}{\psi}{B})}{\ol{\unif{\psi}} \circ \varphi}{A}$ are  isomorphic to the $C^*$-subalgebra 
	\begin{align*}
			\unis{C}{\psi}{\unis{B}{\varphi}{A}}  &\coloneqq \big\{(c + \psi(b) + \ol{\psi}(\varphi(a)), b + \varphi(a),a) \mid a \in A, b \in B, c \in C\big\} \\
			&= \{(m,n,a) \in \Mm(C) \oplus \Mm(B) \oplus A \colon n - \varphi(a) \in B, m - \ol{\psi}(n) \in C\}
	\end{align*}
	of $\Mm(C) \oplus \Mm(B) \oplus A$.
	Under this isomorphism, the proper morphisms $\unif{(\ol{\psi} \circ \alpha_B)} \circ \unif{\varphi}$ and $\unif{(\ol{\unif{\psi}} \circ \varphi)}$ are identified with the map $A \to \unis{C}{\psi}{\unis{B}{\varphi}{A}} $ given by
	\[
	a \mapsto \big(\ol{\psi} (\varphi(a)), \varphi(a), a\big).
	\]
\end{lem}

\begin{proof}
	 We have
	\begin{align*}
		 \unis{C}{\ol{\psi} \circ \alpha_B}{(\unis{B}{\varphi}{A})} 
	 &= \{ (m,x) \in \Mm(C) \oplus (\unis{B}{\varphi}{A}) \mid m - \ol{\psi} \circ \alpha_B(x) \in C\}\\
	 &= \{(c + \ol{\psi} \circ \alpha_B(x), x) \mid c \in C, x \in \unis{B}{\varphi}{A}\}\\
	 &= \big\{ \big(c + \ol{\psi} \circ \alpha_B(b + \varphi(a),a), (b + \varphi(a),a) \big) \bigm| c \in C, b \in B, a \in A\big \}\\
	 &= \big\{ \big(c + \ol{\psi} (b) + \ol{\psi}(\varphi(a)), (b + \varphi(a),a)) \bigm| c \in C, b \in B, a \in A\big \}\\
	 &\cong \unis{C}{\psi}{\unis{B}{\varphi}{A}} .
	\end{align*}
	 For each $a \in A$,
	\[
	\unif{(\ol{\psi} \circ \alpha_B)} \circ \unif{\varphi} (a) 
	= \unif{(\ol{\psi} \circ \alpha_B)} (\varphi(a),a)
	= \big(\ol{\psi}(\varphi(a)), (\varphi(a),a) \big).
	\]
	For the second algebra, 
	\begin{align*}
		\unis{(\unis{C}{\psi}{B})}{\ol{\unif{\psi}} \circ \varphi}{A} 
		&=
		\{(x,a) \in \Mm(\unis{C}{\psi}{B}) \oplus A \mid x - \ol{\unif{\psi}} \circ \varphi(a) \in \unis{C}{\psi}{B}\}\\
		&= 
		\big\{
		\big((c + \psi(b),b) + \ol{\unif{\psi}}  (\varphi(a)),a\big) \bigm| c \in C, b \in B, a \in A \big\}. 
	\end{align*}
	Fix $c \in C$, $b \in B$, and $a \in A$. By \labelcref{eq:map3}, as a multiplier on $\unis{C}{\psi}{B}$, we have
	\begin{align*}
		\big((c + \psi(b),b) + \ol{\unif{\psi}} (\varphi(a))\big)(n,b') 
		&= ((c + \psi(b) + \ol{\psi}(\varphi(a))) n, (b+\varphi(a))b') \\
		&= \big(c + \psi(b) + \ol{\psi}(\varphi(a)), b + \varphi(a)\big) (n,b')
	\end{align*}
	for all $(n,b') \in \unis{C}{\psi}{B}$. So,
	\begin{align*}
		\unis{(\unis{C}{\psi}{B})}{\ol{\unif{\psi}} \circ \varphi}{A}
		= \big\{ \big((c + \psi(b) + \ol{\psi}(\varphi(a)), b + \varphi(a)), a \big)  \bigm| c \in C, b \in B, a \in A\big\}
	\end{align*}
	is isomorphic to $\unis{C}{\psi}{\unis{B}{\varphi}{A}} $. For each $a \in A$, we have
	\[
	\unif{(\ol{\unif{\psi}} \circ \varphi)}(a)
	= \big(\ol{\unif{\psi}} (\varphi(a)),a\big)
	= \big((\ol{\psi}(\varphi(a)),\varphi(a)),a\big). \qedhere
	\]
\end{proof}

We also have the following analogue of \cref{lem:compositions_associative_sub}.

\begin{lem}\label{lem:compositions_associative_quotient}
	Let $A$, $B$, and $C$ be $C^*$-algebras, let $\varphi \in \Mor(A,B)$, and let $\psi \in \Mor(B,C)$. Suppose that $J \trianglelefteq A$ is $\varphi$-proper and $I \trianglelefteq B$ is $\psi$-proper. Let $q^J \colon \unis{B}{\varphi}{A} \to \pers{B}{\varphi}{J}{A}$ and $q^I \colon \unis{C}{\psi}{B} \to \pers{C}{\psi}{I}{B}$ denote the corresponding quotient maps. 
	\begin{enumerate}
		\item \label{itm:tilde_descends}  The map $\ol{\psi} \circ \alpha_B \in \Mor (\unis{B}{\varphi}{A},C)$ descends to a well-defined morphism $\decor{\psi} \in \Mor(\pers{B}{\varphi}{J}{A},C)$ such that $\decor{\psi} \circ q^J = \ol{\psi} \circ \alpha_B$, and $q^J(I \oplus 0)$ is $\decor{\psi}$-proper (if $I$ is $\varphi$-perfect, then $q^J(I \oplus 0)$ is $\decor{\psi}$-perfect). 
		\item \label{itm:bar_extends} The map $\psi_I \in \Morp( B , \pers{C}{\psi}{I}{B})$ extends to a unique unital $*$-homomorphism $\ol{\psi_I} \colon \Mm(B) \to  \Mm(\pers{C}{\psi}{I}{B})$ and $J$ is $(\ol{\psi_I} \circ \varphi)$-proper (if $J$ is $\varphi$-perfect and $I$ is $\psi$-perfect, then $J$ is $(\ol{\psi_I} \circ \varphi)$-perfect).
		\item \label{itm:composition_quotient_algebras}The algebras $
		\pers{C}{\decor{\psi}}{q^J(I \oplus 0)}{(\pers{B}{\varphi}{J}{A})}
		$ and $
		\pers{(\pers{C}{\psi}{I}{B})}{\ol{\psi_I} \circ \varphi}{J}{A}
		$  
		are isomorphic to 
		\begin{align*}
			\pers{C}{\psi}{I}{\pers{B}{\varphi}{J}{A}}
			&\coloneqq \frac {\unis{C}{\psi}{\unis{B}{\varphi}{A}}}{0 \oplus I \oplus J}\\
			= \Big\{(m,&n+I,a+J) \in \Mm(C) \oplus \frac{\Mm(B)}{I} \oplus \frac{\Mm(A)}{J} \Bigm|  n - \varphi(a) \in B, m - \ol{\psi}(n) \in C\Big\}.
		\end{align*}

		\item \label{itm:quotient_alg_maps} Under the isomorphisms of \labelcref{itm:composition_quotient_algebras} the proper morphisms $\perf{\decor{\psi}}{q^J(I \oplus 0)} \circ \varphi_J$ and $\perf{(\ol{\psi_I} \circ \varphi)}{J}$ are identified with the map $A \to \pers{C}{\psi}{I}{\pers{B}{\varphi}{J}{A}}$ given by \[a \mapsto \big(\ol{\psi} (\varphi(a)), \varphi(a)+I, a+J\big).\]
	\end{enumerate}
\end{lem}

\begin{proof}
 \labelcref{itm:tilde_descends} Observe that $\ol{\psi} \circ \alpha_B(0 \oplus J) = 0$, so $\ol{\psi} \circ \alpha_B$ descends to a morphism $\decor{\psi} \in \Mor( \pers{B}{\varphi}{J}{A},C)$ satisfying $\decor{\psi} \circ q^J = \ol{\psi} \circ \alpha_B$. Since $I$ is $\psi$-proper, the ideal $q^J(I \oplus 0)$ is $\decor{\psi}$-proper. Suppose that $I$ is $\psi$-perfect and $\decor{\psi}(b,J) = 0$ for some $b \in I$. Then $\ol{\psi}(b) = 0$. That is, $\lim_{\lambda} \psi(b b_{\lambda})c = 0$ for all $c \in C$ and approximate identities $(b_{\lambda})$ for $B$. Since $b b_{\lambda} \in I$, we have $b b_{\lambda} = 0$ for all $\lambda$, so $b = 0$. Hence, $q^J(I \oplus 0)$ is $\decor{\psi}$-perfect.
	
	 \labelcref{itm:bar_extends} Since $\psi_I$ is proper, it extends to a unique unital $*$-homomorphism $\ol{\psi_I} \colon \Mm(B) \to  \Mm(\pers{C}{\psi}{I}{B})$ such that for any $m \in \Mm(B)$ and any approximate identity $(b_\lambda)$ of $B$, we have
	\begin{equation}\label{eq:overline_psi_I}
		\ol{\psi_I}(m)(n,b+I) 
	= \lim_{\lambda} \psi_I(mb_{\lambda})(n,b+I) 
	= \lim_{\lambda} (\psi(mb_{\lambda})n, mb_{\lambda}b + I)
	= (\ol{\psi}(m)n,mb +I)
	\end{equation}
	 for all $(n,b+I)  \in \pers{C}{\psi}{I}{B}$.
	  Since $J$ is $\varphi$-proper, $\varphi(J) \subseteq B$, so $\ol{\psi_I} \circ \varphi(J) \subseteq \pers{C}{\psi}{I}{B}$. That is, $J$ is $(\ol{\psi_I} \circ \varphi)$-proper. 
	  
	  Suppose that $J$ is $\varphi$-perfect, $I$ is $\psi$-perfect, and $a \in J$ is such that  \[0 = \ol{\psi_I}(\varphi(a))(n,b+I) = (\ol{\psi}(\varphi(a))n, \varphi(a)b + I)\]
	   for all $(n,b+I) \in \pers{C}{\psi}{I}{B}$. Then $\ol{\psi} (\varphi(a))n = 0$ for all $n \in C$ and $\varphi(a)b \in I$ for all $b \in B$. Approximate identity arguments show that $\ol{\psi}(\varphi(a)) = 0$ and $\varphi(a) \in I$. 
	   As $I$ is $\psi$-perfect, $\ol{\psi}|_{I}$ is injective. Since  $\ol{\psi}(\varphi(a)) = 0$ and $ \varphi(a) \in I$, we have $\varphi(a) = 0$. Since $J$ is $\varphi$-perfect, $a = 0$. That is, $J$ is $(\ol{\psi_I} \circ \varphi)$-perfect. 
	 
	 \labelcref{itm:composition_quotient_algebras} The argument is similar to that of \cref{lem:compositions_associative}, so we omit some detail including details of the isomorphism theorems invoked.   
	 By definition,
	 \[
	 \pers{C}{\decor{\psi}}{}{(\pers{B}{\varphi}{J}{A})}
	 = \{(c + \decor{\psi} (t), t )\mid c \in C, t\in \pers{B}{\varphi}{J}{A} \} \subseteq \Mm(C) \oplus (\pers{B}{\varphi}{J}{A}).
	 \]
	 If $t\in \pers{B}{\varphi}{J}{A}$, then $t = (b + \varphi(a), a + J)$ for some $a \in A$ and $b \in B$. Since $\decor{\psi} \circ q^J = \ol{\psi} \circ \alpha_B$, we have 
	 $
	 \decor{\psi}(t) = \decor{\psi}(b + \varphi(a), a + J) = \ol{\psi}\circ \alpha_B( b + \varphi(a), a) = \psi(b) + \ol{\psi}(\varphi(a)).
	 $
	 Hence,
	 \[
	 \pers{C}{\decor{\psi}}{}{(\pers{B}{\varphi}{J}{A})} 
	 =  \big\{\big(c + \psi(b) + \ol{\psi}(\varphi(a)), b + \varphi(a), a + J \big) \bigm| a \in A, b \in B, c \in C \big\}.
	 \]
	 Taking the quotient by $q^J(I \oplus 0)$ gives $\pers{C}{\decor{\psi}}{q^J(I \oplus 0)}{(\pers{B}{\varphi}{J}{A})} \cong \pers{C}{\psi}{I}{\pers{B}{\varphi}{J}{A}}$.
	 
	 By \labelcref{eq:overline_psi_I}, for $a \in A$ we have $\ol{\psi_I}(\varphi(a))(n,b+I) = (\ol{\psi}(\varphi(a)n), \varphi(a)b + I)$ for all $(n,b+I) \in \pers{C}{\psi}{I}{B}$. Using this at the second equality, as subsets of $\Mm(\pers{C}{\psi}{I}{B}) \oplus A/J$, we have	 
	 \begin{align*}
	 	\pers{(\pers{C}{\psi}{I}{B})}{\ol{\psi_I} \circ \varphi}{J}{A}
	 	&= \big\{\big((c + \psi(b),b+I) + \ol{\psi_I} \circ \varphi(a),a+J\big) \bigm| a \in A, b \in B, c \in C\big\}\\
	 	&= \big\{\big((c + \psi(b) + \ol{\psi}(\varphi(a)), b + \varphi(a)+ I), a+ J) \big) \bigm|  a \in A, b \in B, c \in C\big\}
	 \end{align*}
	 which is isomorphic to $\pers{C}{\psi}{I}{\pers{B}{\varphi}{J}{A}}$.
	
	\labelcref{itm:quotient_alg_maps} Analogously to the proof of \cref{lem:compositions_associative}, this follows from \labelcref{itm:composition_quotient_algebras} and some definition chasing. 
\end{proof}

\subsection{Generalised limits of \texorpdfstring{$C^*$}{C*}-algebras}\label{sec:nc_limits}

Let $(A_i,\varphi_i)_{i \in \NN}$ be a directed sequence of $C^*$-algebras $A_i$ and $*$-homomorphisms $\varphi_i \colon A_i \to A_{i+1}$. Its direct limit $\varinjlim(A_i,\varphi_i)$ is the unique $C^*$-algebra (up to isomorphism) with $*$-homomorphisms $\mu_k \colon A_k \to \varinjlim(A_i,\varphi_i)$ satisfying $\mu_{k+1} \circ \varphi_k = \mu_k$ for each $k \in \NN$, that is universal: if $B$ is another $C^*$-algebra with $*$-homomorphisms $\nu_k \colon A_k \to B$ satisfying $\nu_{k+1} \circ \varphi_k = \nu_k$ for all $k \in \NN$, then there exists a unique $*$-homomorphism $\nu \colon \varinjlim(A_i,\varphi_i) \to B$ such that $\nu \circ \mu_k = \nu_k$ for all $k \in \NN$. 

The direct limit of $(A_i,\varphi_i)_{i \in \NN}$ can be described explicitly. For each $k \in \NN$, let $\mu_k \colon A_k \to \prod_{n \in \NN} 
A_n / \bigoplus_{n \in \NN} A_n$ be the $*$-homomorphism given by
\[
\textstyle
\mu_k(a_k) = \Big(0, \cdots , 0 , a_k , \varphi_k(a_k) , \varphi_{k+1} \circ \varphi_k (a_k) , \varphi_{k+2} \circ \varphi_{k+1} \circ \varphi_k(a_k), \cdots\Big) +  \bigoplus_{n \in \NN} A_n
\]
for all $a_k \in A_k$. Then $\varinjlim(A_i,\varphi_i) = \overline{\bigcup_{k \in \NN}\mu_k(A_k)}$ by \cite[Proposition~6.2.4]{RLL00}. The maps $\mu_k$ satisfy $\mu_k \circ \varphi_k = \mu_{k+1}$ and are the universal maps for the direct limit.

Dualising the notion of inverse sequences of locally compact Hausdorff spaces and continuous maps, we consider the following, more general, notion of directed sequences of $C^*$-algebras.

\begin{dfn} A \emph{generalised directed sequence $C^*$-algebras} is a sequence $(A_i,\varphi_i)_{i \in \NN}$ of $C^*$-algebras $A_i$ and morphisms $\varphi_{i} \in \Mor(A_i,A_{i+1})$.
\end{dfn}

\begin{example}
	If $(X_i,f_i)_{i \in \NN}$ is a inverse sequence of locally compact Hausdorff spaces and continuous maps, then $f_i^* \in \Mor(C_0(X_i),C_0(X_{i+1}))$, and so $(C_0(X_i), f_i^*)_{i \in \NN}$ is a generalised directed sequence of $C^*$-algebras. 
\end{example}

\begin{example}
	A generalised directed sequence of $C^*$-algebras $(A_i,\varphi_i)_{i \in \NN}$, in which each $\varphi_i$ is proper, induces a directed sequence of $C^*$-algebras in the usual sense.  
\end{example}

Nontrivial noncommutative examples of generalised directed sequences appear in the theory of $C^*$-correspondences. 
\begin{example}\label{ex:correspondence_generalised_sequence}
	As in \cref{ex:toeplitz}, let $(\varphi,X_A)$ be a nondegenerate $A$--$A$-correspondence, and for each $i \ge 0$ let $X_A^{\ox i} \coloneqq X \ox_A X \ox_A\cdots \ox_A X$  be $i$-th tensor power of $X_A$, balanced over $A$. Let $\varphi_0 \coloneqq \varphi$, and for $k \ge 1$ let $\varphi_i \colon \Kk_A(X_A^{\ox i}) \to \Ll_A(X_A^{\ox i+1}) \cong \Mm(\Kk_A(X_A^{\ox i+1}))$ be the unique $*$-homomorphism satisfying
	\[
	\varphi_i(T) (x \ox y) = (Tx) \ox y 
	\]
	for all $x \in X_A^{\ox i}$ and $y \in X_A$ (see \cite[p.~369]{Kat04cor}). Then $\varphi_i \in \Mor(\Kk_A(X_A^{\ox i}),\Kk_A(X_A^{\ox i+1}))$, so $(\Kk_A(X_A^{\ox i}),\varphi_i)_{i \in \NN}$ is a generalised directed sequence of $C^*$-algebras.
\end{example}

Unlike the case of locally compact Hausdorff spaces, where $\varprojlim(X_i,f_i)$ makes sense as a topological space even if the $f_i$ are not proper, there is no good corresponding notion (as far as the author is aware) of a direct limit of a generalised sequence of $C^*$-algebras. Instead, we are interested in taking more general types of limits, like we did in the topological setting. 

Given a generalised directed sequence $(A_i,\varphi_i)_{i \in \NN}$ the naive approach is to first extend each $\varphi_i \in \Mor(A_i,A_{i+1})$ to a unital $*$-homomorphism $\ol{\varphi_i} \colon \Mm(A_i) \to \Mm(A_{i+1})$, and then take the usual direct limit of the resulting directed sequence of $C^*$-algebras
\[
A_0 \overset{\varphi_0}{\longrightarrow} \Mm(A_1) \overset{\ol{\varphi_1}}{\longrightarrow} \Mm(A_2) \overset{\ol{\varphi_2}}{\longrightarrow} \Mm(A_3) \longrightarrow \cdots .
\]
We denote this limit by $\mlima(A_i,\varphi_i)$. Since $A_i \subseteq \Mm(A_i)$, each of the algebras $A_i$ map into $\mlima(A_i,\varphi_i)$, and the limit is compatible with the $\varphi_i$. The drawback of this approach is that the limit is typically very large: not only does each $A_i$ map into $\mlima(A_i,\varphi_i)$, but so does $\Mm(A_i)$ for all $i \ge 1$. For example, if each $A_i$ is separable, then $\mlima(A_i,\varphi_i)$ is typically nonseperable. 

Instead we seek a smaller notion of limit. More specifically, we insist that the maps from each $A_i$ into the limit in are proper morphisms. 

\begin{lem}\label{lem:limit_properly_nd}
	Let $(A_i,\varphi_i)_{i \in \NN}$ be a
	 directed sequence of $C^*$-algebras, and let $\mu_k \colon A_k \to \varinjlim(A_i,\varphi_i)$ denote the universal $*$-homomorphisms. If each $\varphi_k \in \Morp(A_k,A_{k+1})$, then each $\mu_k \in \Morp(A_k,\varinjlim(A_i,\varphi_i))$. 
\end{lem}

\begin{proof}
	Fix $k \in \NN$ and let $(a_\lambda)$ be an approximate unit for $A_k$ with $\|a_{\lambda}\| \le 1$. Fix $b \in \varinjlim(A_i,\varphi_i)$ and $\varepsilon > 0$. Since  $\varinjlim(A_i,\varphi_i) = \overline{\bigcup_{i \in \NN}\mu_i(A_i)}$, there exists $n \ge k$ and $b_n \in A_n$ such that $\|\mu_n(b_n) - b\| < \varepsilon/3$. Since each $\varphi_i$ is proper, there exists $\lambda_0$ such that for $\lambda \ge \lambda_0$ we have $\|\varphi_{n-1} \circ \cdots \circ \varphi_{k}(a_\lambda)b_n - b_n\| < \varepsilon/3$. Hence, for $\lambda \ge \lambda_0$ we have	
	\begin{align*}
		\| \mu_k(a_\lambda)b - b \|
		&\le \|\mu_k(a_\lambda)b - \mu_k(a_{\lambda})b_n\| + 
		 \|\mu_n(\varphi_{n-1} \circ \cdots \circ \varphi_{k}(a_\lambda)b_n) - \mu_n(b_n)\|\\
		& \qquad \qquad  + \|\mu_n(b_n) - b  \| < \varepsilon.
	\end{align*}
	So, $\mu_k$ is a proper morphism. 
\end{proof}

 Using fibrewise compactifications we can inductively extend a generalised directed sequence of $C^*$-algebras to one with proper morphisms, and \cref{lem:limit_properly_nd} implies that the universal maps into the limit will be proper. For the unified algebra construction we get the following notion of a limit.

\begin{dfn}
Let $(A_i,\varphi_i)_{i \in \NN}$ be a generalised directed sequence of $C^*$-algebras. Let $\widetilde{A}_0 \coloneqq A_0$. Working inductively, for each $i \ge 1$ define
\[
\wt{A}_{i} \coloneqq \unis{A_i}{\ol{\varphi}_{i-1} \circ \alpha_{A_{i-1}}}{\wt{A}_{i-1}},
\]
let $\alpha_{i} \in \Mor(\wt{A}_i,A_i)$ be induced by $A_i \trianglelefteq \wt{A}_i$, let $\ol{\varphi}_i \colon \Mm(\wt{A}_i) \to \Mm(A_{i+1})$ be the unique unital map induced by $\varphi_i$, and let
$\wt{\varphi}_{i} \colon \wt{A}_i \to \wt{A}_{i+1}$ be given by $\wt{\varphi}_i(a) \coloneqq  \unif{(\ol{\varphi}_{i} \circ \alpha_{A_{i}})}$.
 We call the $C^*$-algebra $\unilima(A_i,\varphi_i) \coloneqq \varinjlim(\wt{A}_i,\wt{\varphi}_i)$ the \emph{unified limit} of $(A_i,\varphi_i)_{i \in \NN}$.
\end{dfn}

By \cref{prop:unified_alg_perf}, the $\wt{\varphi}_i$ are injective and belong to $\Morp(\wt{A}_i,\wt{A}_{i+1})$. Explicitly, they satisfy $
\wt{\varphi}_i(a) = (\ol{\varphi}_{i} \circ \alpha_{A_{i}}(a),a)
$ for all $a \in \wt{A}_i$. Extending the notation established in \cref{lem:compositions_associative}, we  write 
\[
\wt{A}_{i} \eqqcolon \unis{A_{i}}{\varphi_{i-1}}{
	\unis{A_{i-1}}{\varphi_{i-2}}{\unis{\cdots}{\varphi_0}{A_0}}}.
\]
By inductively applying \cref{lem:alg_unified_composition}, we may identify $\wt{A}_i$ with
\begin{align*}
	&\big\{(a_i,\ldots,a_1,a_0) \in \Mm(A_i) \oplus \cdots \oplus \Mm(A_1) \oplus A_0 \bigm| \\
	 &\qquad \varphi_0(a_0) - a_1 \in A_1 \text{ and } \ol{\varphi}_k(a_k) - a_{k+1} \in A_{k+1} \text{ for all } 1 \le k < i \big\}.
\end{align*}
Under this identification, the map $\wt{\varphi}_i$ satisfies
\begin{equation}\label{eq:alpha_commute}
\wt{\varphi}_i(a_i,\ldots, a_0) = (\ol{\varphi}_{i} \circ \alpha_{A_{i}}(a_i,\ldots, a_0),(a_i,\ldots, a_0)) = (\ol{\varphi}_i(a_i),a_i,\ldots,a_0).
\end{equation}
We record some properties of the unified limit. 

\begin{prop}\label{prop:unified_limit_algebra}
	Let $(A_i,\varphi_i)_{i \in \NN}$ be a generalised directed sequence of $C^*$-algebras.
	
	\begin{enumerate}
	\item \label{itm:univ_maps}	The universal maps $\mu_k \colon \wt{A}_k \to \unilima(A_i,\varphi_i)$ belong to $\Morp(\wt{A}_k,\unilima(A_i,\varphi_i))$ and are injective.

	\item \label{itm:mult_lim_map}	 The maps $\alpha_{A_i} \in \Mor(\wt{A}_i,A_i)$, induced by the ideal $A_i \trianglelefteq \wt{A}_i$, induce a $*$-homomorphism $\unilima(A_i,\varphi_i) \to \mlima(A_i,\varphi_i)$.

	\item \label{itm:property_limit} If $\Pp$ is a property of $C^*$-algebras that is closed under extensions and countable direct limits (e.g. separability, nuclearity), and $A_i$ satisfies $\Pp$ for all $i \in \NN$, then $\unilima(A_i,\varphi_i)$ also satisfies $\Pp$. 
	\end{enumerate}
\end{prop}

\begin{proof}
	\labelcref{itm:univ_maps} That the $\mu_k$ are injective follows from the $\wt{\varphi}_i$ being injective for all $i \in \NN$. \cref{lem:limit_properly_nd} implies that each $\mu_k \in \Morp(\wt{A}_k,\unilima(A_i,\varphi_i))$.
	
	\labelcref{itm:mult_lim_map} 
	Using \labelcref{eq:alpha_commute} we have  $\alpha_{A_{i+1}}  \circ \wt{\varphi}_i = \ol{\varphi}_i \circ \alpha_{A_i} $ for all $i \ge 0$. The existence of a $*$-homomorphism  $\unilima(A_i,\varphi_i) \to \mlima(A_i,\varphi_i)$ now follows from the universal property of direct limits (cf. \cref{lem:injectivity_limits}).
	
	\labelcref{itm:property_limit} This follows inductively from the fact that, as a unified algebra, each $\wt{A}_{i+1}$ is an extension of $\wt{A}_i$ by $A_{i+1}$.
\end{proof}

By considering quotient fibrewise compactification of the unified algebra, we can also construct smaller limits of a generalised directed sequence.  

\begin{dfn}\label{dfn:regulated_limit_alg} 	A \emph{regulating sequence} for a generalised directed sequence $(A_i,\varphi_i)_{i \in \NN}$ of $C^*$-algebras is a collection $\Jj = (J_i)_{i \in \NN}$ of ideals $J_i \trianglelefteq A_i$ such that each $J_i$ is $\varphi_i$-proper. If each $J_i$ is $\varphi_i$-perfect, then we call $\Jj$ a \emph{perfect} regulating sequence.

	Let $A^{\Jj}_0 \coloneqq A_0$. 
	For $i \ge 1$, inductively define
	\[
	A^\Jj_i \coloneqq \pers{A_{i}}{\decor{\varphi_{i-1}}}{q_{i-1}(J_{i-1} \oplus 0)}{A^{\Jj}_{i-1}},
	\]
	where $\decor{\varphi_{i-1}} \in \Mor(A^{\Jj}_{i-1},A_{i})$ is the map defined in \cref{lem:compositions_associative_quotient}, and 	
	$q_i \colon \unis{A_{i}}{\decor{\varphi_{i-1}}}{A_{i-1}^{\Jj}} \to A^{\Jj}_i$ is the quotient map. For $i \ge 0$, let $\varphi^{\Jj}_i \colon A_{i}^{\Jj} \to A_{i+1}^\Jj$ be given by $\varphi^{\Jj}_i \coloneqq \decor{\varphi}_{q_{i}(J_{i} \oplus 0)}$.	
	We call the $C^*$-algebra $\varinjlim^{\Jj}(A_i,\varphi_i) \coloneqq \varinjlim(A^{\Jj}_{i},\varphi^{\Jj}_i)$ the \emph{$\Jj$-regulated limit} of $(A_i,\varphi_i)_{i \in \NN}$.
\end{dfn}

	By \cref{prop:quotient_perfections_unified}, the $\varphi^{\Jj}_i$ belong to $\Morp(\wt{A}_i,\wt{A}_{i+1})$, and are injective if $\Jj$ is perfect. 
 Extending the notation established in \cref{lem:compositions_associative_quotient}, we  write
 \[
A^{\Jj}_i \eqqcolon \pers{A_{i}}{\varphi_{i-1}}{J_{i-1}}{
 	\pers{A_{i-1}}{\varphi_{i-2}}{J_{i-2}}{\pers{\cdots}{\varphi_0}{J_{0}}{A_0}}}.
 \]
By \cref{lem:compositions_associative_quotient}, we may identify $A_i^\Jj$ with 
\begin{align*}
	&\wt{A}_i/ (0 \oplus J_{i-1} \oplus \cdots \oplus J_{0})\\
	&=\Big\{ (a_i,a_{i-1}+J_{i-1},\ldots,a_1 + J_1,a_0 + J_0 ) \in \Mm(A_i) \oplus \frac{\Mm(A_{i-1})}{J_{i-1}} \oplus \cdots \oplus \frac{\Mm(A_{1})}{J_1} \oplus \frac{A_0}{J_0}  \Bigm| \\
	&\qquad  \qquad \varphi_0(a_0) - a_1 \in A_1 \text{ and } \ol{\varphi}_k(a_k) - a_{k+1} \in A_{k+1} \text{ for all } 1 \le k < i  \Big\}.
\end{align*}
Under this identification, the maps $\varphi^\Jj_i$ satisfy
\begin{equation}\label{eq:phi_J_formula}
	\varphi^\Jj_i(a_i,a_{i-1}+J_{i-1},\ldots,a_0 + J_0 ) = (\ol{\varphi}_i(a_i), a_i + J_i,\ldots ,a_0 + J_0 ).
\end{equation}

We record the following properties of the $\Jj$-regulated limit.

\begin{cor}\label{cor:regulated_limit_alg}
	Let $(A_i,\varphi_i)_{i \in \NN}$ be a generalised directed sequence of $C^*$-algebras with a regulating sequence $\Jj = (J_i)_{i \in \NN}$. 
	\begin{enumerate}
		
		\item \label{itm:reg_a}
		The universal maps $\mu_k \colon A^\Jj_k \to \varinjlim^{\Jj}(A_i,\varphi_i)$ belong to $ \Morp(A^\Jj_k,\varinjlim^{\Jj}(A_i,\varphi_i))$, and are injective if $\Jj$ is perfect. 
		
		\item \label{itm:reg_b} The maps $\alpha_{A_i} \in \Mor(A^{\Jj}_i,\Mm(A_i))$, induced by the ideal $A_i \trianglelefteq A^{\Jj}_i $, induce a $*$-homomorphism $\varinjlim^{\Jj}(A_i,\varphi_i) \to \mlima(A_i,\varphi_i)$.

		\item \label{itm:reg_c} If $\Pp$ is a property of $C^*$-algebras that is closed under taking quotients, ideals, extensions, and countable direct limits (e.g. separability, nuclearity); and $A_i$ satisfies $\Pp$ for all $i \in \NN$, then $\varprojlim^{\Jj}(A_i,\varphi_i)$ satisfies $\Pp$. 
		
		\item \label{itm:reg_d}	Let $\Ii = (I_i)_{i \in \NN}$ be another regulating sequence for $(A_i,\varphi_i)_{i \in \NN}$. If $J_i \trianglelefteq I_i$ for all $i \in \NN$, then there is a surjection $\varinjlim^{\Jj}(A_i,\varphi_i) \twoheadrightarrow \varinjlim^{\Ii}(A_i,\varphi_i)$.
	\end{enumerate}
\end{cor}

\begin{proof}
	Parts \labelcref{itm:reg_a}, \labelcref{itm:reg_b}, and \labelcref{itm:reg_c} follow from arguments analogous to those found in the proof of \cref{prop:unified_limit_algebra}.
	
	For \labelcref{itm:reg_d}, observe that there are quotient maps 
	\[
	\pi_i \colon A_{i}^{\Jj} \cong \frac{\wt{A}_i}{0 \oplus J_{i-1} \oplus \cdots \oplus J_{0}} \to \frac{\wt{A}_i}{0 \oplus I_{i-1} \oplus \cdots \oplus I_{0}} \cong A_i^{\Ii}.  
	\]
	By \labelcref{eq:phi_J_formula}, we have $\pi_{i+1} \circ \varphi^{\Jj}_i = \pi_i \circ \varphi^{\Jj}_{i+1}$. If $\mu_k^\Jj \colon A_k^{\Jj} \to \varinjlim^{\Jj}(A_i,\varphi_i)$ and $\mu_k^\Ii \colon A_k^{\Jj} \to \varinjlim^{\Jj}(A_i,\varphi_i)$ are the universal maps, then the universal property of direct limits gives a $*$-homomorphism $\pi \colon \varinjlim^{\Jj}(A_i,\varphi_i) \to \varinjlim^{\Ii}(A_i,\varphi_i)$ such that $\pi \circ \mu_k^{\Jj} = \mu_k^{\Ii} \circ \pi_k$ for all $k \ge 0$. Since the union of the images of the $\mu_k^{\Ii}$ is dense in $\varinjlim^{\Ii}(A_i,\varphi_i)$, and the $\pi_k$ are surjective, $\pi$ is surjective. 
\end{proof}

We have the following consequence of \cref{cor:commutative_quotients_same}.

\begin{cor}
	Let $(X_i,f_i)_{i \in \NN}$ be an inverse sequence of locally compact Hausdorff spaces and continuous maps with a regulating sequence $\Uu = (U_i)_{i \in \NN}$. Then $\Uu^* \coloneqq (C_0(U_i))_{i \in \NN}$ is a regulating sequence for the generalised directed sequence $(C_0(X_i),f_i^*)_{i \in \NN}$ of $C^*$-algebras. Moreover, $C_0(\varprojlim^{\Uu}(X_i,f_i)) \cong \varinjlim^{\Uu^*}(C_0(X_i),f_i^*)$.
\end{cor}

We highlight two special cases of regulated limits. 

\begin{dfn}
	Let $(A_i,\varphi_i)_{i \in \NN}$ be a generalised directed sequence of $C^*$-algebras. 
	\begin{enumerate}
		\item If $\Jj = (\pim(\varphi_i))_{i \in \NN}$, then we call $\minlima(A_i,\varphi_i) \coloneqq \varinjlim^{\Jj}(A_i,\varphi_i)$ the \emph{minimal regulated limit} of $(A_i,\varphi_i)_{i \in \NN}$.
		\item If $\Jj = (\kat(\varphi_i))_{i \in \NN}$, then we call $\perlima(A_i,\varphi_i) \coloneqq \varinjlim^{\Jj}(A_i,\varphi_i)$ the \emph{minimal perfect regulated limit} of $(A_i,\varphi_i)_{i \in \NN}$.
	\end{enumerate}
\end{dfn}

By \cref{cor:regulated_limit_alg}, the above regulated limits fit into a commuting diagram
\[
\begin{tikzcd}[ampersand replacement=\&,cramped]
	\& {\mlima(A_i,\varphi_i)} \\
	{ \unilima (A_i,\varphi_i)} \& {\perlima(A_i,\varphi_i)} \& {\minlima(A_i,\varphi_i)}
	\arrow[from=2-1, to=1-2]
	\arrow[two heads, from=2-1, to=2-2]
	\arrow[from=2-2, to=1-2]
	\arrow[two heads, from=2-2, to=2-3]
	\arrow[from=2-3, to=1-2]
\end{tikzcd}
\]
where the horizontal maps are surjections, and the maps into $\mlima(A_i,\varphi_i)$ are induced by the inclusions of ideals.

\subsection{Cores of Cuntz--Pimsner algebras as regulated limits}\label{sec:cores}

Let $(\varphi,X_A)$ be an $A$--$A$-correspondence as in \cref{ex:correspondences}.
In this section, we explicitly compute regulated limits of the generalised directed sequence $(\Kk_A(X_A^{\ox i}),\varphi_i)_{i \in \NN}$ of
\cref{ex:correspondence_generalised_sequence} in terms of relative Cuntz--Pimsner algebras. 

Fix a $\varphi$-proper ideal $J \trianglelefteq A$. Let
$\tau^J = (\tau_A^J,\tau_X^J)$ denote a universal $J$-covariant representation of $(\varphi,X_A)$ in the $J$-relative Cuntz--Pimsner algebra $\Oo_{X,J}$. By the universal property of $\Oo_{X,J}$, there is a strongly continuous action $\gamma \colon \TT \to \Aut(\Oo_{X,J})$, called the \emph{gauge action}, satisfying
\[
\gamma_z(\tau_A^J(a)) = \tau_A^J(a) \quad \text{and}
\quad \gamma_z(\tau_X^J(x)) = z \tau_X^J(x)
\]
for all $a \in A$ and $x \in X_A$. The fixed-point algebra $\Oo_{X,J}^{\gamma}$ of $\gamma$ is called the \emph{core} of $\Oo_{X,J}$. In \cite[Proposition~5.7]{Kat04cor} it is shown that $\Oo_{X,J}^\gamma$ can be expressed as a direct limit of $C^*$-algebras. In this section, we show that $\Oo_{X,J}^{\gamma}$ is in fact a regulated limit of $(\Kk_A(X_A^{\ox i}),\varphi_i)_{i \in \NN}$.

We first characterise $\pim(\varphi_i)$ and $\kat(\varphi_i)$. To do this we establish some additional terminology and results about $C^*$-correspondences. 
Let $X_A$ be a right Hilbert $A$-module. By \cite[Corollary~1.4]{Kat07cor}, if $I \trianglelefteq A$, then $XI \coloneqq X \cdot I = \{x \cdot a \mid x \in X, a \in I\}$ is an $A$--$A$-correspondence with the operations induced from $X_A$. 
By \cite[Lemma~2.3]{FMR03}, there is an inclusion $\iota_I \colon \Kk_A(XI) \to \Kk_A(X)$ as an ideal such that $\iota_I(\Theta_{x,y}^{XI}) = \Theta_{x,y}^X$ for all $x,y \in XI$. Moreover, for $T \in \Kk_A(XI)$ the operator $\iota_I(T)$ is the unique extension of $T$ to $\Ll_A(X)$ whose range is contained in $XI$. As such we identify $\Kk_A(XI)$ as an ideal in $\Kk_A(X)$. 

The following characterisation of $\Kk_A(XI)$ is known to experts. 
\begin{lem}\label{lem:compact_ideal_characterisation}
Let $X$ be a right Hilbert $A$-module and let $I \trianglelefteq A$.
Then
\begin{align}\label{eq:K(XI)_characterisation}
	\begin{split}
		\Kk_A(XI) &= \ol{\spaan}\{ \Theta_{\xi,\eta \cdot a} \mid \xi,\eta \in X, a \in I\}\\
		&= \{T \in \Kk_A(X) \mid \langle x | T y \rangle_A \in I \text{ for all } x,y \in X\}.
	\end{split}
\end{align}
\end{lem}

\begin{proof}
For the first equality of \labelcref{eq:K(XI)_characterisation}, observe that $\Theta_{\xi \cdot a, \eta \cdot b} = \Theta_{\xi,\eta  \cdot a^* b}$ for all $x,y \in X$ and $a,b \in I$. So $\Kk_A(XI) \subseteq 	\ol{\spaan}\{ \Theta_{\xi,\eta \cdot a} \mid \xi,\eta\in X, a \in I\}$. For the reverse inclusion fix $\xi,\eta \in X$ and $a \in A$. Since positive elements span $A$, we can suppose that $a$ is positive. Then $\Theta_{\xi,\eta \cdot a} = \Theta_{\xi \cdot a^{1/2}, \eta \cdot a^{1/2}} \in \Kk_A(XI)$.
	
For the second equality of \labelcref{eq:K(XI)_characterisation}, fix $\xi,\eta \in X$ and $a \in I$ and consider $\Theta_{\xi, \eta \cdot a} \in \Kk_A(XI)$. For all $x,y \in X$, we have 
$
\langle x | \Theta_{\xi \cdot, \eta \cdot a} y \rangle_A = \langle x | \xi \cdot \langle \eta \cdot a | y \rangle_A \rangle_A = \langle x | \xi \rangle_A a^* \langle \eta | y \rangle_A \in I.
$
Since the $\Theta_{\xi, \eta \cdot a}$ densely span $\Kk_A(XI)$, if $T \in \Kk_A(XI)$, then $\langle x | T y \rangle_A \in I$ for all $x,y \in X$.

On the other hand, suppose that $T \in \Kk_A(X)$ and  $\langle x | T y \rangle_A \in I$ for all $x,y \in X$. 
By \cite[Proposition~1.3]{Kat07cor}, we have $Ty \in X I$ for all $y \in X$.
Choose an increasing approximate unit $(\sum_{(x,y) \in F_{\lambda}} \Theta_{x,y})_{\lambda \in \Lambda}$ for $\Kk_A(X)$, where for each $\lambda \in \Lambda$, $F_{\lambda}$ is a finite subset of $X \times X$. In particular, we have 
$T = \lim_{\lambda} \sum_{(x,y) \in F_{\lambda}} \Theta_{Tx,y}$ in the norm on $\Kk_A(X)$. Since each $Tx \in XI$, by the first equality of \labelcref{eq:K(XI)_characterisation}, we have $\Theta_{Tx,y} \in \Kk_A(XI)$. The ideal $\Kk_A(XI)$ is closed in $\Kk_A(X)$, so $T \in \Kk_A(XI)$.
\end{proof}

Let $X_A$ be a right Hilbert $A$-module. We write $\langle X | X \rangle \coloneqq \ol{\spaan}\{\langle x | y \rangle_A \colon x,y \in X\}$. We say $X_A$ is \emph{full} if $\langle X \mid X \rangle = A$. If $X_A$ is not full, then $\langle X | X \rangle$ is a proper ideal of $A$. The module $X_A$ provides a Morita equivalence between $\Kk_A(X)$ and $\langle X | X \rangle$ (cf. \cite[Proposition~3.8]{RW98}). As such, every ideal of $\Kk_A(X)$ is isomorphic to $\Kk_A(XI)$ for some unique ideal $I \trianglelefteq \langle X | X \rangle$. Moreover, the bijection $I \leftrightarrow \Kk_A(XI)$ is an isomorphism between the ideal lattices of $\langle X | X \rangle$ and $\Kk_A(X)$ \cite[Proposition~3.24]{RW98}.

For a non-full Hilbert module $X_A$ we can have $\Kk_A(XI) = \Kk_A(XJ)$ for distinct ideals $I,J \trianglelefteq A$. 

\begin{lem}\label{lem:compact_annihilators}
	Let $X_A$ be a right Hilbert $A$-module. Then for every ideal $I \trianglelefteq A$ we have $\Kk_A(X I) = \Kk_A(X \cdot (I \cap \langle X | X \rangle))$ and $\Kk_A(X  I)^{\perp} = \Kk_A( X \cdot (I \cap \langle X | X \rangle)^{\perp})$. 
\end{lem}

\begin{proof}
	Since $I \cap \langle X | X \rangle \trianglelefteq I$, we have $\Kk_A(X \cdot (I \cap \langle X | X \rangle)) \subseteq \Kk_A(X I)$. For the reverse inclusion, fix $\xi,\eta \in X$ and $a \in I$. By \cref{lem:compact_ideal_characterisation}, $\Theta_{\xi,\eta \cdot a} \in \Kk_A(XI)$. For each $y \in X$, we have $\Theta_{\xi,\eta \cdot a}y = \xi \cdot a^* \langle \eta | y \rangle_A$. Since $a^* \langle \eta | y \rangle_A \in I \cap \langle X \mid X \rangle$, we have $\langle x | \Theta_{\xi,\eta \cdot a} y \rangle_A \in I \cap \langle X | X \rangle$ for all $x,y \in X$, so by \cref{lem:compact_ideal_characterisation}, we have $\Theta_{\xi,\eta \cdot a} \in \Kk_A(X \cdot (I \cap \langle X | X \rangle))$. Since the $\Theta_{\xi,\eta \cdot a}$ densely span $\Kk_A(X \cdot (I \cap \langle X | X \rangle))$, we have $\Kk_A(X I) = \Kk_A(X \cdot (I \cap \langle X | X \rangle))$.
	
	Now suppose that $\xi,\eta \in X$ and $a \in (I \cap \langle X | X \rangle)^{\perp}$. Fix $S \in \Kk_A(XI)$. Using \cref{lem:compact_ideal_characterisation}, we have
	$
	\Theta_{\xi,\eta \cdot a}Sx = \xi \cdot a^* \langle \eta | Sx \rangle_A  = 0
	$
	for all $x \in X_A$. Using \cref{lem:compact_ideal_characterisation} again, for all $x,y \in X_A$ we have
	$
	\langle x | S \Theta_{\xi,\eta \cdot a} y \rangle_A= \langle x | S \xi \rangle_A a^* \langle \eta | y \rangle_A = 0,
	$
	so $S \Theta_{\xi,\eta \cdot a} = 0$.  Hence, $\Theta_{\xi,\eta \cdot a} \in \Kk_A(XI)^{\perp}$. 	
	Since such $\Theta_{\xi,\eta \cdot a}$ densely span $\Kk_A( X \cdot (I \cap \langle X | X \rangle)^{\perp})$ we have $\Kk_A( X \cdot (I \cap \langle X | X \rangle)^{\perp}) \subseteq \Kk_A(X  I)^{\perp}$.
	
	Now suppose that $T \in \Kk(X I)^{\perp}$. Fix $a \in I \cap \langle X | X \rangle$. Then there exists $\xi,\eta \in XI$ such that $\langle \xi | \eta \rangle_A =a$. For each $x,y \in X_A$, we have
	\[
	\langle x | Ty \rangle_A a = \langle x | Ty \cdot \langle \xi | \eta \rangle_A \rangle_A = \langle x | T \Theta_{y,\xi} \eta \rangle_A = 0
	\]
	since $\Theta_{y,\xi} \in \Kk_A(XI)$, and
	\[
	a \langle x | Ty \rangle_A = \langle x \cdot \langle \eta | \xi \rangle_A \mid T y \rangle_A = \langle \Theta_{x,\eta} \xi | Ty \rangle_A = \langle \xi | \Theta^*_{x,\eta} T y\rangle_A = 0
	\]
	since $\Theta_{x,\eta} \in \Kk_A(XI)$. Hence, $\langle x | T y \rangle_A \in (I \cap \langle X | X \rangle)^{\perp}$ for all $x,y \in X_A$, so $T \in \Kk_A( X \cdot (I \cap \langle X | X \rangle)^{\perp}).$
\end{proof}

\cref{lem:compact_annihilators} indicates that if a Hilbert module $X_A$ is not full, then the map taking an ideal $I \trianglelefteq A$ to $\Kk_A(XI)$ does not behave nicely with respect to annihilators. This makes a description of the Katsura ideals of the maps $\varphi_i$ from \cref{ex:correspondence_generalised_sequence} difficult. As such, we restrict our attention to full modules. 

\begin{prop} \label{prop:tensor_ideals}
	Let $X_A$ be a full right Hilbert $A$-module and let $(\psi,Y_B)$ be a nondegenerate $A$--$B$-correspondence. Consider the morphism $\lambda \in \Mor(\Kk_A(X), \Kk_B(X \ox Y))$ satisfying $\lambda(T)(x \ox y) = (Tx) \ox y$ for all $T \in \Kk_A(X)$ and $x \ox y \in X \ox_A Y$ (see \cite[p.~190]{Pim97}). Then
	\begin{enumerate}
		\item \label{itm:pim_same}	$\pim(\lambda) = \Kk_A(X \cdot \pim(\psi))$;
		\item \label{itm:ker_same} $\ker(\lambda) = \Kk_A(X \cdot \ker(\psi))$;
		\item \label{itm:perp_same} $\ker(\lambda)^{\perp} = \Kk_A(X \cdot \ker(\psi)^{\perp})$; and
		\item \label{itm:kat_same} $\kat(\lambda) = \Kk_A(X \cdot \kat(\psi)).$
	\end{enumerate}
	Moreover, the Morita equivalence between $A$ and $\Kk_A(X)$, induced by $X_A$, restricts to an isomorphism between the sublattice of $\psi$-proper ($\psi$-perfect) ideals and the sublattice of $\lambda$-proper ($\lambda$-perfect) ideals.
\end{prop}

\begin{proof}
	\labelcref{itm:pim_same} This follows directly from \cite[Corollary~3.7]{Pim97}.
	
	\labelcref{itm:ker_same} Observe that $T \in \ker(\lambda)$ if and only if 
	\[
	0 = \langle x_1 \ox y_1 | Tx_2 \ox y_2\rangle_B = \langle y_1 | \psi(\langle x_1 | Tx_2 \rangle_A)y_2 \rangle_B
	\]
	for all $x_1,x_2 \in X_A$ and $y_1,y_2 \in Y_B$. This is true if and only if $\langle x_1 | T x_2 \rangle_A \in \ker(\psi)$ for all $x_1,x_2 \in X_A$. By \cref{lem:compact_ideal_characterisation}, this is equivalent to $T \in \Kk_A(X \cdot \ker(\psi))$.
	
	\labelcref{itm:perp_same} Since $X_A$ is full, this follows from \labelcref{itm:ker_same} and \cref{lem:compact_annihilators}.
	
	\labelcref{itm:kat_same} Since the Morita equivalence induced by $X_A$ is an isomorphism of the lattice of ideals, we have
	$
	\kat(\lambda) = \Kk_A(X \cdot \pim(\psi)) \cap \Kk_A(X \cdot \ker(\psi)^{\perp}) = \Kk_A(X \cdot (\pim(\psi) \cap \ker(\psi)^{\perp}))  = \Kk_A(X \cdot \kat(\psi)).
	$
	
	The final statement now follows from \labelcref{itm:pim_same}, \labelcref{itm:kat_same}, and \cref{lem:phi_proper_characterisation}.
\end{proof}

 Suppose that $\rho = (\rho_A,\rho_X)$ a representation of an $A$--$A$-correspondence  $(\varphi,X_A)$ in a $C^*$-algebra $B$. For each $i \ge 1$, $\rho$ induces a representation $\rho^i$ of $(\varphi,X_A^{\ox i})$ in $B$ satisfying $\rho^i_A = \rho_A$ and $\rho^i_X(x_1 \ox \cdots \ox x_i) = \rho_X(x_1) \cdots \rho_X(x_i)$ for all $x_1 \ox \cdots \ox x_i \in X^{\ox i}$. We write $\rho_A^{(i)} \coloneqq (\rho^i)_A^{(1)}$ for the induced $*$-homomorphism $\Kk_A(X^{\ox i}) \to B$ of \labelcref{eq:compact_formula}. 

Now suppose that $(\varphi,X_A)$ is full and
 consider the generalised sequence $(\Kk_A(X^{\ox i}), \varphi_i)_{i \in \NN}$ of \cref{ex:correspondence_generalised_sequence}. Since $X^{\ox {i+1}} \cong X^{\ox i} \ox_A X$, it follows inductively from \cref{prop:tensor_ideals} that for any $\varphi$-proper ($\varphi$-perfect) ideal $J \trianglelefteq A$ the ideal $\Kk_A(X^{\ox i} J)$ is $\varphi_i$-proper ($\varphi_i$-perfect). Moreover, $\pim(\varphi_i) = \Kk_A(X^{\ox i} \cdot \pim(\varphi))$ and $\kat(\varphi_i) = \Kk_A(X^{\ox i} \cdot \kat(\varphi))$. As such, for every $\varphi$-proper ($\varphi$-perfect) ideal $J$ there is an associated (perfect) regulating sequence $\Jj = (\Kk_A(X^{\ox i}  J))_{i \in \NN}$ for $(\Kk(X^{\ox i}), \varphi_i)_{i \in \NN}$. 

\begin{thm}\label{prop:cuntz-pimsner_characterisation2}
Let $(\varphi,X_A)$ be a $C^*$-correspondence over $A$ and let $J \trianglelefteq A$ be a $\varphi$-perfect ideal. As in \cref{ex:correspondence_generalised_sequence}, let $\varphi_0 \coloneqq \varphi$, and for each $i \ge 1$, let $\varphi_i \in \Mor(\Kk_A(X^{\ox i}), \Kk_A(X^{\ox {i+1}}))$ be the map satisfying $\varphi_i(T)(x \ox y) = (Tx) \ox y$ for all $T \in \Kk_A(X^{\ox i})$, $x \in X^{\ox i}$, and $y \in X$.  
Let $\Jj = (\Kk_A(X^{\ox i}J))_{i \in \NN}$ be the perfect regulating sequence associated to $J$, and let $\gamma$ be the gauge action on $\Oo_{X,J}$. 
Then $\varinjlim^{\Jj}(\Kk_A(X^{\ox i}),\varphi_i) \cong \Oo_{X,J}^{\gamma}$. In particular,
\[
\unilima(\Kk_A(X^{\ox i}),\varphi_i) \cong \Tt_X^{\gamma} \quad \text{ and } \quad \perlima(\Kk_A(X^{\ox i}),\varphi_i) \cong \Oo_X^{\gamma}.
\]
\end{thm}

\begin{proof}
	Since $J$ is $\varphi$-perfect, the universal representation $\tau^J = (\tau^J_A,\tau^J_X)$ of $(\varphi,X_A)$ in $\Oo_{X,J}$ is faithful. Hence, the maps $(\tau_A^J)^{(i)} \colon \Kk_A(X^{\ox i}) \to \Oo_{X,J}$ are injective for all $i \in \NN$. Following \cite{Kat04cor}, for each $i \in \NN$ we define
	\[
	B_{[0,i]} = \tau_A^J(A) + (\tau_A^J)^{(1)}(\Kk_A(X))
	+ \cdots +
	(\tau_A^J)^{(i)}(\Kk_A(X^{\ox i})).
	\]
	
	Let $\iota_{i} \colon B_{[0,i]} \to B_{[0,i+1]}$ denote the obvious inclusion. 
	By \cite[Proposition~5.7]{Kat04cor}, we have \[\Oo_{X,J}^\gamma = \ol{\bigcup_{i \in \NN} B_{[0,i]}} = \varinjlim(B_{[0,i]}, \iota_i).\]
	Let $J_i \coloneqq \Kk_A(X^{\ox i}J)$, \[A_i^{\Jj} \coloneqq \pers{\Kk_A(X^{\ox i})}{\varphi_{i-1}}{J_{i-1}}{\pers{\Kk_A(X^{\ox i-1})}{\varphi_{i-2}}{J_{i-2}}{\pers{\cdots}{\varphi_0}{J}{A}}},\] and let $\varphi_i^{\Jj} \in \Morp (A_i^{\Jj},A_{i+1}^{\Jj})$ be as in \labelcref{eq:phi_J_formula}. Since $(\tau_A^J)^{(i)}$ is injective for all $i \in \NN$, \cite[Proposition~5.12]{Kat04cor} yields a commuting diagram	\begin{equation*}
		\begin{tikzcd}[ampersand replacement=\&]
			0 \arrow[r] 
			\& \Kk_A(X^{\ox i} J) \arrow[r,"(\tau_A^J)^{(i)}"] \arrow[d,"\varphi_i"] 
			\& B_{[0,i]} \arrow[r] \arrow[d, "\iota_i"]
			\&  B_{[0,i]} /(\tau_A^J)^{(i)}(\Kk_A(X^{\ox i} J)) \arrow[r] \arrow[d, "\cong"] 
			\& 0 
			\\
			0 \arrow[r] 
			\& \Kk_A(X^{\ox i+1}) \arrow[r, "(\tau_A^J)^{(i+1)}"] 
			\& B_{[0,i+1]} \arrow[r]  
			\& B_{[0,i+1]}/(\tau_A^J)^{(i+1)}(\Kk_A(X^{\ox i+1})) \arrow[r]                              
			\& 0
		\end{tikzcd}
	\end{equation*}
	with exact rows. Starting from $i = 0 $, and inductively applying universal property of \cref{thm:unified_quotient}, there are isomorphisms $\Phi_i \colon A_{i}^{\Jj} \to B_{[0,j]}$ such that $\iota_i \circ \Phi_i = \Phi_i \circ \varphi_i^{\Jj}$ for all $i \in \NN$. Consequently,
	\[
	\textstyle \varinjlim^{\Jj}(\Kk(X^{\ox i}), \varphi_i) = \varinjlim(A^{\Jj}_i,\varphi^{\Jj}_i) \cong \varinjlim(B_{[0,i]},\iota_i) = \Oo_{X,J}^{\gamma}. 
	\]
	By \cref{prop:tensor_ideals}, and the fact that $\Tt_X = \Oo_{X,\{0\}}$ and $\Oo_X = \Oo_{X,\kat(\varphi)}$, we have
	$\unilima(\Kk_A(X^{\ox i}),\varphi_i) \cong \Tt_X^{\gamma}$ and $\perlima(\Kk_A(X^{\ox i}),\varphi_i) \cong \Oo_X^{\gamma}$.
\end{proof}

\vspace{10pt}

\scriptsize{
\noindent \textsc{Alexander Mundey}\\
School of Mathematics and Applied Statistics,\\ University of Wollongong, NSW 2522, Australia. \\
Email address: \texttt{amundey@uow.edu.au}\\
}


\end{document}